\documentclass[12pt]{article}

\usepackage[a4paper,bindingoffset=0.2in,%
            left=1in,right=1in,top=1in,bottom=1in,%
            footskip=.25in,
            ]{geometry}  

\usepackage[T1]{fontenc}

\usepackage{amsthm,amssymb}
\usepackage{amsmath}
\usepackage{tikz}
\usepackage{subcaption}
\usepackage{mathtools}
\usepackage{booktabs}
\usepackage{microtype}
\usepackage{bm}
\usepackage{placeins} 
\usepackage{mathrsfs} 
\usepackage{blkarray,bigstrut} 
\usepackage{etoc} 
\usepackage{relsize}  
\usepackage{exscale}  
\usepackage[ugly]{nicefrac} 
\usepackage{xfrac} 
\usepackage{old-arrows} 
\usepackage{vruler}
\usepackage{xr-hyper}
\usepackage[colorlinks,citecolor=blue]{hyperref}
\usepackage{doi} 

\usetikzlibrary{shapes,positioning,fit,matrix,arrows,arrows.meta}
\usetikzlibrary{decorations.pathmorphing,decorations.pathreplacing,decorations.markings}
\usetikzlibrary{calc}

\DeclarePairedDelimiter{\ceil}{\lceil}{\rceil}
\DeclarePairedDelimiter{\floor}{\lfloor}{\rfloor}

\newcommand{\mystrut}{\rule[-0.55\baselineskip]{0pt}{\baselineskip}}

\newcommand\topstrut[1][1.2ex]{\setlength\bigstrutjot{#1}{\bigstrut[t]}}
\newcommand\botstrut[1][0.9ex]{\setlength\bigstrutjot{#1}{\bigstrut[b]}}

\tikzset{
dot/.style={draw,fill,circle,inner sep = 0pt,minimum size = 3pt},
bigdot/.style={dot,minimum size = 4pt},
terminal/.style={draw,circle, inner sep=2.5pt},
vcolour/.style={draw,inner sep=1.5pt,font=\scriptsize,label distance=2pt},
22box/.style={draw,minimum width=2cm,minimum height=1.5cm,font=\LARGE,node contents=\( \Downarrow \)},
22boxsmall/.style={22box,minimum width=1.5cm, minimum height=1cm},
Vset/.style={draw=black!30,ellipse,minimum width=1.5cm,minimum height=2.5cm},
petal/.style={
    decoration = {markings, mark=at position 1.0 with {
\coordinate (centre) at (0,0);
\path (centre)+(120:2) node(v)[bigdot]{};
\path (centre)+(240:2) node(u)[bigdot]{};
\path (centre)+(0:2) node(w)[bigdot]{};
\draw (v)--(u)--(w);
\draw (v) arc (120:-120:2);
\path (v)--(u) node(v1)[pos=0.4][bigdot]{} node(u-v)[pos=0.6][bigdot]{};
\path (u)--(w) node(u-w)[pos=0.4][bigdot]{} node(w-u)[pos=0.6][bigdot]{};
\draw (w)--(v) node(w-v)[pos=0.4][bigdot]{} node(x1)[pos=0.6][bigdot]{};
\draw (u-w)--(u-v)--node[pos=0.3][bigdot]{} (w-v)--(w-u)--node[pos=0.3][bigdot]{} (v1)--(x1)--node[pos=0.3][bigdot]{} (u-w);
}},
    postaction = {decorate}
    },
midArrow/.style={
    decoration = {markings, mark=at position 0.5 with {
\path[->,ultra thick] (0,0)--(0,0);
}},
    postaction = {decorate}
    },
}

\newcommand{\overbar}[1]{\mkern 1.9mu\overline{\mkern-1.5mu#1\mkern-1.5mu}\mkern 3.5mu}

\newcommand{\pagetarget}[2]{%
  \phantomsection%
  \label{#1}%
  \hypertarget{#1}{#2}%
}

\newtheorem{theorem}{Theorem}
\newtheorem{lemma}{Lemma}
\newtheorem{observation}{Observation}
\newtheorem{corollary}{Corollary}

\newtheoremstyle{freethm}
{3pt}
{3pt}
{}
{}
{\bfseries}
{.\\}
{.5em}
{\thmname{#1}\thmnumber{ #2}\thmnote{ (#3)}}

\theoremstyle{freethm}

\newtheorem{problem}{Problem}
\newtheorem{conjecture}{Conjecture}

\newtheorem{construct}{Construction}%
\theoremstyle{plain}

\newtheorem{innercustomthm}{Theorem}

\newtheorem{innercustomcor}{Corollary}
\newenvironment{customcor}[1]
  {\renewcommand\theinnercustomcor{#1 (Restated)}\innercustomcor}
  {\endinnercustomthm}

\title{Star Colouring of Bounded Degree Graphs\\
and Regular Graphs\thanks{The final authenticated version is available online at \href{https://doi.org/10.1016/j.disc.2022.112850}{https://doi.org/10.1016/j.disc.2022.112850}}}

\author{Shalu M.A.\thanks{Supported by SERB(DST), MATRICS scheme MTR/2018/000086.} \qquad Cyriac Antony\\
\small Indian Institute of Information Technology, Design \& Manufacturing\\[-0.8ex]
\small (IIITDM) Kancheepuram, Chennai, India\\[-0.8ex]
\small\tt \{shalu,mat17d001\}@iiitdm.ac.in}
\date{}

\begin{document}

\maketitle


\begin{abstract}
A $k$-star colouring of a graph $G$ is a function $f:V(G)\to\{0,1,\dots,k-1\}$ such that $f(u)\neq f(v)$ for every edge $uv$ of $G$, and every bicoloured connected subgraph of $G$ is a star. 
The star chromatic number of $G$, $\chi_s(G)$, is the least integer $k$ such that $G$ is $k$-star colourable. 
We prove that $\chi_s(G)\geq \lceil (d+4)/2\rceil$ for every $d$-regular graph $G$ with $d\geq 3$. 
We reveal the structure and properties of even-degree regular graphs $G$ that attain this lower bound. 
The structure of such graphs $G$ is linked with a certain type of Eulerian orientations of $G$. 
Moreover, this structure can be expressed in the LC-VSP framework of Telle and Proskurowski (SIDMA, 1997), and hence can be tested by an FPT algorithm with the parameter either treewidth, cliquewidth, or rankwidth. 
We prove that for $p\geq 2$, a $2p$-regular graph $G$ is $(p+2)$-star colourable only if $n\coloneqq |V(G)|$ is divisible by $(p+1)(p+2)$. 
For each $p\geq 2$ and $n$ divisible by $(p+1)(p+2)$, we construct a $2p$-regular Hamiltonian graph on $n$ vertices which is $(p+2)$-star colourable. 

The problem \textsc{$k$-Star Colourability} takes a graph $G$ as input and asks whether $G$ is $k$-star colourable. 
We prove that \textsc{3-Star Colourability} is NP-complete for planar bipartite graphs of maximum degree three and arbitrarily large girth. 
Besides, it is coNP-hard to test whether a bipartite graph of maximum degree eight has a unique 3-star colouring up to colour swaps. 
For $k\geq 3$, \textsc{$k$-Star Colourability} of bipartite graphs of maximum degree $k$ is NP-complete, and does not even admit a $2^{o(n)}$-time algorithm unless ETH fails. 
\end{abstract}

\section{Introduction}\label{sec:intro}

The star colouring is a well-known variant of (vertex) colouring introduced by Gr\"{u}nbaum~\cite{grunbaum} in the 1970s. 
The scientific computing community independently discovered star colouring in the 1980s and used it for lossless compression of symmetric sparse matrices, which is in turn used in the estimation of sparse Hessian matrices (see the survey \cite{gebremedhin2005}). 
A \( k \)-colouring \( f \) of a graph \( G \) is a \emph{\( k \)-star colouring of \( G \)} if every bicoloured component of \( G \) under \( f \) is a star (in other words, \( G \) does not contain a bicoloured 4-vertex path as a subgraph). 
The star chromatic number of \( G \), \( \chi_s(G) \), is the least integer \( k \) such that \( G \) is \( k \)-star colourable. 
Star colouring of a graph \( G \) is known to be linked with orientations of \( G \). 
Albertson et al.~\cite{albertson} proved that a colouring \( f \) of \( G \) is a star colouring if and only if there exists an orientation \( \overrightarrow{G} \) of \( G \) such that edges in each bicoloured 3-vertex path in \( \overrightarrow{G} \) are oriented towards the middle vertex. 
Ne\v{s}et\v{r}il and Mendez~\cite{nesetril_mendez2003} characterized the star chromatic number of \( G \) in terms of orientations of~\( G \).

Fertin et al.~\cite{fertin2004} proved that the star chromatic number of a \( d \)-regular hypercube is at least \( \ceil{(d+3)/2} \). 
Xie et al.~\cite{xie} proved that the star chromatic number of a \( 3 \)-regular graph is at least four. 
We generalize this result: the star chromatic number of a \( d \)-regular graph is at least \( \ceil{(d+4)/2} \), provided \( d\geq 2 \). 
We show that this lower bound is attained for each \( d\geq 2 \) and characterize even-degree regular graphs that attain the lower bound (i.e. \( 2p \)-regular \( (p+2) \)-star colourable graphs). 
We introduce a variant of Eulerian orientation named \emph{\( q \)-colourful Eulerian orientation} (see Section~\ref{sec:orientations}). 
For all \( p\geq 2 \), a \( 2p \)-regular graph \( G \) is \( (p+2) \)-star colourable if and only if \( G \) admits a \( (p+2) \)-colourful Eulerian orientation. 
Using this result, we show that a \( 2p \)-regular \( (p+2) \)-star colourable graph \( G \) does not contain diamond or \( \overbar{C_6} \) as a subgraph, and thus the clique number of \( G \) is at most three. 
We also establish the following properties of \( 2p \)-regular \( (p+2) \)-star colourable graphs: (i)~the number of vertices in \( G \) is divisible by \( (p+1)(p+2) \), (ii)~the independence number of \( G \) is greater than \( n/4 \), and (iii)~the chromatic number of \( G \) is at most \( 3\log_2(p+2) \). 
For every \( p\geq 2 \) and every integer \( n \) divisible by \( (p+1)(p+2) \), we construct a \( 2p \)-regular Hamiltonian graph on \( n \) vertices which is \( (p+2) \)-star colourable. 
For the special case \( p=2 \), the graphs constructed are also planar. 
If a regular graph \( G \) with degree \( d\geq 3 \) is a hypercube (i.e., \( G=Q_d \)) or contains diamond or \( \overbar{C_6} \) as a subgraph, then \( \chi_s(G)\geq \ceil{(d+5)/2} \); this improves on the lower bound \( \chi_s(Q_d)\geq \ceil{(d+3)/2} \) given in \cite{fertin2004}.

For all \( p\geq 2 \), we express the problem of testing whether a \( 2p \)-regular graph \( G \) is \( (p+2) \)-star colourable in the Locally Checkable Vertex Subset and Partitioning problems (LC-VSP) framework of Telle and Proskurowski~\cite{telle_proskurowski}; thus, the problem admits an FPT algorithm with the parameter either treewidth, cliquewidth, rankwidth or booleanwidth (see \cite{bui-xuan2010,bui-xuan2013,oum,telle,telle_proskurowski}). 
We also define a subclass \( \mathscr{G}_{2p} \) of the class of \( 2p \)-regular \( (p+2) \)-star colourable graphs such that testing membership in \( \mathscr{G}_{2p} \) is NP-complete (provided \( p\geq 2 \)).

The decision problem \textsc{\( k \)-Star Colourability} takes a graph \( G \) as input and asks whether \( G \) is \( k \)-star colourable. 
\textsc{3-Star Colourability} is NP-complete for planar bipartite graphs \cite{albertson}, graphs of arbitrarily large girth \cite{bok}, and graphs of maximum degree four (in fact, NP-complete for line graphs of subcubic graphs \cite{lei}). 
We prove that the problem is NP-complete for a subclass of the intersection of these three classes: 
\textsc{3-Star Colourability} is NP-complete for planar bipartite graphs of maximum degree three and arbitrarily large girth. 
Besides, it is coNP-hard to test whether a bipartite graph of maximum degree eight has a unique 3-star colouring up to colour swaps. 
For all \( k\geq 3 \), \textsc{\( k \)-Star Colourability} is NP-complete for bipartite graphs. 
We prove that for all \( k\geq 3 \), \textsc{\( k \)-Star Colourability} restricted to bipartite graphs of maximum degree \( k \) is NP-complete and the problem does not even admit a \( 2^{o(n)} \)-time algorithm unless the Exponential Time Hypothesis (ETH) fails. 

\subsection*{Fixed-Parameter Tractability}\label{sec:fpt}
For every positive integer \( k \), \textsc{\( k \)-Star Colourability} can be expressed in Monadic Second Order (MSO) logic \cite{harshita}. 
In fact, \textsc{\( k \)-Star Colourability} can be expressed in MSO\( _1 \), i.e., MSO logic without edge set quantification (see supplementary material for the formula). 
Therefore, for each \( k \), the problem \textsc{\( k \)-Star Colourability} admits FPT algorithms with the parameter either treewidth or cliquewidth by Courcelle's theorem \cite{borie,courcelle}. 
On the other hand, the reduction from \textsc{\( k \)-Colourability} to \textsc{\( k \)-Star Colourability} produced by Coleman and Mor\'{e}~\cite{coleman_more} is a Polynomial Parameter Transformation (PPT) \cite{fomin2019} when both problems are parameterized by treewidth (see Section~2 of supplementary material for details). 
Thus, we have the following observation since \textsc{\( k \)-Colourability} with parameter treewidth does not admit a polynomial kernel. 
\begin{observation}\label{obs:star intractable w.r.t. treewidth}
For all \( k\geq 3 \), \textsc{\( k \)-Star Colorability} with parameter treewidth does not admit a polynomial kernel unless NP \( \subseteq \) coNP/poly. 
\qed
\end{observation}

The paper is organized as follows. 
Section~\ref{sec:def} contains the definitions used throughout the paper. 
Section~\ref{sec:lb chi_s} discusses a lower bound for the star chromatic number of \( d \)\nobreakdash-regular graphs as well as properties of graphs that attain the lower bound. 
Section~\ref{sec:orientations} introduces colourful Eulerian orientation and discusses its relation to star colouring. 
All computational hardness results appear in Section~\ref{sec:hardness}. 
It is divided into two subsections: Subsection~\ref{sec:hardness 3-star} on 3-star colouring, and Subsection~\ref{sec:hardness k-star k>=4} on \( k \)-star colouring with \( k\geq 4 \). 
We conclude with Section~\ref{sec:open} devoted to open problems and related works.

\section{Definitions}\label{sec:def}
All graphs considered in this paper are finite and simple. 
We follow West~\cite{west} for graph theory terminology and notation. 
Unless otherwise specified, each graph we consider is undirected (the directed graphs that appear in this paper are orientations of undirected graphs). 
An \emph{orientation} of \( G \) is a directed graph obtained from \( G \) by assigning a direction to each edge of \( G \). 
When the graph is clear from the context, we denote the number of edges of the graph by \( m \) and the number of vertices by \( n \). 
For a subset \( S \) of the vertex set of \( G \), the \emph{subgraph of \( G \) induced by \( S \)} is denoted by \( G[S] \). 
The \emph{girth} of a graph with a cycle is the length of its shortest cycle. 
A graph \( G \) is \emph{\( 2 \)-degenerate} if there exists a left-to-right ordering of its vertices such that every vertex has at most two neighbours to its left; in other words, we can turn \( G \) into the empty graph by repeatedly removing vertices of degree at most two.

A \( k \)-colouring of a graph \( G \) is a function \( f \) from the vertex set of \( G \) to a set of \( k \) colours, say \( \{0,1,\dots,k-1\} \), such that \( f \) maps every pair of adjacent vertices to different colours. 
Let us denote the \( i \)th colour class \( f^{-1}(i) \) by \( V_i \). 
A \emph{bicoloured component} of \( G \) (under \( f \)) is a component of \( G[V_i\cup V_j] \) for some pair of colour classes \( V_i \) and \( V_j \). 
A \( k \)-colouring \( f \) of \( G \) is a \emph{\( k \)-star colouring} if every bicoloured component of \( G \) under \( f \) is a star (in other words, there is no 4-vertex path in \( G \) bicoloured by \( f \)). 
The acyclic colouring is a generalization of star colouring. 
A \( k \)-colouring \( f \) of \( G \) is a \emph{\( k \)-acyclic colouring} if every bicoloured component of \( G \) under \( f \) is a tree. 
By definition, every \( k \)-star colouring is a \( k \)-acyclic colouring. 
The star chromatic number \( \chi_s(G) \) is defined analogously to the chromatic number \( \chi(G) \). 
That is, the star chromatic number of \( G \) is the least integer \( k \) such that \( G \) is \( k \)-star colourable.

We say that two colourings \( f_1 \) and \( f_2 \) of \( G \) are the same \emph{up to colour swaps} if there exists a permutation \( \sigma \) of colours such that \( f_2(v)=\sigma(f_1(v)) \) for every vertex \( v \) of \( G \). 
If \( G \) has exactly one \( k \)-(star) colouring up to colour swaps, then \( G \) is said to have a \textit{unique} \( k \)-(star) colouring (we write the word `unique' in italics to indicate ``unique up to colour swaps''). 
We say that two colourings \( f_1 \) and \( f_2 \) of the same graph are \emph{equivalent under colour swaps} if they are the same up to colour swaps.

For every positive integer \( k \), the decision problem \textsc{\( k \)-Colourability} takes a graph \( G \) as input and asks whether \( G \) is \( k \)-colourable. 
The problem \textsc{\( k \)-Star Colourability} is defined likewise. 
To denote the restriction of a decision problem, we write the conditions in parenthesis. 
For instance, \textsc{3-Star Colourability}\( ( \)planar, bipartite, \( \Delta=3 \)\( ) \) denotes the problem \textsc{3-Star Colourability} restricted to the class of planar bipartite graphs \( G \) with the maximum degree \( \Delta(G)=3 \). 
Let \( k\geq 3 \). 
The decision problem \textsc{Unique \( k \)-Colouring} takes a graph \( G \) as input and asks whether \( G \) has a \textit{unique} \( k \)-colouring. 
The problem \textsc{Unique \( k \)-Star Colouring} is defined similarly. 
The problem \textsc{Another \( k \)-Colouring} takes a graph \( G \) and a \( k \)-colouring \( f_1 \) of \( G \) as input and asks whether \( G \) admits a \( k \)-colouring \( f_2 \) of \( G \) which cannot be obtained from \( f_1 \) by merely swapping colours. 
The problem \textsc{Another \( k \)-Star Colouring} is defined likewise.

\section{Lower Bound for Star Chromatic Number\\ of Regular Graphs}\label{sec:lb chi_s}
We know that for every star colouring of \( G \), each bicoloured component is a star. 
By exploiting the simple structure of bicoloured components and employing elementary counting arguments, we produce a tight lower bound for the star chromatic number of regular graphs and characterize even-degree graphs that attain the lower bound. 

\begin{theorem}\label{thm:lb chi_s}
Let \( G \) be a \( d \)-regular graph with \( d\geq 2 \). Then, \( \chi_s(G)\geq \raisebox{1pt}{\big\lceil}\frac{d+4}{2}\raisebox{1pt}{\big\rceil} \).
\end{theorem}
\begin{proof}
We have two cases: (i)~\( d \) is even, and (ii)~\( d \) is odd. 
If \( d \) is even, say \( d=2k \), then at least \( \lceil (d+3)/2 \rceil = (d+4)/2 = k+2 \) colours are needed to acyclic colour \( G \) \cite[Proposition~1]{fertin2003}; hence, at least \( k+2=\lceil (d+4)/2 \rceil \) colours are needed to star colour \( G \). 
Next, we consider the case when \( d \) is odd, say \( d=2k+1 \). 
To prove that \( \chi_s(G)\geq \lceil (d+4)/2 \rceil=k+3 \), it suffices to show that \( G \) is not \( (k+2) \)-star colourable. 
Assume that \( G \) admits a \( (k+2) \)-star colouring \( f\colon V(G)\to\{0,1,\dots,k+1\} \). 
Recall that \( V_i=f^{-1}(i)=\{v\in V(G)\ :\ f(v)=i \} \) for every colour \( i \).\\[5pt]
\noindent \textbf{Claim~1:} For every bicoloured component \( H \) of \( G \), \( |E(H)|\leq \frac{d}{d+1}|V(H)| \), and equality holds only when \( H \) is isomorphic to \( K_{1,d} \).\\[5pt]
Since \( H\cong K_{1,q} \) where \( 0\leq q\leq d \), we have \( |E(H)|/|V(H)|=q/(q+1)\leq d/(d+1) \) and equaltiy holds only when \( q=d \). This proves Claim~1.\\[5pt]
%
\noindent \textbf{Claim~2:} \( G \) has a bicoloured component \( H \) not isomorphic to \( K_{1,d} \).\\[5pt]
On the contrary, assume that every bicoloured component of \( G \) is isomorphic to \( K_{1,d} \). 
Consider an arbitrary bicoloured component \( H \) of \( G \). 
Let \( u \) be the centre of the star \( H\cong K_{1,d} \), and let \( v_1,v_2,\dots,v_d \) be the remaining vertices in \( H \). 
Without loss of generality, assume that \( u \) is coloured~0, and each \( v_i \) is coloured~1 for \( 1\leq i\leq d \). 
Let \( N_G(v_1)=\{u,w_1,w_2,\dots,w_{d-1}\} \). 
Clearly, \( f(w_i)\in \{2,3,\dots,k+1\} \) for \( 1\leq i\leq d-1 \) (if \( f(w_i)=0 \), then path \( v_2,u,v_1,w_i \) is a bicoloured \( P_4 \)). 
Hence, at least two of vertices \( w_1,w_2,\dots,w_{d-1} \) should receive the same colour, say colour~2. 
Since \( |\{w_1,\dots,w_{d-1}\}|=d-1 \), vertex \( v_1 \) has \( q \) neighbrours coloured~2 where \( 2\leq q\leq d-1 \).
Thus, the component of \( G[V_1\cup V_2] \) containing vertex \( v_1 \) is isomorphic to \( K_{1,q}\not\cong K_{1,d} \). This contradiction proves Claim~2. 

For every pair of distinct colours \( i \) and \( j \), let \( \mathbb{G}_{ij} \) denote the set of components of~\( G[V_i\cup V_j] \). By Claims~1 and 2, we have
\[
\sum_{\substack{i,j\\ i\neq j}}\sum_{H\in\mathbb{G}_{ij}}|E(H)|<\sum_{\substack{i,j\\ i\neq j}}\sum_{H\in\mathbb{G}_{ij}}\frac{d}{d+1}|V(H)|=\frac{d}{d+1}\sum_{\substack{i,j\\ i\neq j}}\left(\,|V_i|+|V_j|\,\right).
\]
Since the set of bicoloured components of \( G \) forms an (edge) decomposition of \( G \), the sum on the left side is \( m\coloneqq |E(G)| \). 
The sum on the right side is \( (k+1)n \) where \( n\coloneqq |V(G)| \) (because \( |V_i| \) appears exactly \( (k+1) \) times in the sum for each \( i \)). 
Therefore, the above inequality simplifies to \( m<\frac{d}{d+1}(k+1)n \). 
Since \( G \) is \( d \)-regular, \( m=\frac{d}{2}n \) and thus the inequality reduces to \( \frac{d}{2}n<\frac{d}{d+1}(k+1)n \). 
That is, \( \frac{d+1}{2}<k+1 \). 
This is a contradiction because \( d=2k+1 \). 
This completes the proof when \( d=2k+1 \). 
\end{proof}
\noindent \textbf{Remark:} From Theorem~\ref{thm:lb chi_s}, it follows that the average degree of a graph \( G \) of maximum degree \( d \) is at most \( (\chi_s(G)-1)\frac{2d}{d+1} \), and the graph is bipartite if the average degree is equal to this bound (see Corollary~1 in supplementary material).

Soon, we show that the lower bound \( \lceil (d+4)/2\rceil \) established by Theorem~\ref{thm:lb chi_s} is tight for each \( d\geq 2 \). 
For every even-degree regular graph with degree \( d=2p \), the bound is \( \lceil (d+4)/2\rceil=p+2 \). 
The next theorem shows various properties of \( (p+2) \)-star colourings of \( 2p \)-regular graphs (provided \( p\geq 2 \)). 
\begin{theorem}\label{thm:2p-regular if p+2-star colourable}
Let \( p\geq 2 \). 
Suppose that a \( 2p \)-regular graph \( G(V,E) \) admits a \( (p+2) \)-star colouring \( f \). 
Then, every bicoloured component of \( G \) is isomorphic to \( K_{1,p} \) and all colour classes \( V_i=f^{-1}(i) \) have the same cardinality. 
Besides, \( f \) is not only an equitable colouring but also a fall colouring \( ( \)see \cite{meyer} and \cite{dunbar} respectively for definitions\( ) \). 
\end{theorem}
\begin{proof}
Note that the set \( S \) of all bicoloured components of \( G \) forms an (edge) decomposition of \( G \). 
We produce a partition \( \mathscr{P} \) of \( S \) such that \( \sum_{H\in P} |E(H)|\leq \frac{p}{p+1} \sum_{H\in P} |V(H)| \) for each \( P\in\mathscr{P} \) and equality holds only if all members of \( P \) are \( K_{1,p} \). 

Since \( f \) is a star colouring, every bicoloured component of \( G \) is a star. 
We first deal with bicoloured components \( H \) with the unique centre -- that is, \( H\cong K_{1,\ell} \) where \( \ell\neq 1 \). 
Let \( V' \) denote the set of vertices \( v \) in \( G \) such that \( v \) is the centre of a bicoloured component \( H\cong K_{1,\ell} \) with \( \ell>p \). 
For each \( v\in V' \), define \( C_v=\{ H\in S : H\not\cong K_{1,1}\text{ and } v \text{ is the centre of } H\} \).\\

\noindent \textbf{Claim~1:} For each \( v\in V' \), \( \sum\limits_{H\in C_v} |E(H)| < \frac{p}{p+1} \sum\limits_{H\in C_v} |V(H)| \).\\[5pt]
Let \( v\in V' \). 
Let \( J \) be the set of indices \( j \) such that \( v \) has exactly one neighbour in colour class \( V_j \). 
Let \( x=|J| \). 
By definition of \( V' \), \( v \) has \( p+1 \) or more neighbours in some colour class; therefore, \( x<p \) (since \( x+p+1\leq \deg_G(v)=2p \)). 
By definition of \( C_v \), for every neighbour \( w \) of \( v \) in \( \cup_{j\in J}V_j \), the edge \( vw \) is not counted in the sum \( \sum_{H\in C_v} |E(H)| \). 
On the other hand, the remaining \( 2p-x \) edges incident on \( v \) are counted exactly once in the sum \( \sum_{H\in C_v} |E(H)| \). 
So, \( \sum_{H\in C_v} |E(H)|=2p-x \). 
No neighbour \( w \) of \( v \) in \( \cup_{j\in J}V_j \) is counted in the sum \( \sum_{H\in C_v} |V(H)| \). 
The remaining \( 2p-x \) neighbours of \( v \) are counted exactly once in the sum \( \sum_{H\in C_v} |V(H)| \). 
Also, the vertex \( v \) is counted exactly \( p+1-x \) times in the sum \( \sum_{H\in C_v} |V(H)| \) (assuming \( v\in V_i \), \( v \) is in exactly one component of \( G[V_i\cup V_k] \) for each \( k\in\{0,1,\dots,p+1\}\setminus (J\cup \{i\}) \)). 
Therefore, \( \sum_{H\in C_v}|V(H)|=(2p-x)+(p+1-x)=3p-2x+1 \). 
Since \( x<p \) and \( p>1 \), we have 
\[ 
\frac{\sum_{H\in C_v} |E(H)|}{\strut \sum_{H\in C_v}|V(H)|}=\frac{2p-x}{3p-2x+1}<\frac{p}{p+1}
\]
because the inequality \( (p+1)(2p-x)<p(3p-2x+1) \) simplifies to \( (p-1)(p-x)>0 \). 
This proves Claim~1 since \( v\in V' \) is arbitrary.\\

For distinct vertices \( u \) and \( v \) in \( V' \), the set \( C_u\cap C_v \) is empty because each member of \( C_u \) has vertex \( u \) as its unique centre whereas each member of \( C_v \) has vertex \( v \) as its unique centre. 
We are now ready to construct the partition \( \mathscr{P} \) of \( S \). 
For each \( v\in V' \), include \( C_v \) in \( \mathscr{P} \). 
For each \( H\in S\setminus \bigcup_{v\in V'}C_v \), include \( \{H\} \) in \( \mathscr{P} \). 
Observe that for all \( H\in S\setminus \bigcup_{v\in V'}C_v \), the bicoloured component \( H \) is isomorphic to \( K_{1,\ell} \) with \( \ell\leq p \); as a result, \( |E(H)|\leq \frac{p}{p+1}|V(H)| \) and equality holds only if \( H\cong K_{1,p} \) (as in Claim~1 of Theorem~\ref{thm:lb chi_s}). 
By Claim~1, \( \sum_{H\in P} |E(H)| < \frac{p}{p+1} \sum_{H\in P} |V(H)| \) for each member \( P=C_v \) of \( \mathscr{P} \) (where \( v\in V' \)). 
Thus, we have the following claim.\\[5pt]
\noindent \textbf{Claim~2:} For each \( P\in\mathscr{P} \), \( \sum_{H\in P} |E(H)|\leq \frac{p}{p+1} \sum_{H\in P} |V(H)| \) and equality holds only if every member of \( P \) is isomorphic to \( K_{1,p} \).\\

\noindent \textbf{Claim~3:} Every bicoloured component of \( G \) is isomorphic to \( K_{1,p} \).\\[5pt]
Contrary to Claim~3, assume that there is a bicoloured component \( H' \) of \( G \) not isomorphic to \( K_{1,p} \). 
Then, \( H' \) is in some member \( P' \) of \( \mathscr{P} \). 
By Claim~2, \( \sum_{H\in P'} |E(H)|<\frac{p}{p+1} \sum_{H\in P'} |V(H)| \).  
Since \( \sum_{H\in P} |E(H)|\leq \frac{p}{p+1} \sum_{H\in P} |V(H)| \) for each \( P\in\mathscr{P} \) and the inequality is strict for \( P'\in\mathscr{P} \), we have
\[ \sum_{P\in\mathscr{P}} \sum_{H\in P} |E(H)| < \frac{p}{p+1}\sum_{P\in\mathscr{P}} \sum_{H\in P} |V(H)|. \]
Since \( \cup_{P\in\mathscr{P}} \cup_{H\in P} H \) is a decomposition of \( G \), the sum on the left side is \( m\coloneqq |E(G)| \). 
Besides, the sum on the right side is \( (p+1)n \) where \( n\coloneqq |V(G)| \) because each vertex of \( G \) is counted exactly \( p+1 \) times (note that each vertex of \( G \) is in exactly \( p+1 \) bicoloured components of \( G \)). 
Thus, the inequality simplifies to \( m<\frac{p}{p+1}(p+1)n \). 
That is, \( m<pn \), a contradiction because \( 2m=(2p)n \) for every \( 2p \)-regular graph. 
This proves Claim~3.\\

\noindent \textbf{Claim~4:} For every colour class \( V_i \) and every vertex \( v\notin V_i \), \( v \) has either exactly \( p \) neighbours or exactly one neighbour in \( V_i \). In particular, \( v \) has at least one neighbour in \( V_i \).\\[5pt]
Let \( v\notin V_i \); that is, \( v\in V_j \) for some \( j\neq i \). 
By Claim~3, the component of \( G[V_i\cup V_j] \) containing \( v \) is a star \( H\cong K_{1,p} \). 
If \( v \) is the centre of \( H \), then \( v \) has exactly \( p \) neighbours in \( V_i \). 
Otherwise, \( v \) has exactly one neighbour in \( V_i \). 
This proves Claim~4.\\

\noindent \textbf{Claim~5:} Each vertex \( v \) of \( G \), say with colour~\( i \), has exactly \( p \) neighbours in some colour class \( V_j \) and exactly one neighbour in every other colour class \( V_k \), \( k\notin \{i,j\} \) (see Figure~\ref{fig:neighbourhood of vertex in G}).
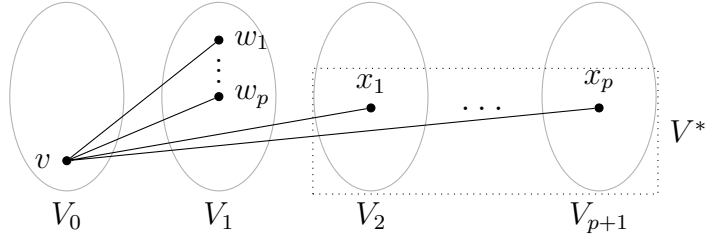
\begin{figure}[hbt]
\centering
\begin{tikzpicture}
\path (0,0)  node(V0)[Vset][label=below:\( V_0 \)]{}++(2,0) node(V1)[Vset][label=below:\( V_1 \)]{}++(2,0) node(V2)[Vset][label=below:\( V_2 \)]{}++(1.5,0) node[font=\large][yshift=-0.15cm]{\dots}++(1.5,0) node(Vp+1)[Vset][label=below:\( V_{p+1} \)]{};
\path (V0)+(0,-0.85) node(v0)[dot][label=left:\( v \)]{};
\path (V1)++(0,0.75) node(v11)[dot][label=right:\( w_1 \)]{}++(0,-0.3) node{\vdots}++(0,-0.45) node(v1p)[dot][label=right:\( w_p \)]{};
\path (V2)+(0,-0.15) node(v2)[dot][label=\( x_1 \)]{};
\path (Vp+1)+(0,-0.15) node(vp+1)[dot][label=\( x_p \)]{};
\draw
(v0)--(v11)
(v0)--(v1p)
(v0)--(v2)
(v0)--(vp+1);
\node [fit={(v2) (vp+1) (V2.-90) (Vp+1.-90) (V2.155) (Vp+1.20)},inner sep=1pt,draw,dotted][label=right:\( V^* \)]{};
\end{tikzpicture}
\caption{An arbitrary vertex \( v\in V_i \) and its neighbourhood in \( G \) (here, \( i=0 \) and \( j=1 \)).}
\label{fig:neighbourhood of vertex in G}
\end{figure}

\noindent We prove Claim~5 for \( i=0 \) (proof is similar for other values of \( i \)). 
Let \( v\in V_0 \). 
Neighbours of \( v \) are from the colour classes \( V_1,\dots,V_{p+1} \). 
Since \( \deg_G(v)=2p>p+1 \), by pigeon-hole principle, \( v \) has at least two neighbours in some colour class \( V_j \). 
Without loss of generality, assume that \( j=1 \). 
Since \( v \) has more than one neighbour in \( V_1 \), \( v \)~has exactly \( p \) neighbours in \( V_1 \) by Claim~4. 
Let \( w_1,\dots,w_p,x_1,\dots,x_p \) be the neighbours of \( v \) where \( w_1,\dots,w_p\in V_1 \). 
The remaining \( p \) neighbours \( x_1,\dots,x_p \) of \( v \) are in \( \bigcup_{k=2}^{p+1} V_k \). 
Since \( v \) has at least one neighbour in each of the colour classes \( V_2,\dots,V_{p+1} \) (by Claim~4), \( v \)~has exactly one neighbour in each of the colour classes \( V_2,\dots,V_{p+1} \).
This proves~Claim~5.\\

\noindent \textbf{Claim~6:} All colour classes have the same cardinality: that is, \( |V_i|=\frac{|V|}{p+2} \) for every colour~\( i \).\\[5pt]
\noindent We prove Claim~6 for \( i=0 \) (proof is similar for other values of \( i \)). 
By Claim~4, every vertex \( w\in V\setminus V_0 \) has either exactly \( p \) neighbours or exactly one neighbour in \( V_0 \). 
Let \( V^* \) denote the set of vertices \( x\in V\setminus V_0 \) such that \( x \) has exactly \( p \) neighbours in \( V_0 \). 
In the sum \( \sum_{v\in V_0} |N_G(v)| \), each member of \( V^* \) is counted exactly \( p \) times (because each \( w\in V^* \) has exactly \( p \) neighbours in \( V_0 \)) and every member of \( V\setminus (V_0\cup V^*) \) is counted exactly once (because each \( w\in V\setminus (V_0\cup V^*) \) has exactly one neighbour in \( V_0 \)). 
Hence, 
\( \sum_{v\in V_0} |N_G(v)|=|V\setminus (V_0\cup V^*)|+p|V^*| \), and thus
\begin{equation}\label{eqn:2p-regular if p+2-star colourable eqn1}
\sum_{v\in V_0} |N_G(v)|=|V|-|V_0|-|V^*|+p|V^*|.
\end{equation}

By counting the number of edges between the sets \( V_0 \) and \( V^* \), we show that \( |V_0|=|V^*| \). 
Consider an arbitrary vertex \( v\in V_0 \). 
Let \( w_1, \dots, w_p, x_1,\dots,x_p \) be the neighbours of \( v \). 
By Claim~5, \( v \) has exactly \( p \) neighbours in some colour class \( V_j \) and exactly one neighbour in every other colour class \( V_k \), \( k\notin \{0,j\} \). 
Without loss of generality, assume that \( w_1,\dots,w_p\in V_1 \), and \( x_r\in V_{r+1} \) for \( 1\leq r\leq p \) (see Figure~\ref{fig:neighbourhood of vertex in G}). 
For \( 1\leq r\leq p \), \( v \) is the unique neighbour of \( w_r \) in \( V_0 \) and thus \( w_r\notin V^* \). 
In contrast, for \( 1\leq r\leq p \), \( x_r \) is the unique neighbour of \( v \) in \( V_{r+1} \); hence, \( x_r \) must have \( p \) neighbours in \( V_0 \) (due to Claim~3) and thus \( x_r\in V^* \). 
So, \( v \) has exactly \( p\) neighbours in \( V^* \), namely \( x_1,\dots,x_p \). 
Since \( v\in V_0 \) is arbitrary, each vertex in \( V_0 \) has exactly \( p \) neighbours in \( V^* \). 
As a result, the number of edges from \( V_0 \) to \( V^* \) is equal to \( p|V_0| \). 
By definition of \( V^* \), each vertex in \( V^* \) has exactly \( p \) neighbours in \( V_0 \), and hence the number of edges from \( V^* \) to \( V_0 \) is equal to \( p|V^*| \). 
Therefore, we have \( p|V_0|=p|V^*| \) and thus \( |V_0|=|V^*| \). 
Since \( G \) is a \( 2p \)-regular graph, \( \sum_{v\in V_0} |N_G(v)|=2p|V_0| \). 
Therefore, equation~(\ref{eqn:2p-regular if p+2-star colourable eqn1}) implies \( 2p|V_0|=|V|-|V_0|-|V_0|+p|V_0| \). 
That is, \( 2p|V_0|=|V|+(p-2)|V_0| \) or \( |V_0|=|V|/(p+2) \). 
This proves Claim~6. 

Since all colour classes have the same cardinality, \( f \) is an equitable colouring. 
For each vertex \( v \) of \( G \), all \( p+2 \) colours are present in the closed neighbourhood of \( v \) in \( G \) by Claim~4. 
Therefore, \( f \) is a fall colouring of~\( G \). 
\end{proof}
%

\begin{theorem}\label{thm:2p-regular p+2-star colourable iff}
Let \( G(V,E) \) be a \( 2p \)-regular graph on \( n \) vertices where \( p\geq 2 \). 
Then, \( G \) is \( (p+2) \)-star colourable if and only if the vertex set of \( G \) can be partitioned into \( (p+1)(p+2) \) sets \( V_{i}^{j} \) with indices \( i,j\in\{0,1,\dots,p+1\} \) and \( i\neq j \) such that for all \( i \) and \( j \), each vertex in \( V_{i}^{j} \) has exactly \( p \) neighbours in \( \bigcup_{k\notin\{i,j\}} V_{j}^{k} \) and exactly one neighbour in \( V_{k}^{i} \) for each \( k\notin\{i,j\} \) \( ( \)see Figure~\ref{fig:partition into vijs}\( ) \). 
If the vertex set of \( G \) can be partitioned into such sets \( V_{i}^{j} \), then all sets \( V_{i}^{j} \) have the same cardinality and thus \( n \) is divisible by \( (p+1)(p+2) \).
\end{theorem}
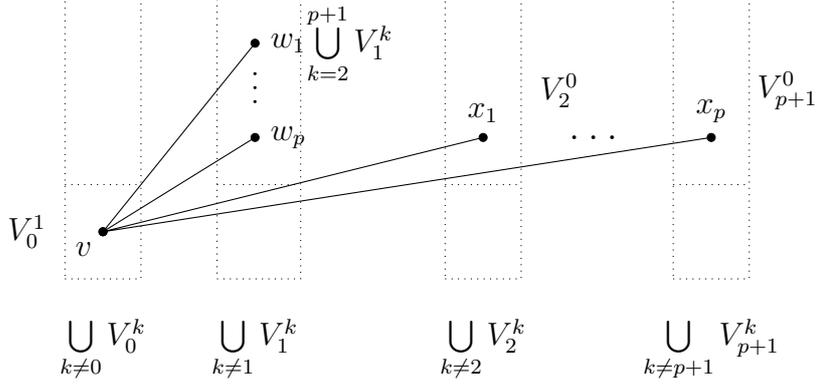
\begin{figure}[hbt]
\centering
\begin{tikzpicture}[x=1cm,y=1.25cm]
\path (0.5,-0.5) coordinate(01) ++(0,-1) coordinate(02) ++(0,-1) coordinate(03) node[below left]{\( v \)};
\path (01) ++(2,0) coordinate(10) ++(0,-1) coordinate(12) ++(0,-1) coordinate(13);
\path (10) ++(3,0) coordinate(20) ++(0,-1) coordinate(21) ++(0,-1) coordinate(23);
\path (20) ++(3,0) coordinate(30) ++(0,-1) coordinate(31) ++(0,-1) coordinate(32);

\draw[dotted,step=1] (0,0) rectangle (1,-3);
\draw[dotted,step=1] (0,-2) -- (1,-2);
\draw[dotted,step=1] (2,0) rectangle (3.1,-3);
\draw[dotted,step=1] (2,-2) -- (3.1,-2);
\draw[dotted,step=1] (5,0) rectangle (6,-3);
\draw[dotted,step=1] (5,-2) -- (6,-2);
\draw[dotted,step=1] (8,0) rectangle (9,-3);
\draw[dotted,step=1] (8,-2) -- (9,-2);

\path (03) node[dot]{};
\draw (03) node[dot]{}
(03) -- (10) node[dot][label=right:\( w_1 \)]{}
(03) -- (12) node[dot][label=right:\( w_p \)]{}
(03) -- (21) node[dot][label=\( x_1 \)]{}
(03) -- (31) node[dot][label=\( x_p \)]{};
\path (10)--node[sloped]{\dots} (12);
\path (21)--node[sloped][font=\Large]{\dots} (31);

\path (03)+(0,-1.25) node{\( \bigcup\limits_{k\neq 0}V_{0}^{k} \)};
\path (13)+(0,-1.25) node{\( \bigcup\limits_{k\neq 1}V_{1}^{k} \)};
\path (23)+(0,-1.25) node{\( \bigcup\limits_{k\neq 2}V_{2}^{k} \)};
\path (32)+(0,-1.25) node{\( \bigcup\limits_{k\neq p+1}V_{p+1}^{k} \)};
\path (03)+(-1,0) node{\( V_{0}^{1} \)};
\path (10)+(1.25,0) node{\( \bigcup\limits_{k=2}^{p+1}V_{1}^{k} \)};
\path (20)+(1,-0.5) node{\( V_{2}^{0} \)};
\path (30)+(1,-0.5) node{\( V_{p+1}^{0} \)};
\end{tikzpicture}
\caption{Neighbours of an arbitrary vertex \( v\in V_{i}^{j} \) (here, \( i=0 \) and \( j=1 \)).}
\label{fig:partition into vijs}
\end{figure}

\begin{proof}
First, we prove the characterization of \( 2p \)-regular \( (p+2) \)-star colourable graphs where \( p\geq 2 \). 

\noindent \emph{Necessary part:}   
Suppose that \( G \) admits a \( (p+2) \)-star colouring \( f\colon V\to \{0,1,\dots,p+1\} \). 
By Theorem~\ref{thm:2p-regular if p+2-star colourable}, every bicoloured component of \( G \) (under \( f \)) is isomorphic to \( K_{1,p} \). 
Moreover, by Claim~5 of Theorem~\ref{thm:2p-regular if p+2-star colourable}, each vertex \( v \) of \( G \), say with colour~\( i \), has exactly \( p \) neighbours in some colour class \( V_j \) and exactly one neighbour in every other colour class \( V_k \), \( k\notin\{i,j\} \).

For every pair of distinct colours \( i \) and \( j \), let \( V_{i}^{j} \) denote the set of vertices \( x\in V_i \) such that \( x \) has exactly \( p \) neighbours in \( V_j \). 
Since each vertex \( v\in V_i \) has exactly \( p \) neighbours in some colour class, \( \{V_{i}^{j}\ :\ 0\leq j\leq p+1 \text{ and } j\neq i\} \) is a partition of \( V_i \). 
Therefore, \( \{V_{i}^{j}\ :\ 0\leq i\leq p+1,\ \allowbreak 0\leq j\leq p+1, \allowbreak \text{ and } i\neq j\} \) is a partition of \( V=V(G) \).\\

\noindent \textbf{Claim~1:} Each vertex in \( V_{i}^{j} \) has exactly \( p \) neighbours in \( \bigcup_{k\notin\{i,j\}} V_{j}^{k} \) and exactly one neighbour in \( V_{k}^{i} \) for each \( k\notin\{i,j\} \).\\[5pt]
We prove Claim~1 for \( i=0 \) and \( j=1 \) (the proof is similar for other values). 
Let \( v\in V_{0}^{1} \).
By definition of set \( V_{0}^{1} \), \( v \) has \( p \) neighbours in \( V_1 \). 
By Claim~5 of Theorem~\ref{thm:2p-regular if p+2-star colourable}, \( v \) has exactly one neighbour in \( V_k \) for \( 2\leq k\leq p+1 \). 
Let \( w_1,\dots,w_p,x_1,\dots,x_p \) be the neighbours of \( v \) where \( w_1,\dots,w_p\in V_1 \) and \( x_r\in V_{r+1} \) for \( 1\leq r\leq p \) (see Figure~\ref{fig:neighbourhood of vertex in G}). 
Recall that every bicoloured component of \( G \) is isomorphic to \( K_{1,p} \). 
For \( 1\leq r\leq p \), \( v \) is the unique neighbour of \( w_r \) in \( V_0 \) and thus \( w_r\notin V_{1}^{0} \). 
Hence, \( v \) has \( p \) neighbours in \( V_1\setminus V_{1}^{0}=\bigcup_{k=2}^{p+1} V_{1}^{k} \). 
On the other hand, for \( 1\leq r\leq p \), \( x_r \) is the unique neighbour of \( v \) in \( V_{r+1} \) and thus \( x_r \) must have \( p \) neighbours in \( V_0 \) (if not, the bicoloured component of \( G \) containing edge \( vx_r \) is isomorphic to \( K_{1,1} \)); that is, \( x_r\in V_{r+1}^{0} \). 
So, for \( 2\leq k\leq p+1 \), \( v \) has a neighbour in \( V_{k}^{0} \).
Therefore, \( v \) has exactly \( p \) neighbours in \( V_1\setminus V_{1}^{0}=\bigcup_{k=2}^{p+1} V_{1}^{k} \) and exactly one neighbour in \( V_{k}^{0} \) for \( 2\leq k\leq p+1 \). 
This proves Claim~1, and thus completes the proof of the necessary part.\\

\noindent \emph{Sufficient part:}   
Suppose that the vertex set of \( G \) can be partitioned into \( (p+1)(p+2) \) sets \( V_{i}^{j} \) with indices \( i,j\in\{0,1,\dots,p+1\} \) and \( i\neq j \) such that for all \( i \) and \( j \), each vertex in \( V_{i}^{j} \) has exactly \( p \) neighbours in \( \bigcup_{k\notin\{i,j\}} V_{j}^{k} \) and exactly one neighbour in \( V_{k}^{i} \) for each \( k\notin\{i,j\} \). 
We claim that the function \( f \) defined as \( f(v)=i \) for all \( v\in V_{i}^{j} \) is a \( (p+2) \)-star colouring of \( G \). 
On the contrary, assume that there is a 4-vertex path \( u,v,w,x \) in \( G \) bicoloured by \( f \). 
Without loss of generality, assume that \( f(u)=f(w)=0 \) and \( f(v)=f(x)=1 \). 
Since \( v \) is coloured~1 and it has two neighbours coloured~0, \( v \) has exactly \( p \) neighbours in \( V_0 \) by Claim~4 of Theorem~\ref{thm:2p-regular if p+2-star colourable}.
Thus, \( v\in V_{1}^{0} \). 
Similarly, since \( w \) is coloured~0 and it has two neighbours coloured~1, \( w\in V_{0}^{1} \). 
By the condition on the sets \( V_{i}^{j} \), the vertex \( w\in V_{0}^{1} \) has exactly \( p \) neighbours in \( V_1\setminus V_{1}^{0}=\bigcup_{k=2}^{p+1} V_{1}^{k} \) and exactly one neighbour in \( V_{k}^{0} \) for \( 2\leq k\leq p+1 \). 
In particular, \( w \) has no neighbour in \( V_{1}^{0} \). 
We have a contradiction since the neighbour \( v \) of \( w \) is in \( V_{1}^{0} \). 
Hence, \( f \) is indeed a \( (p+2) \)-star colouring of \( G \). 
This completes the proof of the sufficient part.\\

Finally, we prove that the condition on the sets \( V_{i}^{j} \) implies that these sets are of equal size.\\[5pt]
\noindent \textbf{Claim~2:} For every pair of indices \( i \) and \( j \), \( |V_{i}^{j}|=\frac{|V|}{\mystrut (p+1)(p+2)} \).\\[5pt]
%
%
We prove Claim~2 for \( i=0 \) and \( j=1 \) (the proof is similar for other values). 
We know that \( G \) admits a \( (p+2) \)-star colouring \( f:V\to\{0,1,\dots,p+1\} \). 
Besides, the colour class \( V_0=\bigcup_{k\neq 0} V_{0}^{k} \) and the colour class \( V_1=\bigcup_{k\neq 1} V_{1}^{k} \). 
Consider the sets \( V_{0}^{1} \), \( V_0\setminus V_{0}^{1} \), \( V_{1}^{0} \) and \( V_1\setminus V_{1}^{0} \). 
By definition of \( V_{1}^{0} \), each vertex in \( V_{1}^{0} \) has exactly \( p \)  neighbours in \( \bigcup_{k=2}^{p+1} V_{0}^{k}=V_0\setminus V_{0}^{1} \). 
We know that every component of \( G[V_0\cup V_1] \) is isomorphic to \( K_{1,p} \) (by Theorem~\ref{thm:2p-regular if p+2-star colourable}). 
So, each vertex in \( V_{1}^{0} \) has exactly \( p \) neighbours in \( V_0\setminus V_{0}^{1} \), and each vertex in \( V_0\setminus V_{0}^{1} \) has exactly one neighbour in \( V_{1}^{0} \). 
Hence,
\begin{equation}
\#\text{edges between } V_{1}^{0} \text{ and } V_0\setminus V_{0}^{1}\ =\ p\,|V_{1}^{0}|\ =\ |V_0\setminus V_{0}^{1}|\ =\ |V_0|-|V_{0}^{1}|.
\end{equation}
Similarly, each vertex in \( V_{0}^{1} \) has exactly \( p \) neighbours in \( V_1\setminus V_{1}^{0}=\bigcup_{k=2}^{p+1} V_{1}^{k} \), and each vertex in \( V_1\setminus V_{1}^{0} \) has exactly one neighbour in \( V_{0}^{1} \). 
Thus, we have
\begin{equation}
\#\text{edges between } V_{0}^{1} \text{ and } V_1\setminus V_{1}^{0}\ =\ p\,|V_{0}^{1}|\ =\ |V_1\setminus V_{1}^{0}|\ =\ |V_1|-|V_{1}^{0}|.
\end{equation}

\noindent We also know that \( |V_0|=|V_1| \) (by Theorem~\ref{thm:2p-regular if p+2-star colourable}). 
Hence, by equations~(2) and (3), \( p\,|V_{1}^{0}|+|V_{0}^{1}|=|V_0|=|V_1|=p\,|V_{0}^{1}|+|V_{1}^{0}| \). 
The equation \( p\,|V_{1}^{0}|+|V_{0}^{1}|=p\,|V_{0}^{1}|+|V_{1}^{0}| \) simplifies to \( (p-1)|V_{1}^{0}|=(p-1)|V_{0}^{1}| \).
As \( p\geq 2 \), we have \( |V_{1}^{0}|=|V_{0}^{1}| \). 
Therefore, equation~(3) implies \( p\,|V_{0}^{1}|=|V_1|-|V_{0}^{1}| \). 
That is, \( |V_{0}^{1}|=\frac{|V_1|}{p+1} \). 
Since all colour classes under \( f \) have the same cardinality \( \frac{|V|}{p+2} \), we have 
\( |V_{0}^{1}|=\frac{|V_1|}{p+1}=\frac{|V|}{(p+1)(p+2)} \). 
This proves Claim~2. 
Since \( |V_{0}^{1}| \) is an integer, \( n=|V| \) is divisible by \( (p+1)(p+2) \). 
\end{proof}
The following corollary improves on the known lower bound \( \ceil{(d+3)/2} \) \cite{fertin2004} for the star chromatic number of the \( d \)-regular hypercube (provided \( d\geq 3 \)). 
\begin{corollary}
Let \( G \) be a \( d \)-regular  hypercube with \( d\geq 3 \). 
Then, \( \chi_s(G)\geq \lceil\frac{d+5}{2} \rceil \).
\end{corollary}
\begin{proof}
The lower bound holds for every odd number \( d\geq 3 \) since \( \ceil{\frac{d+4}{2}}=\ceil{\frac{d+5}{2}} \) and \( \chi_s(G)\geq \ceil{\frac{d+4}{2}} \) (see Theorem~\ref{thm:lb chi_s}). 
Hence, it suffices to establish the lower bound for every even number \( d\geq 4 \). 
Suppose that \( d=2p \) where \( p\geq 2 \). 
To prove that \( \chi_s(G)\geq \lceil(d+5)/2 \rceil =p+3 \), it suffices to show that \( G \) does not admit a \( (p+2) \)-star colouring. 
On the contrary, assume that \( G \) admits a \( (p+2) \)-star colouring. 
By Theorem~\ref{thm:2p-regular p+2-star colourable iff}, \( n=|V(G)| \) is divisible by \( (p+1)(p+2) \). 
Since \( p+1 \) or \( p+2 \) is an odd number greater than one, \( n \) is divisible by an odd number greater than one.
Since the number of vertices in a hypercube is a power of two, we have a contradiction. 
\end{proof}
\begin{corollary}\label{cor:diamond-free omega alpha etc}
Let \( G \) be a \( 2p \)-regular \( (p+2) \)-star colourable graph where \( p\geq 2 \). Then, the following hold: \( (i) \)~\( G \) is \( (\textup{diamond},  K_4) \)-free, \( (ii) \)~\( \alpha(G)>n/4 \), \( (iii) \)~\( \chi(G)\leq 3\log_2(p+2) \), \( (iv) \)~\( G \) admits a \( P_4 \)-decomposition, and \( (v) \)~if \( G \) contains no asteroidal triple, then \( G \) is 3-colourable. 
\end{corollary}
\noindent The proof of Corollary~\ref{cor:diamond-free omega alpha etc} is deferred to the end of the section.


\begin{figure}[hbt]
\centering
\begin{subfigure}[b]{0.575\textwidth}
\centering
\begin{tikzpicture}[scale=2.25]
\path (0,0) node(01)[bigdot][label=left:{\( (0,1) \)}]{}++(0,-1) node(02)[bigdot][label=left:{\( (0,2) \)}]{}++(0,-1) node(03)[bigdot][label=left:{\( (0,3) \)}]{};
\path (01)++(1,0) node(10)[bigdot][label=left:{\( (1,0) \)}]{}++(0,-1) node(12)[bigdot][label={[font=\small,xshift=5pt,yshift=-4pt]left:{\( (1,2) \)}}]{}++(0,-1) node(13)[bigdot][label=left:{\( (1,3) \)}]{};
\path (10)++(1,0) node(20)[bigdot][label=right:{\( (2,0) \)}]{}++(0,-1) node(21)[bigdot][label={[font=\small,xshift=-5pt,yshift=4pt]right:{\( (2,1) \)}}]{}++(0,-1) node(23)[bigdot][label=right:{\( (2,3) \)}]{};
\path (20)++(1,0) node(30)[bigdot][label=right:{\( (3,0) \)}]{}++(0,-1) node(31)[bigdot][label=right:{\( (3,1) \)}]{}++(0,-1) node(32)[bigdot][label=right:{\( (3,2) \)}]{};

\draw
(01)--(12)
(01)--(13)
(10)--(02)
(10)--(03)
(12)--(20)
(12)--(23)
(21)--(10)
(21)--(13)
(23)--(30)
(23)--(31)
(32)--(20)
(32)--(21);
\draw
(02)--(23)
(02)--(30)
(31)--(10)
(31)--(03);
\draw
(03) to[bend left=20] (20)
(13) to[bend right=20] (30)
(12) to[bend left=20] (31)
(02) to[bend right=20] (21)
(01) to[bend left=15] (20)
(01) to[bend left=20] (30)
(32) to[bend left=15] (13)
(32) to[bend left=20] (03);
\end{tikzpicture}
\caption{}
\end{subfigure}%
\hfill 
\begin{subfigure}[b]{0.375\textwidth}
\centering
\begin{tikzpicture}[scale=1.25]
\path (0:2) node(a)[bigdot][label=right:{\( (0,1) \)}]{};
\path (120:2) node(u)[bigdot][label=above:{\( (2,0) \)}]{};
\path (-120:2) node(w)[bigdot][label=below:{\( (1,2) \)}]{};
\draw (a)--(u)--(w)--(a);
\draw (0,0) circle (2);
\path (a)--(u) node(v)[pos=0.4][bigdot][label=right:{\( (1,3) \)}]{} node(u-a)[pos=0.6][bigdot][label=above:{\( (3,2) \)}]{};
\path (u)--(w) node(u-w)[pos=0.4][bigdot][label=left:{\( (0,3) \)}]{} node(w-u)[pos=0.6][bigdot][label=left:{\( (3,1) \)}]{};
\path (w)--(a) node(w-a)[pos=0.4][bigdot][label=below:{\( (2,3) \)}]{} node(x)[pos=0.6][bigdot][label=right:{\( (3,0) \)}]{};
\draw (u-w)--(u-a)--node[pos=0.3][bigdot][label={[label distance=-6pt,font=\scriptsize]above left:{\( (2,1) \)}}]{} (w-a)--(w-u)--node[pos=0.3][bigdot][label={[label distance=-2pt,font=\scriptsize]above:{\( (1,0) \)}}]{} (v)--(x)--node[pos=0.3][bigdot][label={[label distance=-6pt,font=\scriptsize]below left:{\( (0,2) \)}}]{} (u-w);
\end{tikzpicture}
\caption{}
\end{subfigure}%
\caption{(a)~Graph \( G_4 \), and (b)~its plane drawing.}
\label{fig:G4}
\end{figure}
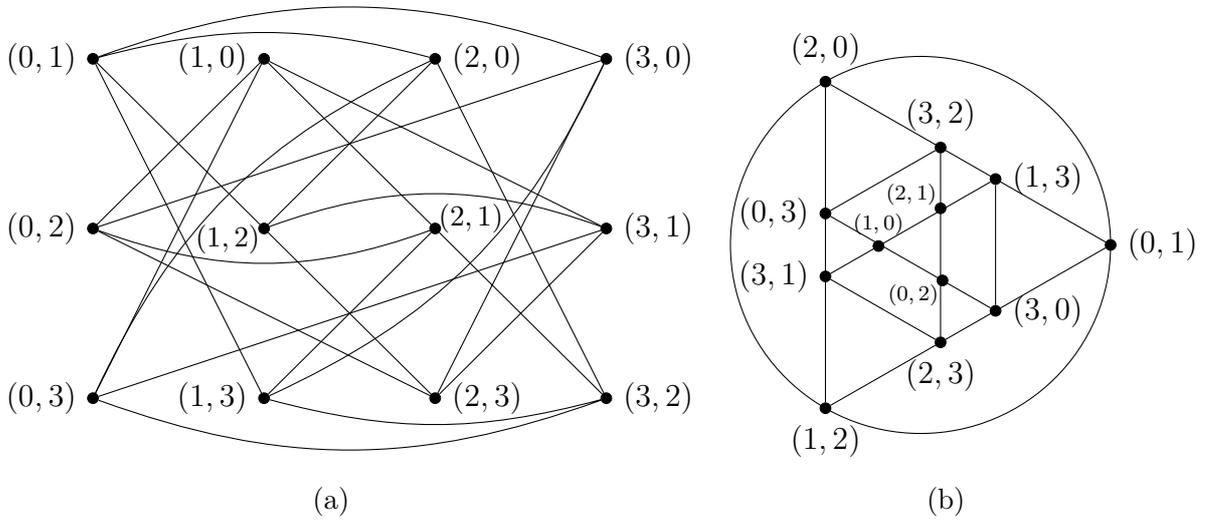

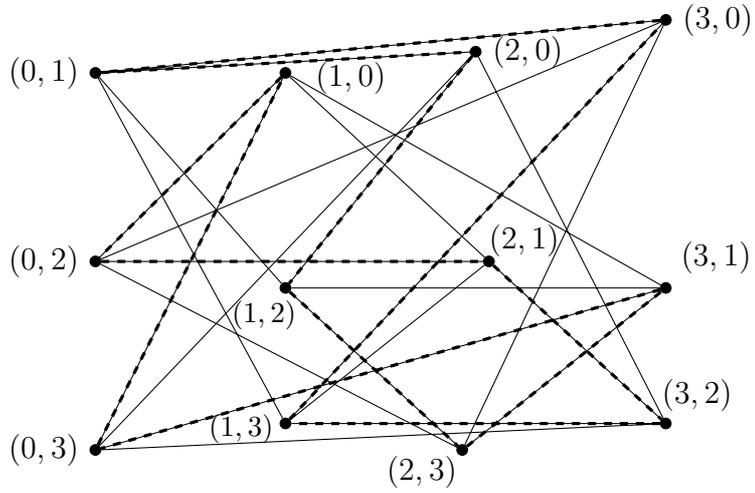
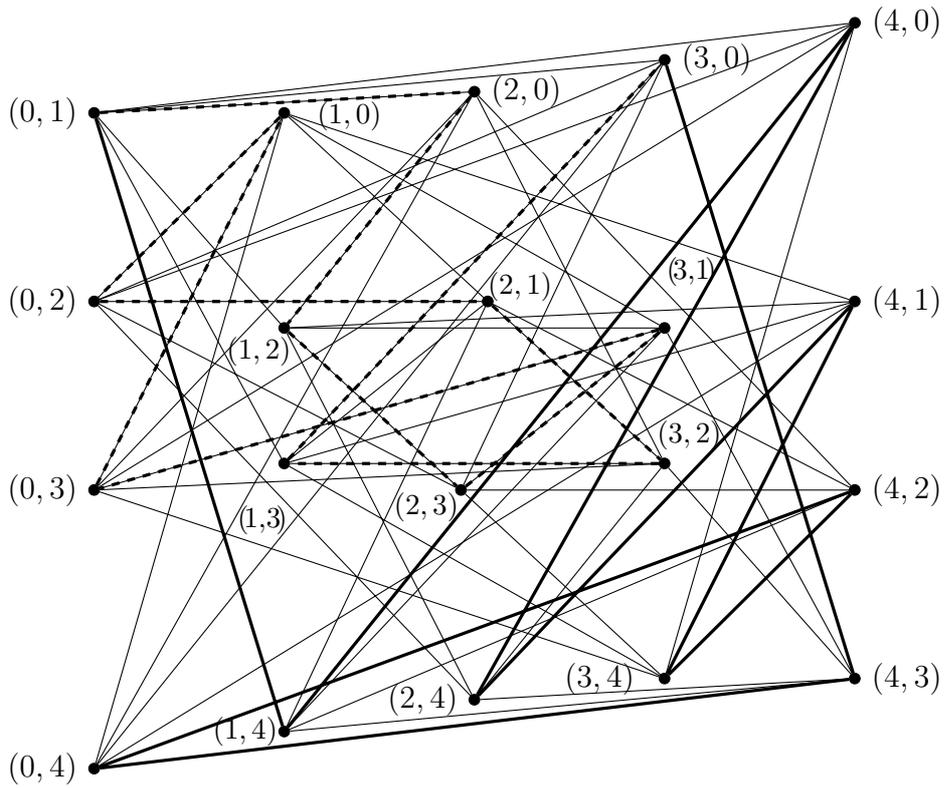
\begin{figure}[hbtp]
\centering
\begin{subfigure}[b]{0.8\textwidth}
\centering
\begin{tikzpicture}[scale=2.5]
\path (0,0) node(01)[bigdot][label=left:{\( (0,1) \)}]{}++(0,-1) node(02)[bigdot][label=left:{\( (0,2) \)}]{}++(0,-1) node(03)[bigdot][label=left:{\( (0,3) \)}]{};
\path (01)++(1,0) node(10)[bigdot][label={[label distance=5pt,yshift=-2pt]right:{\( (1,0) \)}}]{}++(0,-1) node(12)[bigdot][yshift=-10pt][label={[label distance=-14pt,yshift=-9pt,font=\small]below left:{\( (1,2) \)}}]{}++(0,-1) node(13)[bigdot][yshift=10pt][label={[xshift=2pt,yshift=-2pt,font=\small]left:{\( (1,3) \)}}]{};
\path (10)++(1,0) coordinate(20pos)++(0,-1) node(21)[bigdot][xshift=5pt][label={[label distance=-8pt,yshift=2pt]above right:{\( (2,1) \)}}]{}++(0,-1) node(23)[bigdot][xshift=-5pt][label={[label distance=-5pt]below left:{\( (2,3) \)}}]{};
\path (20pos)++(1,0) coordinate(30pos)++(0,-1) node(31)[bigdot][yshift=-10pt][label={above right:{\( (3,1) \)}}]{}++(0,-1) node(32)[bigdot][yshift=10pt][label={[label distance=-10pt,yshift=6pt]above right:{\( (3,2) \)}}]{};

\path
(20pos) node(20)[bigdot][yshift=8pt][label=right:{\( (2,0) \)}]{}
(30pos) node(30)[bigdot][yshift=20pt][label=right:{\( (3,0) \)}]{}
;

\draw
(01)--(12)
(01)--(13)
(02)--(21)
(02)--(23)
(03)--(31)
(03)--(32);
\draw
(10)--(02)
(10)--(03)
(12)--(20)
(12)--(23)
(13)--(30)
(13)--(32);
\draw
(20)--(01)
(20)--(03)
(21)--(10)
(21)--(13)
(23)--(30)
(23)--(31);
\draw
(30)--(01)
(30)--(02)
(31)--(10)
(31)--(12)
(32)--(20)
(32)--(21);

\draw [very thick,dashed]
(01)--(20)--(12)--(23)--(31)--(03)--(10)--(02)--(21)--(32)--(13)--(30)--(01);
\end{tikzpicture}
\caption{}
\label{fig:Hcycle in G4}
\end{subfigure}%

\begin{subfigure}[b]{0.8\textwidth}
\centering
\begin{tikzpicture}[scale=2.5]
\path (0,0) node(01)[bigdot][label=left:{\( (0,1) \)}]{}++(0,-1) node(02)[bigdot][label=left:{\( (0,2) \)}]{}++(0,-1) node(03)[bigdot][label=left:{\( (0,3) \)}]{}++(0,-1) node(04)[bigdot][yshift=-34pt][label=left:{\( (0,4) \)}]{};
\path (01)++(1,0) node(10)[bigdot][label={[label distance=6pt,yshift=-1pt,font=\small]right:{\( (1,0) \)}}]{}++(0,-1) node(12)[bigdot][yshift=-10pt][label={[label distance=-12pt,yshift=-6pt,font=\small]below left:{\( (1,2) \)}}]{}++(0,-1) node(13)[bigdot][yshift=10pt][label={[xshift=7pt,yshift=-10pt,font=\small]below left:{\( (\!1,\!3\!) \)}}]{}++(0,-1) node(14)[bigdot][yshift=-20pt][label={[xshift=4pt,font=\small]left:{\( (1,4) \)}}]{};
\path (10)++(1,0) coordinate(20pos)++(0,-1) node(21)[bigdot][xshift=5pt][label={[label distance=-8pt,font=\small]above right:{\( (2,1) \)}}]{}++(0,-1) node(23)[bigdot][xshift=-5pt][label={[label distance=-7pt,font=\small]below left:{\( (2,3) \)}}]{}++(0,-1) node(24)[bigdot][yshift=-8pt][label=left:{\( (2,4) \)}]{};
\path (20pos)++(1,0) coordinate(30pos)++(0,-1) node(31)[bigdot][yshift=-10pt][label={[xshift=-5pt,yshift=10pt,font=\small]above right:{\( (\!3,\!1\!) \)}}]{}++(0,-1) node(32)[bigdot][yshift=10pt][label={[label distance=-12pt,yshift=8pt,font=\small]above right:{\( (3,2) \)}}]{}++(0,-1) node(34)[bigdot][label={[label distance=5pt]left:{\( (3,4) \)}}]{};
\path (30pos)++(1,0) coordinate(40pos)++(0,-1) node(41)[bigdot][label=right:{\( (4,1) \)}]{}++(0,-1) node(42)[bigdot][label=right:{\( (4,2) \)}]{}++(0,-1) node(43)[bigdot][label=right:{\( (4,3) \)}]{};

\path
(20pos) node(20)[bigdot][yshift=8pt][label=right:{\( (2,0) \)}]{}
(30pos) node(30)[bigdot][yshift=20pt][label=right:{\( (3,0) \)}]{}
(40pos) node(40)[bigdot][yshift=34pt][label=right:{\( (4,0) \)}]{}
;

\draw
(01)--(12)
(01)--(13)
(01)--(14)
(02)--(21)
(02)--(23)
(02)--(24)
(03)--(31)
(03)--(32)
(03)--(34)
(04)--(41)
(04)--(42)
(04)--(43);
\draw
(10)--(02)
(10)--(03)
(10)--(04)
(12)--(20)
(12)--(23)
(12)--(24)
(13)--(30)
(13)--(32)
(13)--(34)
(14)--(40)
(14)--(42)
(14)--(43);
\draw
(20)--(01)
(20)--(03)
(20)--(04)
(21)--(10)
(21)--(13)
(21)--(14)
(23)--(30)
(23)--(31)
(23)--(34)
(24)--(40)
(24)--(41)
(24)--(43);
\draw
(30)--(01)
(30)--(02)
(30)--(04)
(31)--(10)
(31)--(12)
(31)--(14)
(32)--(20)
(32)--(21)
(32)--(24)
(34)--(40)
(34)--(41)
(34)--(42);
\draw
(40)--(01)
(40)--(02)
(40)--(03)
(41)--(10)
(41)--(12)
(41)--(13)
(42)--(20)
(42)--(21)
(42)--(23)
(43)--(30)
(43)--(31)
(43)--(32);

\draw [very thick]
(01)--(14)--(40)--(24)--(41)--(34)--(42)--(04)--(43)--(30);
\draw [very thick,dashed]
(01)--(20)--(12)--(23)--(31)--(03)--(10)--(02)--(21)--(32)--(13)--(30);
\end{tikzpicture}
\caption{}
\end{subfigure}%
%
\caption{(a)~A Hamiltonian cycle \( L_4 \) of \( G_4 \) containing edge \( \{(0,1),(3,0)\} \), and (b)~the Hamiltonian cycle of \( G_6 \) obtained by replacing edge \( \{(0,1),(3,0)\} \) of cycle \( L_4 \) by the path \( (0,1),(1,4),(4,0),(2,4),(4,1),(3,4),(4,2),(0,4),(4,3),(3,0) \).}
\label{fig:making Hcycle in G2p}
\end{figure}

\begin{theorem}\label{thm:intro G2p}
For every \( p\geq 2 \), there exists a unique \( 2p \)-regular \( (p+2) \)-star colourable graph \( G_{2p} \) on \( (p+1)(p+2) \) vertices. 
Moreover, \( G_{2p} \) is vertex-transitive, edge-transitive and Hamiltonian for every \( p \geq 2 \).
\end{theorem}
\begin{proof}
By Theorem~\ref{thm:2p-regular p+2-star colourable iff}, if a \( 2p \)-regular graph \( G \) is \( (p+2) \)-star colourable, then the number of vertices in \( G \) is divisible by \( (p+1)(p+2) \), and thus \( G \) has at least \( (p+1)(p+2) \) vertices. 
The following claim is a direct consequence of the characterization of \( 2p \)-regular \( (p+2) \)-colourable graphs given in Theorem~\ref{thm:2p-regular p+2-star colourable iff}.\\[5pt]
\noindent \textbf{Claim 1:} For \( p\geq 2 \), a \( 2p \)-regular graph \( G \) on \( (p+1)(p+2) \) vertices is \( (p+2) \)-star colourable if and only if the vertex set of \( G \) can be partitioned into \( (p+1)(p+2) \) singleton sets \( V_{i}^{j} \) with indices \( i,j\in\{0,1,\dots,p+1\} \) and \( i\neq j \) such that for all \( i \) and \( j \), the unique vertex in \( V_{i}^{j} \) is adjacent to the unique vertex in \( V_{j}^{k} \) and the unique vertex in \( V_{k}^{i} \) for each \( k\notin\{i,j\} \).\\

\noindent The definition of \( G_{2p} \) is motivated by Claim~1. 
The vertex set of \( G_{2p} \) is \( \{(i,j)\ :\ 0\leq i\leq p+1,\ \allowbreak 0\leq j\leq p+1, \text{ and } i\neq j\} \). 
A vertex \( (i,j) \) of \( G_{2p} \) is adjacent to a vertex \( (k,\ell)\) if (i)~\( k=j \) and \( \ell\notin \{i,j\} \), or (ii)~\( k\notin \{i,j\} \) and \( \ell=i \). 
The graph \( G_4 \) is displayed in Figure~\ref{fig:G4}. 
It is easy to verify that the sets \( V_{i}^{j}\coloneqq \{(i,j)\} \) satisfy the condition in Claim~1. 
Therefore, \( G_{2p} \) is \( (p+2) \)-star colourable. 
Moreover, every \( 2p \)-regular \( (p+2) \)-star colourable graph \( G \) on \( (p+1)(p+2) \) vertices is isomorphic to \( G_{2p} \) (call the unique vertex in \( V_{i}^{j} \) as \( (i,j) \) for all \( i \) and \( j \)). 

By the definition of \( G_{2p} \), whether vertex \( (i,j) \) is adjacent to vertex \( (k,\ell) \) depends only on equality and inequality between indices \( i,j,k,\ell \). 
Hence, for every bijection \( h \) from \( \{0,1,\dots,p+1\} \) to itself, relabelling each index \( i\in\{0,1,\dots,p+1\} \) by \( h(i) \) in vertex labels of \( G_{2p} \) (e.g.: relabel \( (0,1) \) as \( (h(0),h(1)) \)) gives \( G_{2p} \) itself. 

Thus, we have the following.\\[3pt] 
\noindent \textbf{Claim~2:} For every bijection \( h \) from \( \{0,1,\dots,p+1\} \) to itself, the function \( \psi\colon V(G_{2p})\to V(G_{2p}) \) defined as \( \psi((x,y))=(h(x),h(y)) \) is an automorphism of \( G_{2p} \).\\

With the help of Claim~2, we show that \( G_{2p} \) is vertex-transitive and edge-transitive.\\[5pt]
\textbf{Claim~3:} \( G_{2p} \) is vetex-transitive for each \( p\geq 2 \).\\[5pt]
To construct an automorphism \( \psi \) that maps a vertex \( (i,j) \) to a vertex \( (k,\ell) \), first choose a bijection \( h \) from \( \{0,1,\dots,p+1\} \) to itself such that \( h(i)=k \) and \( h(j)=\ell \), and then define \( \psi\big((x,y)\big)=\big(h(x),h(y)\big) \) for all \( (x,y)\in V(G_{2p}) \). 
This proves Claim~3.\\

\noindent \textbf{Claim~4:} \( G_{2p} \) is edge-transitive for each \( p\geq 2 \).\\[5pt]
For each vertex \( (i,j) \) of \( G_{2p} \), each neighour of \( (i,j) \) in \( G_{2p} \) is either of the form \( (j,k) \) for some \( k\notin\{i,j\} \) or \( (k,i) \) for some \( k\notin\{i,j\} \). 
So, edges incident on \( (i,j) \) are \( \{(i,j),(j,k)\} \) where \( k\notin\{i,j\} \) or \( \{(k,i),(i,j)\} \) where \( k\notin\{i,j\} \). 
As a result, each edge of \( G_{2p} \) is of the form \( \{(q,r),(r,s)\} \) where \( q,r,s\in\{0,1,\dots,p+1\} \), \( q\neq r \), \( r\neq s \) and \( s\neq q \). 
To construct an automorphism \( \psi \) that maps an edge \( \{(i,j),(j,k)\} \) to an edge \( \{(q,r),(r,s)\} \), first choose a bijection \( h \) from \( \{0,1,\dots,p+1\} \) to itself such that \( h(i)=q \), \( h(j)=r \) and \( h(k)=s \), and then define \( \psi\big((x,y)\big)=\big(h(x),h(y)\big) \) for all \( (x,y)\in V(G_{2p}) \). 
This proves Claim~4.\\

Next, we prove that \( G_{2p} \) is Hamiltonian for \( p\geq 2 \). 
We employ induction on \( p\geq 2 \).\\
\noindent \textit{Base case} \( (p=2) \): Figure~\ref{fig:Hcycle in G4} exhibits a Hamiltonian cycle in \( G_4 \). 
This proves the base case.\\
\noindent \textit{Induction step} (\( p\geq 3 \)): Assume that \( G_{2(p-1)} \) is Hamiltonian. 
Since \( G_{2(p-1)} \) is edge-transitive (see Claim~4), \( G_{2(p-1)} \) has a Hamiltonian cycle \( L_{2(p-1)} \) containing the edge \( \{(0,1),(p,0)\} \). 
In \( G_{2p} \), \( L_{2(p-1)} \) is a cycle, and the only vertices not in \( L_{2(p-1)} \) are \( (0,p+1),(1,p+1),\dots,(p,p+1),\allowbreak \) \( (p+1,0),\allowbreak (p+1,1),\dots,(p+1,p) \). 
Replacing the edge \( \{(0,1),(p,0)\} \) of the cycle \( L_{2(p-1)} \) in \( G_{2p} \) by the path \( (0,1),(1,p+1),(p+1,0),(2,p+1),(p+1,1),\dots,(p,p+1),(p+1,p-1),(0,p+1),(p+1,p),(p,0) \) gives a Hamiltonian cycle of \( G_{2p} \) (see Figure~\ref{fig:making Hcycle in G2p} for an example). 
This proves the induction step. 
\end{proof}

Next, we show that the lower bound \( \ceil{(d+4)/2} \) for the star chromatic number of \( d \)-regular graphs established by Theorem~\ref{thm:lb chi_s} is attained for all \( d\geq 2 \). 
\begin{theorem}\label{thm:chi_s attained}
For all \( d\geq 2 \), there exists a \( d \)-regular graph \( G \) with \( \chi_s(G)=\raisebox{1pt}{\big\lceil}\frac{d+4}{2}\raisebox{1pt}{\big\rceil} \).
\end{theorem}
\begin{proof}
By Theorem~\ref{thm:lb chi_s}, \( \chi_s(G)\geq \ceil{(d+4)/2} \) for every \( d \)-regular graph \( G \) with \( d\geq 2 \). 
Hence, to prove the theorem, it suffices to show that for every \( d\geq 2 \), there exists a \( d \)-regular graph \( G \) with \( \chi_s(G)\leq \lceil(d+4)/2\rceil \).

First, we consider the case when \( d \) is even, say \( d=2p \). 
If \( p=1 \), then there exists a 2-regular graph \( C_4 \) such that \( \chi_s(C_4)=3=p+2 \). 
For \( p\geq 2 \), there exists a \( 2p \)-regular graph \( G_{2p} \) with \( \chi_s(G_{2p})\leq p+2 \) by Theorem~\ref{thm:intro G2p}. 
So, for all \( p\geq 1 \), there exists a \( 2p \)-regular graph \( G \) with \( \chi_s(G)\leq p+2=\ceil{(d+4)/2} \). 
That is, for every even number \( d\geq 2 \), there exists a \( d \)-regular graph with the star chromatic number at most \( (d+4)/2 \), and thus the lower bound \( \ceil{(d+4)/2} \) for the star chromatic number is attained for \( d \). 

Next, we consider the case when \( d \) is odd, say \( d=2p-1 \) for some \( p\geq 2 \). 
We know that \( G_{2p} \) is a \( 2p \)-regular Hamiltonian graph on \( (p+1)(p+2) \) vertices (see Theorem~\ref{thm:intro G2p}). 
Since \( G_{2p} \) is a Hamiltonian graph on an even number of vertices, \( G_{2p} \) admits a perfect matching \( M \) (pick alternate edges from a Hamiltonian cycle). 
Hence, \( H_{2p}\coloneqq G_{2p}-M \) is a \( (2p-1) \)-regular \( (p+2) \)-star colourable graph. 
So, \( \chi_s(H_{2p})\leq p+2=\ceil{(d+4)/2} \). 
This proves that the lower bound \( \ceil{(d+4)/2} \) for the star chromatic number of \( d \)-regular graphs is attained for every odd number \( d\geq 3 \) as well. 
\end{proof}


Next, we show that the structure of \( 2p \)-regular \( (p+2) \)-star colourable graphs proved in Theorem~\ref{thm:2p-regular p+2-star colourable iff} can be expressed in the Locally Checkable Vertex Subset and Partitioning problems (LC-VSP) framework of Telle and Proskurowski~\cite{telle_proskurowski}. 
For a fixed integer \( q\geq 1 \) and a fixed \( q\times q \) matrix \( D_q \) each entry of which is a subset of \( \mathbb{Z}_0^+\coloneqq\{0,1,\dots\} \), the \pagetarget{def:D_q-partition}{\( \exists D_q \)\emph{-partition} problem} in the LC-VSP framework is the decision problem that takes a graph \( G \) as input and asks whether the vertex set of \( G \) can be partitioned into \( q \) sets \( U_1,U_2,\dots,U_q \) such that for every \( i \) and \( j \), each vertex \( v\in U_i \) satisfy \( |N_G(v)\cap U_j|\in D_q[i,j] \) \big(we write the \( (i,j) \)-th entry of \( D_q \) as \( D_q[i,j] \)\,\big). 
By Theorem~\ref{thm:2p-regular p+2-star colourable iff}, a \( 2p \)-regular graph \( G \) is \( (p+2) \)-star colourable if and only if the vertex set of \( G \) can be partitioned into sets \( V_{i}^{j} \) with indices \( i,j\in\{0,1,\dots,p+1\} \) and \( i\neq j \) such that for all \( i \) and \( j \), each vertex in \( V_{i}^{j} \) has exactly one neighbour in \( V_{k}^{i} \) for each \( k\neq j \) and exactly \( p \) neighbours in \( \bigcup_{k\neq i} V_{j}^{k} \). 
We can rephrase the condition on sets \( V_{i}^{j} \) as follows: for all \( i \) and \( j \), each vertex \( v \) in \( V_{i}^{j} \) has exactly one neighbour in \( V_{k}^{i} \) for each \( k\neq j \), \( v \) has no neighbour in \( V_{j}^{i} \), and \( v \) has no neighbour in \( V_{k}^{\ell} \) for each \( k\neq j \) and \( \ell\neq i \). 
This formulation of the structure can be directly expressed as the \( \exists D_q \)-partition problem in the LC-VSP framework where \( q=(p+1)(p+2) \) and \pagetarget{def:D_q}{\( D_q \)} is the symmetric \( q\times q \) matrix whose rows represent sets \( V_{i}^{j} \) and the entry of \( D_q \) at the intersection of the row for \( V_{i}^{j} \) and the column for \( V_{k}^{\ell} \) is \( \{1\} \) if \( k\neq j \) and \( \ell=i \), the entry is \( \mathbb{Z}_0^+ \) if \( k=j \) and \( \ell\neq i \), and the entry is \( \{0\} \) in all other cases. 
For the special case \( p=2 \), the matrix \( D_q \) is given below.
\begin{equation*}
  D_{12}=
  \begin{blockarray}{lccc@{\hskip 6mm}ccc@{\hskip 6mm}ccc@{\hskip 6mm}ccc}
    \begin{block}{lccc@{\hskip 6mm}ccc@{\hskip 6mm}ccc@{\hskip 6mm}ccc}
              &  V_{0}^{1} & V_{0}^{2} &  V_{0}^{3} & V_{1}^{0} & V_{1}^{2} & V_{1}^{3} & V_{2}^{0} &  V_{2}^{1} & V_{2}^{3} & V_{3}^{0} & V_{3}^{1} & V_{3}^{2}\\
    \end{block}
    \begin{block}{l[ccc@{\hskip 6mm}ccc@{\hskip 6mm}ccc@{\hskip 6mm}ccc]}
      V_{0}^{1}  & \{0\}         & \{0\}         &  \{0\}         & \{0\}         & \mathbb{Z}_0^+  & \mathbb{Z}_0^+  & \{1\}         & \{0\}         & \{0\}         & \{1\}         & \{0\}       &\{0\}\topstrut \\
      V_{0}^{2}  & \{0\}         & \{0\}         &  \{0\}         & \{1\}         & \{0\}         & \{0\}         & \{0\}         & \mathbb{Z}_0^+  & \mathbb{Z}_0^+  & \{1\}         & \{0\}       &             \{0\}\\
      V_{0}^{3}  & \{0\}         & \{0\}         &  \{0\}         & \{1\}         & \{0\}         & \{0\}         & \{1\}         & \{0\}         & \{0\}         & \{0\}         & \mathbb{Z}_0^+& \mathbb{Z}_0^+\\[2mm]
      V_{1}^{0}  & \{0\}         & \mathbb{Z}_0^+  & \mathbb{Z}_0^+  & \{0\}         & \{0\}         & \{0\}         & \{0\}         & \{1\}         & \{0\}         & \{0\}         & \{1\}       &             \{0\}\\
      V_{1}^{2}  & \{1\}         & \{0\}         & \{0\}         & \{0\}         & \{0\}         & \{0\}         & \mathbb{Z}_0^+  & \{0\}         & \mathbb{Z}_0^+  & \{0\}         & \{1\}       &             \{0\}\\
      V_{1}^{3}  & \{1\}         & \{0\}         & \{0\}         & \{0\}         & \{0\}         & \{0\}         & \{0\}         & \{1\}         & \{0\}         & \mathbb{Z}_0^+  & \{0\}       & \mathbb{Z}_0^+\\[2mm]
      V_{2}^{0}  & \mathbb{Z}_0^+  & \{0\}         & \mathbb{Z}_0^+  & \{0\}         & \{1\}         & \{0\}         & \{0\}         & \{0\}         & \{0\}         & \{0\}         & \{0\}       &             \{1\}\\
      V_{2}^{1}  & \{0\}         & \{1\}         &  \{0\}        & \mathbb{Z}_0^+  & \{0\}         & \mathbb{Z}_0^+   & \{0\}         & \{0\}         & \{0\}         & \{0\}         & \{0\}       &            \{1\}\\
      V_{2}^{3}  & \{0\}         & \{1\}         &  \{0\}        & \{0\}         & \{1\}         & \{0\}          & \{0\}         & \{0\}         & \{0\}         & \mathbb{Z}_0^+  &  \mathbb{Z}_0^+  &            \{0\}\\[2mm]
      V_{3}^{0}  & \mathbb{Z}_0^+  &  \mathbb{Z}_0^+ &  \{0\}         & \{0\}         & \{0\}         & \{1\}         & \{0\}         & \{0\}         & \{1\}         & \{0\}         & \{0\}       &            \{0\}\\
      V_{3}^{1}  & \{0\}         & \{0\}         &  \{1\}         &  \mathbb{Z}_0^+ &  \mathbb{Z}_0^+ & \{0\}         & \{0\}         & \{0\}         & \{1\}         & \{0\}         & \{0\}       &            \{0\}\\
      V_{3}^{2}  & \{0\}         & \{0\}         &  \{1\}         & \{0\}         & \{0\}         & \{1\}         &  \mathbb{Z}_0^+ &  \mathbb{Z}_0^+  & \{0\}         & \{0\}         & \{0\}       &            \{0\}\botstrut \\
    \end{block}
  \end{blockarray}
\end{equation*}

The entry of the matrix \( D_{12} \) at the intersection of the first row (i.e., row for \( V_{0}^{1} \)) and the seventh column (i.e., column for \( V_{2}^{0} \)) is \( D_{12}[1,7]=\{1\} \). 
Allow us to explain the first row of \( D_{12} \) in detail. 
Consider an arbitrary vertex \( v \) in \( V_{0}^{1} \). 
The first row contains two \( \{1\} \) entries: one for column \( V_{2}^{0} \) and one for column \( V_{3}^{0} \). 
These entries demand that \( v \) has exactly one neighbour in \( V_{2}^{0} \) and exactly one neighbour in \( V_{3}^{0} \). 
The first row also contains two \( \mathbb{Z}_0^{+} \) entries: one for column \( V_{1}^{2} \) and one for column \( V_{1}^{3} \). 
These entries put no restriction on the number of neighbours of \( v \) in sets \( V_{1}^{2} \) and \( V_{1}^{3} \). 
Every other entry of the first row is \( \{0\} \). 
These entries demand that \( v \) has no neighbour outside of \( V_{2}^{0}\cup V_{3}^{0}\cup V_{1}^{2}\cup V_{1}^{3} \). 
Hence, by the first row of \( D_{12} \), \( v \) has exactly one neighbour in \( V_{2}^{0} \), exactly one neighbour in \( V_{3}^{0} \), and exactly two neighbours in \( V_{1}^{2}\cup V_{1}^{3} \) (note that \( \deg_G(v)=4 \) since \( G \) is a regular graph with degree \( 2p=4 \)). 

For each \( p\geq 2 \), a \( 2p \)-regular graph \( G \) is \( (p+2) \)-star colourable if and only if  \( G \) admits a \( D_q \)-partition; also, each entry of the degree constraint matrix \( D_q \) is either finite or co-finite. 
Thanks to results in \cite{bui-xuan2010,bui-xuan2013,telle_proskurowski}, this proves that for each \( p\geq 2 \), the problem of testing whether a \( 2p \)-regular graph is \( (p+2) \)-star colourable admits an FPT algorithm with the parameter either treewidth, cliquewidth, rankwidth, or booleanwidth (see also \cite{oum,telle}). 
In particular, for graph classes with bounded treewidth (resp.\ cliquewidth, rankwidth, or booleanwidth), the problem is polynomial-time solvable. 
Moreover, by results from \cite{belmonte}, the problem is polynomial-time solvable in several other graph classes including interval graphs, permutation graphs, trapezoid graphs, convex graphs and Dilworth-\( k \) graphs. 
It is worth mentioning that the problem also fits in the framework of Gerber and Kobler~\cite{gerber_kobler} since every entry in \( D_q \) is a set of consecutive integers.

Motivated by the structure of \( G_{2p} \), we define a subclass \( \mathscr{G}_{2p} \) of the class of \( 2p \)-regular \( (p+2) \)-star colourable graphs. 
A graph \( G \) belongs to \( \mathscr{G}_{2p} \) if and only if the vertex set of \( G \) can be partitioned into \( (p+1)(p+2) \) sets \( V_{i}^{j} \) with indices \( i,j\in \{0,1,\dots,p+1\} \) and \( i\neq j \) such that for all \( i \) and \( j \), each vertex in \( V_{i}^{j} \) has exactly one neighbour in \( V_{j}^{k} \) and exactly one neighbour in \( V_{k}^{i} \) for each \( k\notin \{i,j\} \). 
In other words, a graph \( G \) is in \( \mathscr{G}_{2p} \) if and only if \( G \) admits an \( E_q \)-partition where \( q=(p+1)(p+2) \) and \( E_q \) is the matrix obtained from \( D_q \) by replacing entries \( \mathbb{Z}_0^{+} \) by \( \{1\} \). 
Hence, \( 2p \)-regular graphs in \( \mathscr{G}_{2p} \) are \( (p+2) \)-star colourable. 
Observe that for each \( p\geq 2 \), \( E_q \) is precisely the matrix obtained from the adjacency matrix of \( G_{2p} \) by replacing entries 0 by \( \{0\} \) and entries 1 by \( \{1\} \). 
This natural connection between graph \( G_{2p} \) and the class \( \mathscr{G}_{2p} \) can be expressed by the notion of locally bijective homomorphisms. 
Given a graph \( H \) with vertex set \( \{u_1,u_2,\dots,u_r\} \), a graph \( G \) is said to have a \emph{locally bijective homomorphism} to \( H \) (also called an \( H \)-cover) if the vertex set of \( G \) can be partitioned into \( r \) sets \( U_1,U_2,\dots,U_r \) such that the following hold: (i)~\( u_i \) is adjacent to \( u_j \) in \( H \) implies that each vertex in \( U_i \) is adjacent to exactly one vertex in \( U_j \) in \( G \), and (ii)~\( u_i \) is not adjacent to \( u_j \) in \( H \) implies that no vertex in \( U_i \) is adjacent to any vertex in \( U_j \) in \( G \) \cite{fiala_kratochvil}. 
In other words, \( G \) has a \pagetarget{def:H cover}{locally bijective homomorphism} to \( H \) if we can label the vertices of \( G \) by names of vertices in \( H \) such that the following hold: (i)~each vertex of \( G \) labelled \( u \) (where \( u\in V(H) \)) is called a copy of \( u \) in \( G \), and (ii)~for each \( u\in V(H) \) and each copy \( u^{(s)} \) of \( u \) in \( G \), \( \deg_G(u^{(s)})=\deg_H(u) \) and neighbours of \( u^{(s)} \) in \( G \) are exactly copies of neighbours of \( u \) in \( H \). 
Clearly, a graph \( G \) is in \( \mathscr{G}_{2p} \) if and only if \( G \) has a locally bijective homomorphism to \( G_{2p} \) (label members of \( V_{i}^{j} \) by \( (i,j) \)). 
It is known that for every (simple) \( d \)-regular graph \( H \) with \( d\geq 3 \), it is NP-complete to test whether an input graph \( G \) has a locally bijective homomorphism to \( H \)  \cite[Theorem 20]{fiala_kratochvil}. 
Therefore, we have the following result by using \( H=G_{2p} \). 
\begin{observation}\label{obs:scriptG_2p npc}
For all \( p\geq 2 \), it is NP-complete to test whether a \( 2p \)-regular graph belongs to \(  \mathscr{G}_{2p} \).
\qed
\end{observation}

This motivates the following problem. 
\begin{problem}
For \( p\geq 2 \), is it NP-complete to test whether a \( 2p \)-regular graph is \( (p+2) \)-star colourable?
\end{problem}
%
The special case \( p=2 \) of Observation~\ref{obs:scriptG_2p npc} says that it is NP-complete to test whether a 4-regular graph is in \( \mathscr{G}_4 \). 
The following is a decomposition result for members of \( \mathscr{G}_4 \) (see supplementary material for proof). 
\begin{observation}\label{obs:graphs is G4 allow C_3q decomposition}
Every graph \( G \in \mathscr{G}_4 \) can be decomposed into cycles of length divisible by three. 
\end{observation}

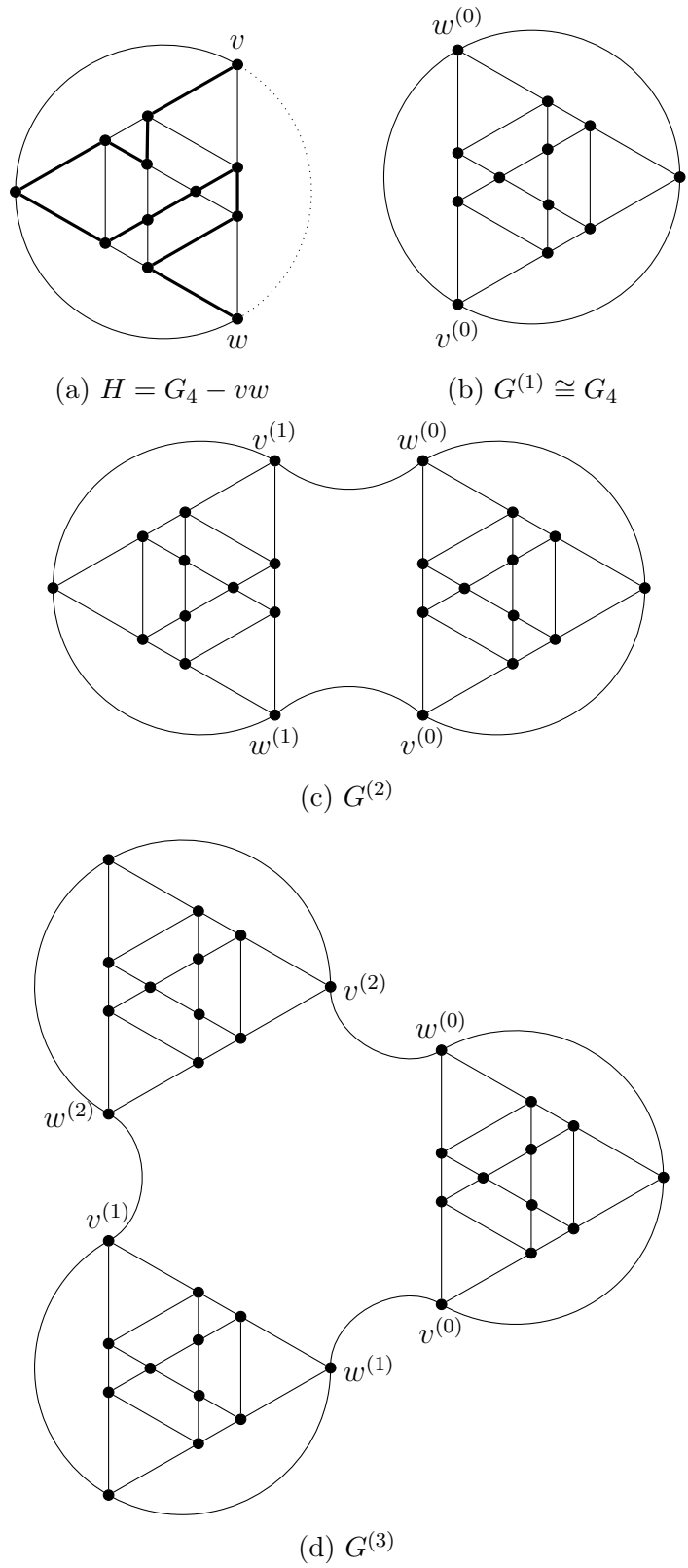
\begin{figure}[hbtp]
\centering
\begin{subfigure}[b]{0.3\textwidth}
\centering
\begin{tikzpicture}
\path (60:2) node(a)[bigdot][label=above:\( v \)]{};
\path (180:2) node(u)[bigdot]{};
\path (-60:2) node(w)[bigdot][label=below:\( w \)]{};
\draw (a)--(u)--(w)--(a);
\draw (a) arc (60:300:2);
\draw[dotted] (w) arc (-60:60:2);
\path (a)--(u) node(v)[pos=0.4][bigdot]{} node(u-a)[pos=0.6][bigdot]{};
\path (u)--(w) node(u-w)[pos=0.4][bigdot]{} node(w-u)[pos=0.6][bigdot]{};
\path (w)--(a) node(w-a)[pos=0.4][bigdot]{} node(x)[pos=0.6][bigdot]{};
\draw (u-w)--(u-a)--node(m1)[pos=0.3][bigdot]{} (w-a)--(w-u)--node[pos=0.3][bigdot]{} (v)--(x)--node[pos=0.3][bigdot]{} (u-w);

\draw [very thick] (a)--(v)--(m1)--(u-a)--(u)--(u-w)--(x)--(w-a)--(w-u)--(w);
\end{tikzpicture}
\caption{\( H=G_4-vw \)}
\label{fig:G_4 in seq}
\end{subfigure}%
\hspace{0.25cm}
\begin{subfigure}[b]{0.3\textwidth}
\centering
\begin{tikzpicture}
\path (-120:2) node(a)[bigdot][label=below:\( v^{(0)} \)]{};
\path (0:2) node(u)[bigdot]{};
\path (120:2) node(w)[bigdot][label=above:\( w^{(0)} \)]{};
\draw (a)--(u)--(w)--(a);
\draw (0,0) circle (2);
\path (a)--(u) node(v)[pos=0.4][bigdot]{} node(u-a)[pos=0.6][bigdot]{};
\path (u)--(w) node(u-w)[pos=0.4][bigdot]{} node(w-u)[pos=0.6][bigdot]{};
\path (w)--(a) node(w-a)[pos=0.4][bigdot]{} node(x)[pos=0.6][bigdot]{};
\draw (u-w)--(u-a)--node[pos=0.3][bigdot]{} (w-a)--(w-u)--node[pos=0.3][bigdot]{} (v)--(x)--node[pos=0.3][bigdot]{} (u-w);
\end{tikzpicture}
\caption{\( G^{(1)}\cong G_4 \)}
\label{fig:G^(1) in seq}
\end{subfigure}%

\begin{subfigure}[b]{0.65\textwidth}
\centering
\begin{tikzpicture}
\path[petal] (0:0)--(0:2) coordinate(p1);
\path[petal] (0:0)--(180:2) coordinate(p2);
\path (p1)+(-120:2) node(v1) [inner sep=0pt][label=below:\( v^{(0)} \)]{};
\path (p1)+( 120:2) node(w1) [inner sep=0pt][label=above:\( w^{(0)} \)]{};
\path (p2)+( 60:2) node(v2) [inner sep=0pt][label=above:\( v^{(1)} \)]{};
\path (p2)+(-60:2) node(w2) [inner sep=0pt][label=below:\( w^{(1)} \)]{};
\draw (v1) to[bend right=40] (w2);
\draw (v2) to[bend right=40] (w1);
\end{tikzpicture}
\caption{\( G^{(2)} \)}
\label{fig:G^(2) in seq}
\end{subfigure}%

\begin{subfigure}[b]{0.6\textwidth}
\centering
\begin{tikzpicture}
\path[petal] (0:0)--(0:3) coordinate(p1);
\path[petal] (0:0)--(120:3) coordinate(p2);
\path[petal] (0:0)--(240:3) coordinate(p3);
\path (p1)+(-120:2) node(v1) [inner sep=0pt][label=below:\( v^{(0)} \)]{};
\path (p1)+( 120:2) node(w1) [inner sep=0pt][label=above:\( w^{(0)} \)]{};
\path (p2)+(0:2) node(v3) [inner sep=0pt][label=right:\( v^{(2)} \)]{};
\path (p2)+(-120:2) node(w3) [inner sep=0pt][label=left:\( w^{(2)} \)]{};
\path (p3)+(120:2) node(v2) [inner sep=0pt][label=above:\( v^{(1)} \)]{};
\path (p3)+(0:2) node(w2) [inner sep=0pt][label=right:\( w^{(1)} \)]{};
\draw (v1) to[bend right=60] (w2);
\draw (v2) to[bend right=60] (w3);
\draw (v3) to[bend right=60] (w1);
\end{tikzpicture}
\caption{\( G^{(3)} \)}
\label{fig:G^(3) in seq}
\end{subfigure}%
\caption{Graphs \( G^{(1)} \), \( G^{(2)} \) and \( G^{(3)} \) in Theorem~\ref{thm:exists for each n=12t}.}
\label{fig:planar 4-regular 4-star colourable seq}
\end{figure}

Theorem~\ref{thm:chi_s attained} showed that for all \( p\geq 2 \), \( G_{2p} \) is a \( 2p \)-regular \( (p+2) \)-star colourable graph. 
The graph \( G_4 \) is planar as shown in Figure~\ref{fig:G4} whereas \( G_{2p} \) is non-planar for \( p>2 \) (because planar graphs are 5-degenerate). 
Next, we construct infinitely many planar 4-regular connected graphs that are 4-star colourable (recall that a 4-regular graph \( G \) is 4-star colourable only if \( n=|V(G)| \) is divisible by twelve).

\begin{theorem}\label{thm:exists for each n=12t}
For every integer \( n \) divisible by twelve, there exists a planar 4-regular Hamiltonian graph on \( n \) vertices which is 4-star colourable. 
\end{theorem}
\begin{proof}
For every positive integer \( t \), we construct a planar 4-regular Hamiltonian graph \( G^{(t)} \) on \( 12t \) vertices which is 4-star colourable. 
Recall that \( G_4 \) is Hamiltonian by Theorem~\ref{thm:intro G2p}. 
Choose an edge \( vw \) of \( G_4 \) which is part of a Hamiltonian cycle of \( G_4 \). 
Choose a plane drawing of \( G_4 \) such that edge \( vw \) appear in the outer face. 
The graph \( G^{(t)} \) is made of \( t \) copies \( H^{(0)},H^{(1)},\dots,H^{(t-1)} \) of \( H\coloneqq G_4-vw \) and edges between them in the cyclic order. 
For each vertex \( u \) of \( G_4 \), let us denote the copy of \( u \) in \( H^{(s)} \) by \( u^{(s)} \) for all \( s\in \mathbb{Z}_t \). 
For each \( s\in\mathbb{Z}_t \), add an edge \( v^{(s)}w^{(s+1)} \) where superscript \( (s+1) \) is modulo \( t \). 
Examples are exhibited in Figure~\ref{fig:planar 4-regular 4-star colourable seq}.

Since \( G^{(t)} \) is composed of \( t \) copies of the planar graph \( H \) and edges between those copies in the cyclic order, \( G^{(t)} \) is planar (where \( t\in\mathbb{N} \)). 
Next, we show that \( G^{(t)} \) is Hamiltonian for every \( t\in\mathbb{N} \). 
Since the edge \( vw \) of \( G_4 \) is part of a Hamiltonian cycle of \( G_4 \), there is a Hamiltonian path from \( w \) to \( v \) in the graph \( H=G_4-vw \) (one such path is highlighted in Figure~\ref{fig:G_4 in seq}). 
So, for each \( s\in\mathbb{Z}_t \), the graph \( H^{(s)} \) has a Hamiltonian path from \( w^{(s)} \) to \( v^{(s)} \). 
For \( s\in\mathbb{Z}_t \), let \( w^{(s)},P^{(s)},v^{(s)} \) denote one such Hamiltonian path in \( H^{(s)} \) (where \( P^{(s)} \) stands for a segment of the path). 
Then, \( (w^{(0)},P^{(0)},v^{(0)},w^{(1)},P^{(1)},v^{(1)},\dots,w^{(t-1)},P^{(t-1)},v^{(t-1)}) \) is a Hamiltonian cycle in \( G^{(t)} \). 
So, \( G^{(t)} \) is Hamiltonian for every \( t\in\mathbb{N} \). 
Therefore, for each \( t\in \mathbb{N} \), \( G^{(t)} \) is a planar 4-regular Hamiltonian graph on \( 12t \) vertices. 

Finally, we show that \( G^{(t)} \) is 4-star colourable for each \( t\in \mathbb{N} \). 
To prove this, it suffices to show that \( G^{(t)} \) is in \( \mathscr{G}_4 \). 
Note that for each \( u\in V(G_4) \) and each copy \( u^{(s)} \) of \( u \) in \( G^{(t)} \) (where \( 0\leq s\leq t-1 \)), \( \deg_{G^{(t)}}(u^{(s)})=\deg_{G_4}(u) \) and the neighbours of \( u^{(s)} \) in \( G^{(t)} \) are exactly copies of neighbours of \( u \) in \( G_4 \). 
Therefore, labelling each vertex \( u^{(s)} \) of \( G^{(t)} \) by \( u \) for all \( u\in V(G_4) \) and \( 0\leq s\leq t-1 \) gives a locally bijective homomorphism from \( G^{(t)} \) to \( G_4 \) (see the \hyperlink{def:H cover}{alternate definition} of locally bijective homomorphism on Page~\pageref{def:H cover}). 
Since there is a locally bijective homomorphism from \( G^{(t)} \) to \( G_4 \), \( G^{(t)} \) is in \( \mathscr{G}_4 \). 
In particular, \( G^{(t)} \) is 4-star colourable.
\end{proof}
\begin{theorem}\label{thm:exists for each n=(p+1)(p+2)t}
Let \( p\geq 2 \). For every integer \( n \) divisible by \( (p+1)(p+2) \), there exists a planar \( 2p \)-regular Hamiltonian graph on \( n \) vertices which is \( (p+2) \)-star colourable. 
\end{theorem}
\noindent The proof is similar to Theorem~\ref{thm:exists for each n=12t} and hence moved to supplementary material.\\

Next, we provide proof of Corollary~\ref{cor:diamond-free omega alpha etc}. 
\begin{customcor}{\ref{cor:diamond-free omega alpha etc}}
Let \( G \) be a \( 2p \)-regular \( (p+2) \)-star colourable graph where \( p\geq 2 \). Then, the following hold: \( (i) \)~\( G \) is \( (\textup{diamond},  K_4) \)-free, \( (ii) \)~\( \alpha(G)>n/4 \), \( (iii) \)~\( \chi(G)\leq 3\log_2(p+2) \), \( (iv) \)~\( G \) admits a \( P_4 \)-decomposition, and \( (v) \)~if \( G \) contains no asteroidal triple, then \( G \) is 3-colourable. 
\end{customcor}
\begin{proof}
Let \( f \) be a \( (p+2) \)-star colouring of \( G \). 
By Theorem~\ref{thm:2p-regular if p+2-star colourable}, every bicoloured component of \( G \) under \( f \) is isomorphic to \( K_{1,p} \). 
By Theorem~\ref{thm:2p-regular p+2-star colourable iff}, \( f \) induces a partition of the vertex set of \( G \) into sets \( V_{i}^{j} \) with indices \( i,j\in \{0,1,\dots,p+1\} \) and \( i\neq j \) such that the following hold for every pair of indices \( i \) and \( j \): (i)~each vertex in \( V_{i}^{j} \) has exactly \( p \) neighbours in \( \bigcup_{k\notin \{i,j\}} V_{j}^{k} \) and exactly one neighbour in \( V_{k}^{i} \) for each \( k\notin \{i,j\} \), and (ii)~\( V_i=\bigcup_{j\neq i} V_{i}^{j} \) for every colour \( i \) (where \( V_i=f^{-1}(i) \)).\\

\noindent \textbf{Claim 1:} If \( u\in V_i \), \( v\in V_j \) and \( uv \) is an edge in \( G \), then either \( u\in V_{i}^{j} \) or \( v\in V_{j}^{i} \).\\[5pt]
Suppose that \( u\in V_i \), \( v\in V_j \) and \( uv \) is an edge in \( G \). 
By Theorem~\ref{thm:2p-regular if p+2-star colourable}, the component of \( G[V_i\cup V_j] \) containing edge \( uv \) is a star \( H\cong K_{1,p} \). 
So, either \( u \) or \( v \) is the centre of \( H \). 
If \( u \) is the centre of \( H \), then \( u \) has \( p \) neighbours in \( V_j \) and thus \( u\in V_{i}^{j} \). 
If \( v \) is the centre of \( H \), then \( v \) has \( p \) neighbours in \( V_i \) and thus \( v\in V_{j}^{i} \). 
This proves Claim~1.\\

To prove that \( G \) is (diamond, \( K_4 \))-free, it suffices to show that diamond is not a subgraph of \( G \). 
On the contrary, assume that \( G \) contains diamond as a subgraph; that is, there exist vertices \( a,b,x,y \) of \( G \) such that \( ax, ay, bx,by \) and \( xy \) are edges in \( G \). 
Without loss of generality, assume that \( x\in V_0 \) and \( y\in V_1 \). 
Since \( xy \) is an edge, either \( x\in V_{0}^{1} \) or \( y\in V_{1}^{0} \) by Claim~1. 
Without loss of generality, assume that \( x\in V_{0}^{1} \). 
Due to Claim~5 of Theorem~\ref{thm:2p-regular if p+2-star colourable}, \( x \) has exactly \( p \) neighbours in \( V_1 \) and exactly one neighbour in each of the colour classes \( V_2,V_3,\dots,V_{p+1} \). 
In particular, the vertices \( a \) and \( b \) are in different colour classes (if not, \( a,b\in V_i \) for some \( i\in \{2,3,\dots,p+1\} \) and hence \( x \) has two neighbours in \( V_i \); a contradiction). 
Hence, we may assume without loss of generality that \( a\in V_2 \) and \( b\in V_3 \). 
Since \( \{V_{i}^{j}\ :\ 0\leq i\leq p+1,\ \allowbreak 0\leq j\leq p+1, \allowbreak \text{ and } i\neq j\} \) is a partition of \( V(G) \), no vertex of \( G \) is in two distinct sets \( V_{i}^{j} \) and \( V_{k}^{\ell} \) (where \( k\neq i \) or \( j\neq \ell \)). 
Since \( ax \) is an edge, either \( a\in V_{2}^{0} \) or \( x\in V_{0}^{2} \) by Claim~1. 
Since \( x\in V_{0}^{1} \), \( x\notin V_{0}^{2} \) and hence \( a\in V_{2}^{0} \). 
Since \( bx \) is an edge, either \( b\in V_{3}^{0} \) or \( x\in V_{0}^{3} \) by Claim~1. 
Since \( x\in V_{0}^{1} \), \( x\notin V_{0}^{3} \) and hence \( b\in V_{3}^{0} \). 
Since \( ay \) is an edge, either \( a\in V_{2}^{1} \) or \( y\in V_{1}^{2} \) by Claim~1. 
Since \( a\in V_{2}^{0} \), \( a\notin V_{2}^{1} \) and hence \( y\in V_{1}^{2} \). 
Since \( by \) is an edge, either \( b\in V_{3}^{1} \) or \( y\in V_{1}^{3} \) by Claim~1. 
Since \( b\in V_{3}^{0} \), \( b\notin V_{3}^{1} \) and hence \( y\in V_{1}^{3} \). 
Since \( y\in V_{1}^{2} \) and \( y\in V_{1}^{3} \), we have a contradiction to \( V_{1}^{2}\cap V_{1}^{3}=\emptyset \). 
Hence, \( G \) does not contain diamond as a subgraph, and thus \( G \) is  \( (\text{diamond}, K_4) \)-free.
This proves~(i).\\

%
Next, we show that the independence number \( \alpha(G)>n/4 \). 
By Claim~1, if \( u\in V_0 \) is adjacent to \( v\in V_1 \), then either \( u\in V_{0}^{1} \) or \( v\in V_{1}^{0} \). 
Therefore, \( (V_0\cup V_1)\setminus (V_{0}^{1}\cup V_{1}^{0}) \) is an independent set in \( G \). 
For the same reason, \( \bigcup_{i=0}^{2}\bigcup_{j=3}^{p+1} V_{i}^{j} = (V_0\cup V_1\cup V_2)\setminus (V_{0}^{1}\cup V_{0}^{2}\cup V_{1}^{0}\cup V_{1}^{2}\cup V_{2}^{0}\cup V_{2}^{1}) \) is an independent set in \( G \). 
In general, for \( 0<t<p+1 \), the set \( I_t\coloneqq \bigcup_{i=0}^{t-1}\bigcup_{j=t}^{p+1} V_{i}^{j} \) 
is an independent set of \( G \) with cardinality \( \sum_{i=0}^{t-1}\sum_{j=t}^{p+1} |V_{i}^{j}| \). 
By Theorem~\ref{thm:2p-regular p+2-star colourable iff}, each set \( V_{i}^{j} \) has cardinality \( n/(p+1)(p+2) \) and thus \( |I_t| = \sum_{i=0}^{t-1}\sum_{j=t}^{p+1} \frac{n}{(p+1)(p+2)} = t(p+2-t)\frac{n}{(p+1)(p+2)} \) for \( 0<t<p+1 \). 

In particular, for \( t=\ceil[\big]{\frac{p+2}{2}} \), \( I_t \) is an independent set of size \( \floor[\big]{\frac{p+2}{2}}\ceil[\big]{\frac{p+2}{2}}\frac{n}{(p+1)(p+2)}>\frac{n}{4} \). 
Hence, \( \alpha(G)>n/4 \). 
This proves (ii).\\

Next, we show that the chromatic number of \( G \) is at most \( 3\log_2(p+2) \). 
From Theorem~\ref{thm:2p-regular p+2-star colourable iff}, we know that \( G \) admits a \( D_{(p+1)(p+2)} \)-partition where \( D_q \)-partition problem is \hyperlink{def:D_q-partition}{defined} on Page~\pageref{def:D_q-partition} and the matrix \( D_{(p+1)(p+2)} \) is \hyperlink{def:D_q}{defined} on Page~\pageref{def:D_q}. 
The next claim follows from the definition of the \( D_q \)-partition problem and the definition of the matrix \( D_{(p+1)(p+2)} \).
Note that  the definition of \( D_{(p+1)(p+2)} \) is valid for all \( p\geq 0 \).\\[5pt]  
\textbf{Claim 2:} A graph \( G^* \) admits a \( D_{(p+1)(p+2)} \)-partition if and only if the vertex set of \( G^* \) can be partitioned into sets \( V_{i}^{j} \) with indices \( i,j\in\{0,1,\dots,p+1\} \) and \( i\neq j \) such that for each pair \( i,j \) (with \( i\neq j \)) and each vertex \( v\in V_{i}^{j} \), \( v \) has exactly one neighbour in \( V_{k}^{i} \) for each \( k\notin\{i,j\} \), and every neighbour of \( v \) is either in \( V_{k}^{i} \) for some \( k\notin\{i,j\} \) or in \( V_{j}^{k} \) for some \( k\notin\{i,j\} \).\\[4pt]
Note that the notation \( k\notin\{i,j\} \) abbreviates \( k\in\{0,1,\dots,p+1\}\setminus\{i,j\} \).\\ 

We prove by induction on \( p \) that every graph \( G^* \) which admits a \( D_{(p+1)(p+2)} \)-partition satisfies \( \chi(G^*)\leq 3\floor{\log_2(p+2)} \). 
Note that \( G^* \) is not necessarily a regular graph.\\[5pt]
Base Case \( (p=0, 1) \): If \( p=0 \) and \( G^* \) admits a \( D_{(p+1)(p+2)} \)-partition, then \( V(G^*) \) can be partitioned into two independent sets \( V_{0}^{1} \) and \( V_{1}^{0} \) such that there is no edge between \( V_{0}^{1} \) and \( V_{1}^{0} \). 
That is, \( G^* \) is the complement of a complete graph. 
Since \( G^* \) is 1-colourable, \( \chi(G)\leq 3\floor{\log_2(p+2)}=3 \) is true. 

If \( p=1 \) and \( G^* \) admits a \( D_{(p+1)(p+2)} \)-partition, then \( V(G^*) \) can be partitioned into six sets \( V_{i}^{j} \) with indices \( i,j\in\{0,1,2\} \) and \( i\neq j \) such that the number of neighbours of a vertex \( v\in V_{i}^{j} \) in the set \( V_{k}^{\ell} \) is determined by the matrix \( D_6 \) given below.
\begin{equation*}
  D_{6}=
  \begin{blockarray}{lcc@{\hskip 6mm}cc@{\hskip 6mm}cc}
    \begin{block}{lcc@{\hskip 6mm}cc@{\hskip 6mm}cc}
              &  V_{0}^{1}       & V_{0}^{2}         & V_{1}^{0}        & V_{1}^{2}        & V_{2}^{0}        &  V_{2}^{1}\\
    \end{block}
    \begin{block}{l[cc@{\hskip 6mm}cc@{\hskip 6mm}cc]}
      V_{0}^{1}  & \bm{\{0\}}    & \bm{\{0\}}     & \{0\}         & \mathbb{Z}_0^+& \{1\}           & \{0\}\topstrut \\
      V_{0}^{2}  & \bm{\{0\}}    & \bm{\{0\}}     & \{1\}         & \{0\}         & \{0\}           & \mathbb{Z}_0^+\\[2mm]
      V_{1}^{0}  & \{0\}         & \mathbb{Z}_0^+ & \bm{\{0\}}    & \bm{\{0\}}    & \{0\}           & \{1\}\\
      V_{1}^{2}  & \{1\}         & \{0\}          & \bm{\{0\}}    & \bm{\{0\}}    & \mathbb{Z}_0^+  & \{0\}\\[2mm]
      V_{2}^{0}  & \mathbb{Z}_0^+  & \{0\}        & \{0\}         & \{1\}         & \bm{\{0\}}      & \bm{\{0\}}\\
      V_{2}^{1}  & \{0\}         & \{1\}          & \mathbb{Z}_0^+& \{0\}         & \bm{\{0\}}      & \bm{\{0\}}\botstrut \\
    \end{block}
  \end{blockarray}
\end{equation*}

Clearly, \( V_{0}^{1}\cup V_{0}^{2} \), \( V_{1}^{0}\cup V_{1}^{2} \) and \( V_{2}^{0}\cup V_{2}^{1} \) are independent sets in \( G^* \). 
Since these three independent sets form a partition of \( V(G^*) \), \( G^* \) is 3-colourable. 
So, the inequality \( \chi(G^*)\leq 3\floor{\log_2(p+2)}=3 \) is true.\\

\noindent Induction Step \( (p\geq 2) \): Suppose that \( G^* \) admits a \( D_{(p+1)(p+2)} \)-partition. 
By Claim~2, the vertex set of \( G^* \) can be partitioned into sets \( V_{i}^{j} \) with indices \( i,j\in\{0,1,\dots,p+1\} \) and \( i\neq j \) such that for each pair \( i,j \) (with \( i\neq j \)) and each vertex \( v\in V_{i}^{j} \), the following hold: (i)~\( v \) has exactly one neighbour in \( V_{k}^{i} \) for each \( k\in\{0,1,\dots,p+1\}\setminus\{i,j\} \), and (ii)~for every neighbour \( w \) of \( v \), either \( w\in V_{k}^{i} \) for some \( k\in\{0,1,\dots,p+1\}\setminus \{i,j\} \) or \( w\in V_{j}^{k} \) for some \( k\in\{0,1,\dots,p+1\}\setminus \{i,j\} \).\\

\noindent \textbf{Claim~3:} If \( u\in \bigcup_{q\neq i} V_{i}^{q} \), \( v\in \bigcup_{q\neq j} V_{j}^{q} \), and \( uv \) is an edge in \( G^* \), then \( u\in V_{i}^{j} \) or \( v\in V_{j}^{i} \).\\[5pt]
On the contrary, assume that \( u\in V_{i}^{r} \), \( v\in V_{j}^{s} \), and \( uv \) is an edge in \( G^* \) where \( r,s\notin\{i,j\} \). 
Due to Claim~2, every neighbour of \( u \) is either in \( V_{k}^{i} \) or in \( V_{r}^{k} \) for some \( k\in\{0,1,\dots,p+1\}\setminus\{i,r\} \). 
Since \( v\in V_{j}^{s} \) is a neighbour of \( u \) and \( s\neq i \), the only possibility is \( V_{j}^{s}=V_{r}^{k} \); that is, \( r=j \) and \( k=s \). This is a contradiction because \( r\notin \{i,j\} \). 
This proves Claim~3.\\

Due to Claim~3, \( (\,\bigcup_{k\notin\{0,1\}} V_{0}^{k})\cup (\,\bigcup_{k\notin\{0,1\}} V_{1}^{k}) \) is an independent set in \( G^* \). 
Similarly, as in proof of \( \alpha(G)>n/4 \), for each \( t\in\{1,2,\dots,p\} \), the set \( I_t\coloneqq \bigcup_{i=0}^{t-1}\bigcup_{j=t}^{p+1} V_{i}^{j} \) is an independent set in \( G^* \).

We partition the vertex set of \( G^* \) into five sets \( A,B,C,W_1 \) and \( W_2 \) defined as\\
%
\(
\displaystyle
A=\bigcup_{i=0}^{\left\lfloor\!{\nicefrac{p}{2}}\!\right\rfloor}\bigcup_{j=\left\lfloor\!{\nicefrac{p}{2}}\!\right\rfloor+1}^{p+1}V_{i}^{j}, \quad 
B=\bigcup_{i=\left\lceil\!{\nicefrac{p}{2}}\!\right\rceil+1}^{p+1}\bigcup_{j=0}^{\left\lceil\!{\nicefrac{p}{2}}\!\right\rceil}V_{i}^{j}, \quad 
W_1=\bigcup_{i=0}^{\left\lfloor\!{\nicefrac{p}{2}}\!\right\rfloor}\bigcup_{\substack{j=0\\(j\neq i)}}^{\left\lfloor\!{\nicefrac{p}{2}}\!\right\rfloor}V_{i}^{j}, \quad
W_2=\bigcup_{i=\left\lceil\!{\nicefrac{p}{2}}\!\right\rceil+1}^{p+1}\,\bigcup_{\substack{j=\left\lceil\!{\nicefrac{p}{2}}\!\right\rceil +1\\(j\neq i)}}^{p+1}V_{i}^{j},\\[3pt]
 \text{ and } 
C=\mathsmaller{\bigcup}\limits_{j\neq (p+1)/2} \, V_{(p+1)/2}^{j} \text{ if \( p \) is odd, and } C=\emptyset \text{ otherwise (see Figure~\ref{fig:A,B,C,W_1,W_2})}
\).
\begin{figure}[hbtp]
\centering
\begin{subfigure}[b]{\textwidth}
\centering
\begin{tikzpicture}[x=1.5cm,y=1.75cm]
\draw[dotted,step=1] (0,0) grid (5,-4);
\path (0.5,-0.5) node(01){\( V_{0}^{1} \)} ++(0,-1) node(02){\( V_{0}^{2} \)} ++(0,-1) node(03){\( V_{0}^{3} \)} ++(0,-1) node(04){\( V_{0}^{4} \)};
\path (01) ++(1,0) node(10){\( V_{1}^{0} \)} ++(0,-1) node(02){\( V_{1}^{2} \)} ++(0,-1) node(03){\( V_{1}^{3} \)} ++(0,-1) node(04){\( V_{1}^{4} \)};
\path (10) ++(1,0) node(20){\( V_{2}^{0} \)} ++(0,-1) node(02){\( V_{2}^{1} \)} ++(0,-1) node(03){\( V_{2}^{3} \)} ++(0,-1) node(04){\( V_{2}^{4} \)};
\path (20) ++(1,0) node(30){\( V_{3}^{0} \)} ++(0,-1) node(02){\( V_{3}^{1} \)} ++(0,-1) node(03){\( V_{3}^{2} \)} ++(0,-1) node(04){\( V_{3}^{4} \)};
\path (30) ++(1,0) node(40){\( V_{4}^{0} \)} ++(0,-1) node(02){\( V_{4}^{1} \)} ++(0,-1) node(03){\( V_{4}^{2} \)} ++(0,-1) node(04){\( V_{4}^{3} \)};

\node [fit={(0,-1) (2,-4)},draw,xshift=-5pt,yshift=-5pt][label=left:\( A \)]{};
\node [fit={(0,0) (2,-1)},draw,xshift=-5pt,yshift=5pt][label=left:\( W_1 \)]{};
\node [fit={(2.1,0.1) (2.9,-4.25)},inner sep=0pt,yshift=3pt,draw][label=\( C \)]{};
\node [fit={(3,0) (5,-3)},draw,xshift=5pt,yshift=5pt][label=right:\( B \)]{};
\node [fit={(3,-3) (5,-4)},draw,xshift=5pt,yshift=-5pt][label=right:\( W_2 \)]{};


\end{tikzpicture}
\caption{\( p=3 \)}
\vspace*{7pt}%
\end{subfigure}%

\begin{subfigure}[b]{\textwidth}
\centering
\begin{tikzpicture}[x=1.5cm,y=1.75cm]
\draw[dotted,step=1] (0,0) grid (6,-5);
\path (0.5,-0.5) node(01){\( V_{0}^{1} \)} ++(0,-1) node(02){\( V_{0}^{2} \)} ++(0,-1) node(03){\( V_{0}^{3} \)} ++(0,-1) node(04){\( V_{0}^{4} \)} ++(0,-1) node(05){\( V_{0}^{5} \)};
\path (01) ++(1,0) node(10){\( V_{1}^{0} \)} ++(0,-1) node(12){\( V_{1}^{2} \)} ++(0,-1) node(13){\( V_{1}^{3} \)} ++(0,-1) node(14){\( V_{1}^{4} \)} ++(0,-1) node(15){\( V_{1}^{5} \)};
\path (10) ++(1,0) node(20){\( V_{2}^{0} \)} ++(0,-1) node(21){\( V_{2}^{1} \)} ++(0,-1) node(23){\( V_{2}^{3} \)} ++(0,-1) node(24){\( V_{2}^{4} \)} ++(0,-1) node(25){\( V_{2}^{5} \)};
\path (20) ++(1,0) node(30){\( V_{3}^{0} \)} ++(0,-1) node(31){\( V_{3}^{1} \)} ++(0,-1) node(32){\( V_{3}^{2} \)} ++(0,-1) node(34){\( V_{3}^{4} \)} ++(0,-1) node(35){\( V_{3}^{5} \)};
\path (30) ++(1,0) node(40){\( V_{4}^{0} \)} ++(0,-1) node(41){\( V_{4}^{1} \)} ++(0,-1) node(42){\( V_{4}^{2} \)} ++(0,-1) node(43){\( V_{4}^{3} \)} ++(0,-1) node(45){\( V_{4}^{5} \)};
\path (40) ++(1,0) node(50){\( V_{5}^{0} \)} ++(0,-1) node(51){\( V_{5}^{1} \)} ++(0,-1) node(52){\( V_{5}^{2} \)} ++(0,-1) node(53){\( V_{5}^{3} \)} ++(0,-1) node(54){\( V_{5}^{4} \)};

\node [fit={(0,-2) (3,-5)},draw,xshift=-5pt,yshift=-5pt][label=left:\( A \)]{};
\node [fit={(0,0) (3,-2)},draw,xshift=-5pt,yshift=5pt][label=left:\( W_1 \)]{};
\node [fit={(3,0) (6,-3)},draw,xshift=5pt,yshift=5pt][label=right:\( B \)]{};
\node [fit={(3,-3) (6,-5)},draw,xshift=5pt,yshift=-5pt][label=right:\( W_2 \)]{};
\end{tikzpicture}
\caption{\( p=4 \) (here, \( C=\emptyset \))}
\label{fig: partition when p=4}
\end{subfigure}%
\caption{Examples of partitioning \( V(G^*) \) into sets \( A,B,C,W_1 \) and \( W_2 \).}
\label{fig:A,B,C,W_1,W_2}
\end{figure}
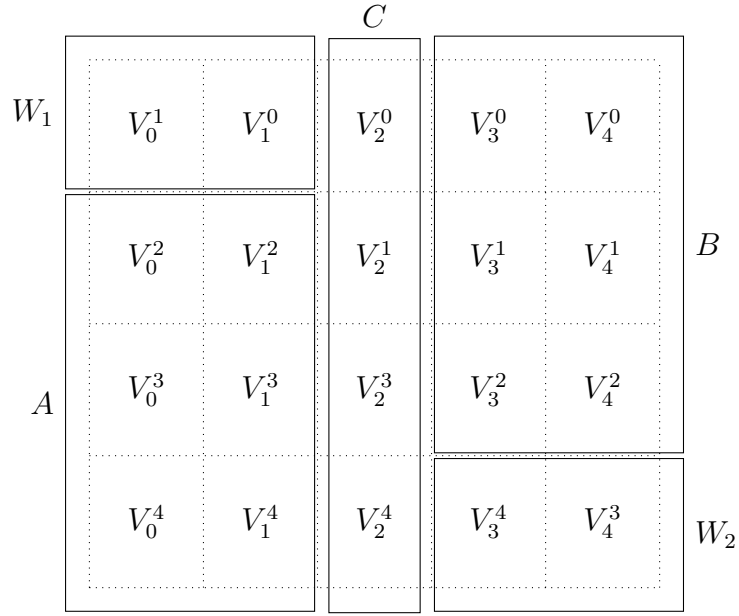
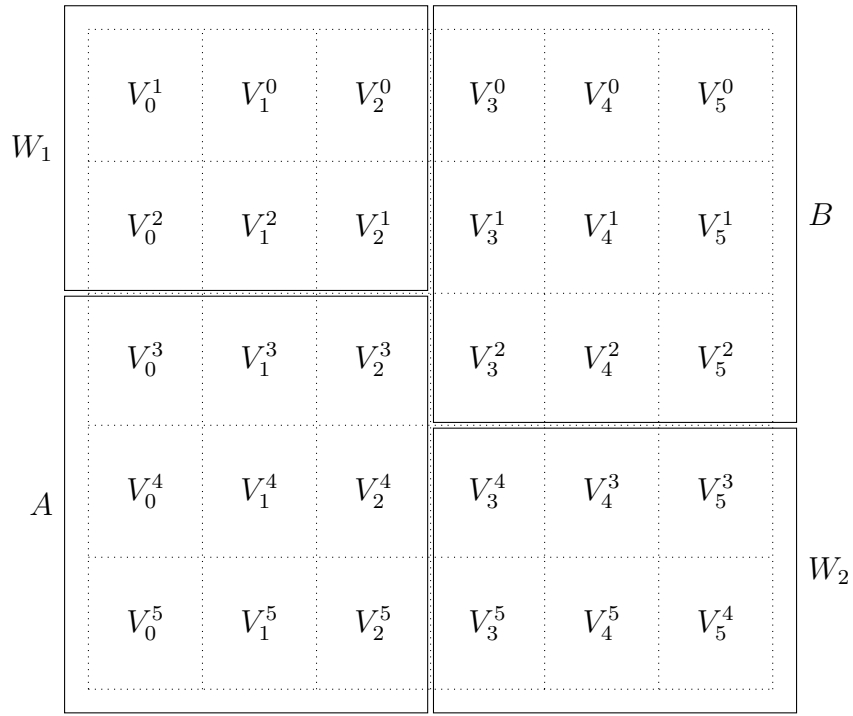


Note that \( A \) is precisely the independent set \( I_t \) for \( t=\floor{p/2}+1 \). 
So, \( A \) is an independent set in \( G^* \). 
By similar arguments, \( B \) is an independent set in \( G^* \) (due to Claim~3). 
Since \( C \) is either the empty set or \( C=\bigcup_{j\neq (p+1)/2}\, V_{(p+1)/2}^{j} \), \( C \) is also an independent set. 
Since \( A,B \) and \( C \) are independent sets, \( \chi(G^*)\leq \chi(G^*[W_1\cup W_2])+3 \). 
We know that for each pair \( i,j \) (with \( i\neq j \)) and for each vertex \( v \) in \( V_{i}^{j} \) and each neighbour \( w \) of \( v \), either \( w\in V_{k}^{i} \) for some \( k\notin\{i,j\} \) or \( w\in V_{j}^{k} \) for some \( k\notin\{i,j\} \). 
In particular, if \( k\notin\{i,j\} \) and  \( \ell\notin\{i,j\} \), then no vertex in \( V_{i}^{j} \) has a neighbour in \( V_{k}^{\ell} \). 
As a result, there is no edge between \( W_1 \) and \( W_2 \) in \( G^* \). 
Hence, we have \( \chi(G^*[W_1\cup W_2])=\max \{\chi(G^*[W_1]), \chi(G^*[W_2])\} \).\\

\noindent \textbf{Claim~4: } \( G^*[W_1] \) admits a \( D_{q_{\vphantom{|}0}} \)-partition with \( q_{0}=\floor{\frac{p}{2}}(\floor{\frac{p}{2}}+1) \).\\[5pt]
We know that the following hold for each \( i,j\in\{0,1,\dots,p+1\} \) with \( i\neq j \) and each vertex \( v\in V_{i}^{j} \):\\
(i)~\( v \)  has exactly one neighbour in \( V_{k}^{i} \) for each \( k\in\{0,1,\dots,p+1\}\setminus \{i,j\} \), and\\
(ii)~for each neighbour \( w \) of \( v \), either \( w\in V_{k}^{i} \) for some \( k\in\{0,1,\dots,p+1\}\setminus \{i,j\} \) or \( w\in V_{j}^{k} \) for some \( k\in\{0,1,\dots,p+1\}\setminus \{i,j\} \). 

Note that \( \{V_{i}^{j}\ :\ 0\leq i\leq \floor{p/2},\ 0\leq j\leq \floor{p/2}, \text{ and } i\neq j\} \) is a partition of the vertex set of \( G^*[W_1] \). 
By the property of the sets \( V_{i}^{j} \) stated in the previous paragraph, the following hold in the graph \( G^*[W_1] \) for each \( i,j\in \{0,1,\dots,\floor{p/2}\} \) with \( i\neq j \) and each vertex \( v \) in \( V_{i}^{j} \): (i)~\( v \) has exactly one neighbour in \( V_{k}^{i} \) for each \( k\in\{0,1,\dots,\floor{p/2}\}\setminus\{i,j\} \), and (ii)~for every neighbour \( w \) of \( v \), either \( w\in V_{k}^{i} \) for some \( k\in\{0,1,\dots,\floor{p/2}\}\setminus\{i,j\} \) or \( w\in V_{j}^{k} \) for some \( k\in\{0,1,\dots,\floor{p/2}\}\setminus\{i,j\} \). 
Therefore, by Claim~2, \( G^*[W_1] \) admits a \( D_{q_{\vphantom{|}0}} \)-partition where \( q_{0}=\floor{\frac{p}{2}}(\floor{\frac{p}{2}}+1) \). 
This proves Claim~4.

Thanks to Claim~4, by induction hypothesis, \( \chi(G^*[W_1])\leq 3\log_2(\floor{\nicefrac{p}{2}}+1)=3\log_2(\floor{\nicefrac{p+2}{2}}) \). 
Similarly, \( G^*[W_2] \) admits a \( D_{q_{\vphantom{|}0}} \)-partition and thus \( \chi(G^*[W_2])\leq 3\log_2(\floor{\nicefrac{p+2}{2}}) \). 
So, \( \chi(G^*[W_1\cup W_2])=\allowbreak \max \{\chi(G^*[W_1]), \chi(G^*[W_2])\}\leq 3\log_2(\floor{\nicefrac{p+2}{2}}) \). 
Therefore, \( \chi(G^*)\leq 3\log_2(\floor{\nicefrac{p+2}{2}})+3\leq 3(\log_2(\nicefrac{p+2}{2})+1)=\allowbreak 3\log_2(p+2) \). 
Thus, by mathematical induction, for all \( p\geq 0 \), \( \chi(G^*)\leq 3\log_2(p+2) \) if \( G^* \) admits a \( D_{(p+1)(p+2)} \)-partition.\\ 

Since \( G \) is \( 2p \)-regular and \( (p+2) \)-star colourable, \( G \) admits a \( D_{(p+1)(p+2)} \)-partition by Theorem~\ref{thm:2p-regular p+2-star colourable iff}. 
Hence, \( \chi(G)\leq 3\log_2(p+2) \). 
This proves (iii).\\


Next, we show that \( G \) admits a \( P_4 \)-decomposition. 
Oksimets~\cite{oksimets} proved that a \( 2p \)-regular graph \( G^* \) admits a \( P_4 \)-decomposition if and only if \( |E(G^*)| \) is divisible by three. 
By Theorem~\ref{thm:2p-regular p+2-star colourable iff}, the number of vertices in \( G \) is divisible by \( (p+1)(p+2) \). 
Hence, \( |E(G)|=np \) is divisible by \( p(p+1)(p+2) \) which is in turn divisible by three. 
Hence, \( G \) admits a \( P_4 \)-decomposition. This proves (iv).\\

Finally, we prove that \( G \) is 3-colourable if \( G \) contains no asteroidal triple. 
Stacho established that every (diamond, \( K_4 \))-free graph without asteroidal triple is 3-colourable \cite[Theorem~1.3]{stacho}. 
From (i), we know that \( G \) is \( (\text{diamond}, K_4) \)-free. 
This proves (v). 
\end{proof}

\section{Star Colourings and Orientations}\label{sec:orientations}
Let \( G \) be an (undirected) graph. 
Recall that an \emph{orientation} of \( G \) is a directed graph obtained from \( G \) by assigning a direction to each edge of \( G \). 
Star colouring of \( G \) is known to be linked with orientations of \( G \). 
Albertson et al.~\cite{albertson} proved that a colouring \( f \) of \( G \) is a star colouring if and only if there exists an orientation \( \overrightarrow{G} \) of \( G \) such that edges in each bicoloured 3-vertex path in \( \overrightarrow{G} \) are oriented towards the middle vertex. 
Nešetřil and Mendez~\cite{nesetril_mendez2003} characterized the star chromatic number of \( G \) in terms of orientations of \( G \) (see the next paragraph for details). 
Motivated by Claim~5 in proof of Theorem~\ref{thm:2p-regular if p+2-star colourable}, we introduce a new type of orientation named colourful Eulerian orientation and reveal its connection to star colouring in even-degree regular graphs.

For each orientation \( \overrightarrow{G_i} \) of \( G \), let us define \( G_i^+ \) as the undirected graph with \( V(G_i^+)=V(G) \) and \( E(G_i^+)=E(G)\cup E_i^+ \) where \( E_i^+=\{uv\ :\ u,v\in V(G), \text{ and } \exists w\in N(u)\cap N(v) \text{ such that } (u,w)\notin E(\overrightarrow{G_i}) \text{ or } (v,w)\notin E(\overrightarrow{G_i}) \} \). 
In other words, \( G_i^+ \) is obtained from \( G \) by adding edges \( uv \) whenever \( u \) and \( v \) have a common neighbour \( w \) such that at least one edge in path \( u,w,v \) is oriented away from the middle vertex \( w \). 
For every graph \( G \), the star chromatic number \( \chi_s(G)=\min_{i\in I} \chi(G_i^+) \) where \( I \) is an index set and \( \{\overrightarrow{G}_i\ :\ i\in I\} \) is the set of all orientations of \( G \) \cite[Corollary~3]{nesetril_mendez2003} (a different notation is used in \cite{nesetril_mendez2003}).

Let us see a few definitions first. 
Recall that a colouring \( f \) of \( G \) is a star colouring if and only if there exists an orientation \( \overrightarrow{G} \) of \( G \) such that edges in each bicoloured 3-vertex path in \( \overrightarrow{G} \) are oriented towards the middle vertex \cite{albertson}. 
If \( f \) is a star colouring of \( G \), an \emph{in-orientation of \( G \) induced by \( f \)} is an orientation \( \overrightarrow{G} \) of \( G \) obtained by orienting edges in each bicoloured 3-vertex path in \( G \) towards the middle vertex, and then orienting the remaining edges arbitrarily. 
The notion of in-orientation \( \overrightarrow{G} \) induced by \( f \) is the same as the notion of `colored in-orientation' \( (f,\overrightarrow{G}) \) in \cite{dvorak_esperet}. 
If \( f \) is a star colouring of \( G \) and no bicoloured component of \( G \) under \( f \) is isomorphic to \( K_{1,1} \), then the in-orientation of \( G \) induced by \( f \) is unique (because every edge in \( G \) is part of a bicoloured 3-vertex path). 

An orientation \( \overrightarrow{G} \) is an \emph{Eulerian orientation} if the number of in-neighbours of \( v \) equals the number of out-neighbours of \( v \) for every vertex \( v \) of \( \overrightarrow{G} \) (i.e., \( \text{in-degree}(v)=\text{out-degree}(v)\ \forall v\in V(\overrightarrow{G}) \)) \cite{levit}. 
If \( G \) admits an Eulerian orientation, then clearly every vertex of \( G \) is of even degree. 
Conversely, if every vertex of \( G \) is of even degree, then \( G \) admits an Eulerian orientation \cite{levit}. 
Connected graphs \( G \) such that every vertex of \( G \) is of even degree is called an \emph{Eulerian graph} because \( G \) admits an Eulerian tour. 

Let \( G \) be an Eulerian graph. 
Then, \( G \) admits Eulerian orientations. 
We say that an Eulerian orientation \( \overrightarrow{G} \) of \( G \) is a \emph{\( q \)-colourful Eulerian orientation} if there exists a \( q \)-colouring \( f \) of \( G \) such that the following hold under \( f \) for every vertex \( v \) of \( \overrightarrow{G} \):\\
(i)~in-neighbours of \( v \) have the same colour, say colour \( c_v \),\\
(ii)~no out-neighbour of \( v \) has colour \( c_v \), and\\
(iii)~out-neighbours of \( v \) have pairwise distinct colours.\\
An Eulerian orientation \( \overrightarrow{G} \) of \( G \) is said to be a \emph{colourful Eulerian orientation} if \( \overrightarrow{G} \) is a \( q \)-colourful Eulerian orientation for some integer \( q \).

We remark that there exist Eulerian graphs which do not admit a colourful Eulerian orientation (see Theorem~\ref{thm:graphs without CEO}). 
The next theorem shows for all \( p\geq 2 \), a \( 2p \)-regular graph \( G \) is \( (p+2) \)-star colourable if and only if \( G \) admits a \( (p+2) \)-colourful Eulerian orientation. 

\begin{theorem}\label{thm:chi_s bound iff p+2 CEO}
Let \( p\geq 2 \), and let \( G \) be a \( 2p \)-regular graph. 
Then, \( G \) is \( (p+2) \)-star colourable if and only if \( G \) admits a \( (p+2) \)-colourful Eulerian orientation. 
\end{theorem}
\begin{proof}
Suppose that \( G \) admits a \( (p+2) \)-star colouring \( f\colon V(G)\to\{0,1,\dots,p+1\} \). 
By Theorem~\ref{thm:2p-regular if p+2-star colourable}, every bicoloured component of \( G \) under \( f \) is isomorphic to \( K_{1,p} \). 
Hence, the in-orientation \( \overrightarrow{G} \) of \( G \) induced by \( f \) is unique. 
Also, for every bicoloured component \( H\cong K_{1,p} \) of \( G \), the edges in \( H \) are oriented by \( \overrightarrow{G} \) towards the centre of \( H \). 
We claim that \( \overrightarrow{G} \) is a \( (p+2) \)-colourful Eulerian orientation of \( G \) with \( f \) as the underlying \( (p+2) \)-colouring. 
To show this, it suffices to prove the following claim (recall that we denote the colour class \( f^{-1}(i) \) by \( V_i \) for every colour \( i \)).\\[5pt]
\textbf{Claim 1:} For each colour \( i\in\{0,1,\dots,p+1\} \) and each vertex \( v\in V_i \), all \( p \) in-neighbours of \( v \) are in some colour class \( V_j \), and every other colour class \( V_k \), \( k\notin\{i,j\} \), contains exactly one out-neighbour of~\( v \).\\[5pt]
We prove Claim~1 under the assumption \( i=0 \) (the proof is similar for other values of \( i \)). 
Let \( v\in V_0 \). 
By Claim~5 of Theorem~\ref{thm:2p-regular if p+2-star colourable}, \( v \) has exactly \( p \) neighbours in some colour class \( V_j \) and exactly one neighbour in every other colour class \( V_k \), \( k\notin\{i,j\} \) (see Figure~\ref{fig:re-neighbourhood of vertex in G}). 

\begin{figure}[hbt]
\centering
\begin{tikzpicture}
\path (0,0)  node(V0)[Vset][label=below:\( V_0 \)]{}++(2,0) node(V1)[Vset][label=below:\( V_1 \)]{}++(2,0) node(V2)[Vset][label=below:\( V_2 \)]{}++(1.5,0) node[font=\large][yshift=-0.15cm]{\dots}++(1.5,0) node(Vp+1)[Vset][label=below:\( V_{p+1} \)]{};
\path (V0)+(0,-0.85) node(v0)[dot][label=left:\( v \)]{};
\path (V1)++(0,0.75) node(v11)[dot][label=right:\( w_1 \)]{}++(0,-0.3) node{\vdots}++(0,-0.45) node(v1p)[dot][label=right:\( w_p \)]{};
\path (V2)+(0,-0.15) node(v2)[dot][label=\( x_1 \)]{};
\path (Vp+1)+(0,-0.15) node(vp+1)[dot][label=\( x_p \)]{};
\draw
(v0)--(v11)
(v0)--(v1p)
(v0)--(v2)
(v0)--(vp+1);
\end{tikzpicture}
\caption{Colours on the neighbourhood of an arbitrary vertex \( v\in V_0 \) (here, \( j=1 \)).}
\label{fig:re-neighbourhood of vertex in G}
\end{figure}
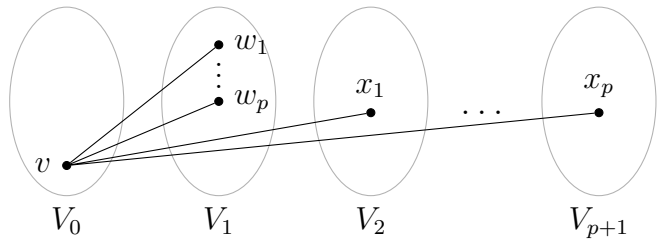

Without loss of generality, assume that \( j=1 \). 
Let \( w_1,\dots,w_p,x_1,\dots,x_p \) be the neighbours of \( v \) in \( G \) where \( w_1,\dots,w_p\in V_1 \) and \( x_r\in V_{r+1} \) for \( 1\leq r\leq p \). 
The subgraph of \( G \) induced by the set \( \{v,w_1,\dots,w_p\} \) is a bicoloured component \( H \) of \( G \), and hence edges in \( H \) are oriented by \( \overrightarrow{G} \) towards the centre \( v \) of \( H \). 
So, \( w_1,\dots,w_p \) are in-neighbours of \( v \) (in \( \overrightarrow{G} \)). 
On the other hand, for \( 1\leq r\leq p \), \( v \) has exactly one neighbour \( x_r \) in \( V_{r+1} \) and hence \( x_r \) has exactly \( p \) neighbours in \( V_0 \). 
Since \( x_r \) together with its neighbours in \( V_0 \) induce a bicoloured component with \( x_r \) as the centre, \( x_r \) is an out-neighbour of \( v \) (provided \( 1\leq r\leq p \)). 
Therefore, the orientation of edges incident on \( v \) are as shown in Figure~\ref{fig:orientation in neighbourhood of vertex in G}.

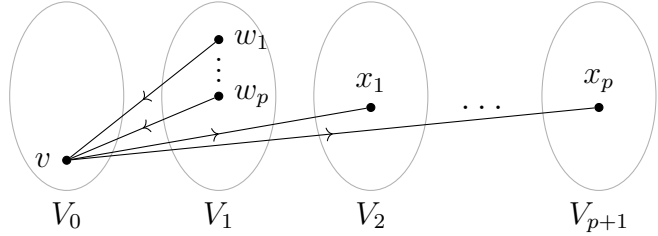
\begin{figure}[hbt]
\centering
\begin{tikzpicture}
\path (0,0)  node(V0)[Vset][label=below:\( V_0 \)]{}++(2,0) node(V1)[Vset][label=below:\( V_1 \)]{}++(2,0) node(V2)[Vset][label=below:\( V_2 \)]{}++(1.5,0) node[font=\large][yshift=-0.15cm]{\dots}++(1.5,0) node(Vp+1)[Vset][label=below:\( V_{p+1} \)]{};
\path (V0)+(0,-0.85) node(v0)[dot][label=left:\( v \)]{};
\path (V1)++(0,0.75) node(v11)[dot][label=right:\( w_1 \)]{}++(0,-0.3) node{\vdots}++(0,-0.45) node(v1p)[dot][label=right:\( w_p \)]{};
\path (V2)+(0,-0.15) node(v2)[dot][label=\( x_1 \)]{};
\path (Vp+1)+(0,-0.15) node(vp+1)[dot][label=\( x_p \)]{};
\begin{scope}[decoration={
    markings,
    mark=at position 0.5 with {\arrow{>}}},
    postaction={decorate}
    ] 
\draw [midArrow] 
(v11)--(v0);
\draw [midArrow] 
(v1p)--(v0);
\draw [midArrow]
(v0)--(v2);
\draw  [postaction={decorate}]
(v0)--(vp+1);
\end{scope}
\end{tikzpicture}
\captionsetup{width=0.75\textwidth}
\caption{Orientation of edges incident on the neighbourhood\\ of an arbitrary vertex \( v\in V_0 \).}
\label{fig:orientation in neighbourhood of vertex in G}
\end{figure}

\noindent This proves Claim~1. 
Therefore, \( \overrightarrow{G} \) is a \( (p+2) \)-colourful Eulerian orientation of \( G \).

Conversely, suppose that \( G \) admits a \( (p+2) \)-colourful Eulerian orientation with a \( (p+2) \)-colouring \( f \) as the underlying colouring. 
By the definition of a \( (p+2) \)-colourful Eulerian orientation, each bicoloured component \( H \) of \( G \) under \( f \) consists of some vertex \( v \) and all \( p \) in-neighbours of \( v \); this implies that \( H \) is isomorphic to \( K_{1,p} \). 
Since every bicoloured component of \( G \) under \( f \) is a star, \( f \) is a \( (p+2) \)-star colouring of \( G \). 
This proves the converse part. 
\end{proof}

The next theorem shows that the presence of some subgraphs makes it impossible for a graph to admit a colourful Eulerian orientation. 
\begin{figure}[hbt]
\centering
\begin{subfigure}[b]{0.3\textwidth}
\centering
\begin{tikzpicture}
\draw (0,0) node(a)[dot][label=left:\( a \)]{}--++(1,1)  node(x)[dot][label=above:\( x \)]{}--++(1,-1) node(b)[dot][label=right:\( b \)]{}--++(-1,-1) node(y)[dot][label=below:\( y \)]{}--(a);
\draw (x)--(y);
\end{tikzpicture}
\caption{diamond}
\label{fig:CEO obstruction 1}
\end{subfigure}%
\begin{subfigure}[b]{0.4\textwidth}
\centering
\begin{tikzpicture}
\draw (0,0) node(u)[dot][label=left:\( u \)]{}--++(-1,-1)  node(v)[dot][label=left:\( v \)]{}--++(0,2) node(w)[dot][label=left:\( w \)]{}--(u);
\draw (u)--++(2,0)  node(x)[dot][label=right:\( x \)]{}--++(-1,-1) node(a)[dot][label=\( a \)]{}--(u);
\draw (x)--++(1,1)  node(z)[dot][label=right:\( z \)]{}--++(0,-2) node(y)[dot][label=right:\( y \)]{}--(x);
\draw (a)--+(0,-1.5) node(b)[dot][label=below:\( b \)]{};
\draw (w)--(z)  (v)--(b)--(y);
\end{tikzpicture}
\caption{}
\label{fig:CEO obstruction 2}
\end{subfigure}%
\vspace*{5pt}

\begin{subfigure}[b]{0.3\textwidth}
\centering
\begin{tikzpicture}
\draw (0,0) node(u)[dot][label=left:\( u \)]{}--++(-1,-1)  node(v)[dot][label=left:\( v \)]{}--++(0,2) node(w)[dot][label=left:\( w \)]{}--(u);
\draw (u)--++(2,0)  node(x)[dot][label=right:\( x \)]{};
\draw (x)--++(1,1)  node(z)[dot][label=right:\( z \)]{}--++(0,-2) node(y)[dot][label=right:\( y \)]{}--(x);
\draw (w)--(z)  (v)--(y);
\end{tikzpicture}
\caption{\( \overbar{C_6} \)}
\label{fig:CEO obstruction 3}
\end{subfigure}%
\hfill
\begin{subfigure}[b]{0.3\textwidth}
\centering
\begin{tikzpicture}
\draw (0,0) node(u)[dot][label=left:\( u \)]{}--++(-1,-1)  node(v)[dot][label=left:\( v \)]{}--++(0,2) node(w)[dot][label=left:\( w \)]{}--(u);
\path (u)--++(2,0)  node(x)[dot][label=right:\( x \)]{};
\draw (x)--++(1,1)  node(z)[dot][label=right:\( z \)]{}--++(0,-2) node(y)[dot][label=right:\( y \)]{}--(x);
\draw (w)--(z)  (v)--(y);
\draw (w)--node(a)[dot][label=above:\( a \)]{} (x);
\draw (y)--node(b)[dot][label=below:\( b \)]{} (u);
\draw (a)--(b);
\end{tikzpicture}
\caption{}
\label{fig:CEO obstruction 4}
\end{subfigure}%
\hfill
\begin{subfigure}[b]{0.3\textwidth}
\centering
\begin{tikzpicture}
\draw (0,0) node(u)[dot][label=left:\( u \)]{}--++(-1,-1)  node(v)[dot][label=left:\( v \)]{}--++(0,2) node(w)[dot][label=left:\( w \)]{}--(u);
\path (u)--++(2,0)  node(x)[dot][label=right:\( x \)]{};
\draw (x)--++(1,1)  node(z)[dot][label=right:\( z \)]{}--++(0,-2) node(y)[dot][label=right:\( y \)]{}--(x);
\draw (w)--(z)  (v)--(y);
\draw (w)--node(a)[dot][pos=0.6][label=above:\( a \)]{} (x);
\draw (v)--node(c)[dot][pos=0.6][label=below:\( c \)]{} (x);
\draw (a)--node(b)[dot][label=right:\( b \)]{} (c);
\draw (u)--(b);
\end{tikzpicture}
\caption{}
\label{fig:CEO obstruction 5}
\end{subfigure}%
\caption{Some obstructions to colourful Eulerian orientation.}
\label{fig:CEO obstructions}
\end{figure}
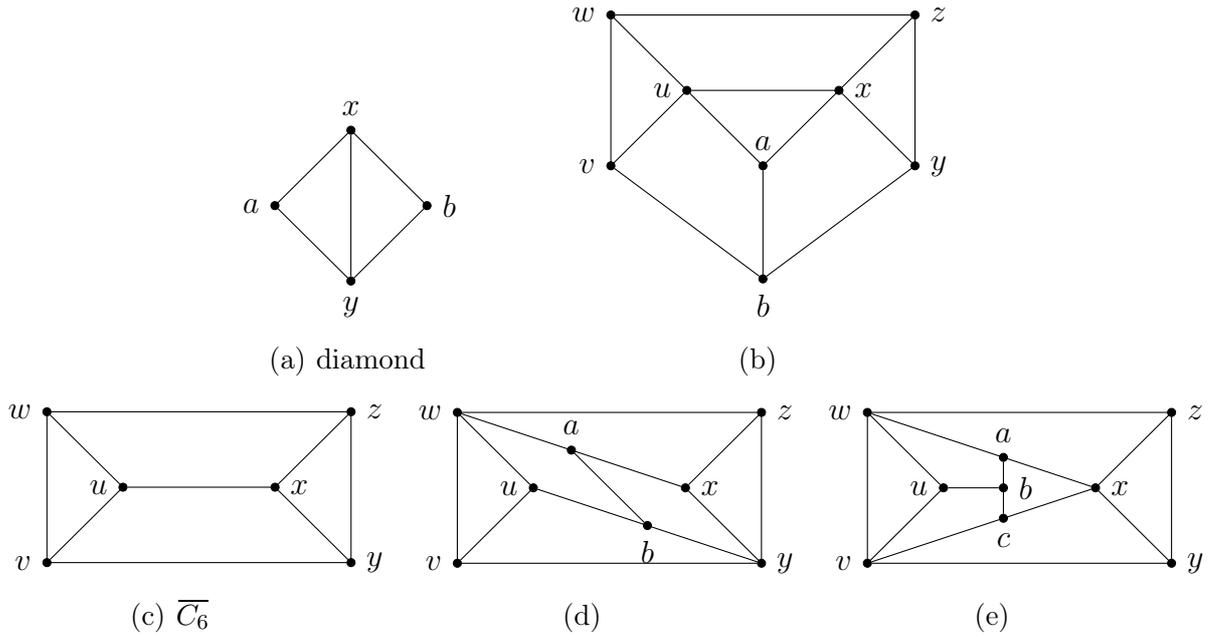
\begin{theorem}\label{thm:graphs without CEO}
Let \( G \) be a graph that contains at least one of the graphs in Figure~\ref{fig:CEO obstructions} as a subgraph. 
Then, \( G \) does not admit a colourful Eulerian orientation. 
In particular, if \( G \) is \( d \)-regular, then \( \chi_s(G)\geq \ceil{(d+5)/2} \). 
\end{theorem}
See supplementary material for the proof of Theorem~\ref{thm:graphs without CEO} and a discussion of the complexity of colourful Eulerian orientations.

\section{Hardness Results}\label{sec:hardness}
A graph \( G \) is 1-star colourable if and only if \( G \) is the complement of a complete graph. 
A graph \( G \) is 2-star colourable if and only if \( G \) is a disjoint union of stars. 
Hence, \textsc{\( k \)-Star Colourability} is polynomial-time solvable for \( k\leq 2 \). 
We prove that for all \( k\geq 3 \), \textsc{\( k \)-Star Colourability} in NP-complete for graphs of maximum degree \( k \). 
We present hardness results on 3-star colouring in Subsection~\ref{sec:hardness 3-star} and results on \( k \)-star colouring with \( k\geq 4 \) in Subsection~\ref{sec:hardness k-star k>=4}. 

\subsection{3-Star Colouring}\label{sec:hardness 3-star}
It is known that \textsc{3-Star Colourability} is NP-complete for planar bipartite graphs \cite{albertson}. 
We prove that the problem remains NP-complete when further restricted to graphs of maximum degree three and arbitrarily large girth. 
Let us start with a simple observation on 3-star colourings, which is quite useful in our reductions. 
By our convention, every 3-star colouring uses colours 0,1 and 2. 
\begin{observation}\label{obs:P3 force colours}
Let \( f \) be a 3-star colouring of a graph \( G \). If \( u,v,w \) is a path in \( G \) which is bicoloured, say with colours 0 and 1, then every neighbour of \( w \) except \( v \) must be coloured 2. Moreover, if \( w \) has two neighbours \( x_1 \) and \( x_2 \) besides \( v \), then their neighbours  \( y\neq w \) must be coloured \( f(v) \) for the same reason (see Figure~\ref{fig:P3 force colours}). 
\end{observation}

\begin{figure}[hbt]
\centering
\begin{tikzpicture}
\draw (0,0) node[dot](a)[label=below:\( u \)][label={[vcolour]above:1}]{} --++(1,0) node[dot](b)[label=below:\( v \)][label={[vcolour]above:0}]{}  --++(1,0) node[dot](c)[label=below:\( w \)][label={[vcolour]above:1}]{};
\draw[dashed] (c)  --+(20:1) node[dot](d1)[label=below:\( x_1 \)]{}
              (c)  --+(-20:1) node[dot](d2)[label=below:\( x_2 \)]{};
\path (c) --++(1.5,0) coordinate (from) --+(0.5,0) coordinate(to);

\draw[-Implies,double distance=1.3pt,thick,shorten <=-0.5pt] (from)--(to);
\path (to) --+(0.75,0) coordinate (mid);

\draw (mid) node[dot](a')[label=below:\( u \)][label={[vcolour]above:1}]{} --++(1,0) node[dot](b')[label=below:\( v \)][label={[vcolour]above:0}]{}  --++(1,0) node[dot](c')[label=below:\( w \)][label={[vcolour]above:1}]{};
\draw[dashed] (c')  --+(20:1) node[dot](d1')[label=below:\( x_1 \)][label={[solid,vcolour]above:2}]{}
              (c')  --+(-20:1) node[dot](d2')[label=below:\( x_2 \)][label={[solid,vcolour]right:2}]{};
\path (c') --++(1.75,0) coordinate (from) --+(0.5,0) coordinate(to);

\draw[-Implies,double distance=1.3pt,thick,shorten <=-0.5pt] (from)--(to);
\path (to) --+(0.75,0) coordinate (rhs);
\draw (rhs) node[dot](a')[label=below:\( u \)][label={[vcolour]above:1}]{} --++(1,0) node[dot](b')[label=below:\( v \)][label={[vcolour]above:0}]{}  --++(1,0) node[dot](c')[label=below:\( w \)][label={[vcolour]above:1}]{};
\draw[dashed] (c')  --++(20:1) node[dot](d1')[label=below:\( x_1 \)][label={[solid,vcolour]above:2}]{} --+(1,0) node[dot][label=below:\( y \)][label={[solid,vcolour]above:0}]{}
              (c')  --+(-20:1) node[dot](d2')[label=below:\( x_2 \)][label={[solid,vcolour]right:2}]{};
\path (c') --++(1.5,0) coordinate (from) --+(0.5,0) coordinate(to);
\end{tikzpicture}
\caption{Under a 3-star colouring, colours are forced on neighbours of endpoints of bicoloured \( P_3 \)'s.}
\label{fig:P3 force colours}
\end{figure}
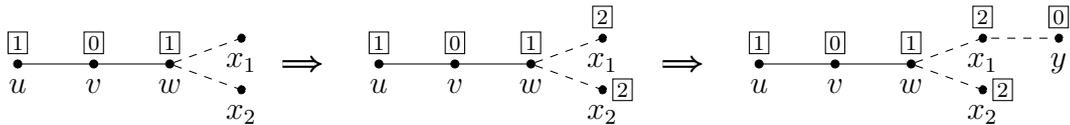

To make gadgets, we use the bipartite graph displayed in Figure~\ref{fig:gadget component}; let us call it the gadget component (only some of its vertices are labelled; these labels are essential to proof of Lemma~\ref{lem:gadget component}).  
%
Lemma~\ref{lem:gadget component} shows that the gadget component admits a \textit{unique} 3-star colouring (recall that `unique' in italics indicates ``unique up to colour swaps'').

\begin{lemma}\label{lem:gadget component}
The colouring displayed in Figure~\ref{fig:3-star colouring of component} is the \textit{unique} 3-star colouring of the gadget component. 
In particular, \( f(x)=f(y)=f(z) \) for every 3-star colouring \( f \) of the gadget component. 
\end{lemma}

\begin{figure}[hbt]
\centering
\begin{subfigure}[b]{0.5\textwidth}
\centering
\begin{tikzpicture}
\node[dot] (w)[label=below:\( w \)]{};
\draw (w)--+(90:2.5) node[dot] (x)[label=above:\( x \)]{}
      (w)--+(-30:2.5) node[dot] (z)[label=right:\( z \)]{}
      (w)--+(-150:2.5) node[dot] (y)[label=left:\( y \)]{};
\draw (x)--node[dot][pos=0.17](afterXtoY)[label=left:\( a_4 \)]{} node[dot][pos=0.33][label=left:\( a_3 \)]{} node[dot][pos=0.5][label=left:\( a_2 \)]{} node[dot][pos=0.66][label=left:\( a_1 \)]{} node[dot][pos=0.83](afterYtoX){} (y);
\draw (x)--node[dot][pos=0.17](afterXtoZ)[label=right:\( a_5 \)]{} node[dot][pos=0.33][label=right:\( a_6 \)]{} node[dot][pos=0.5][label=right:\( a_7 \)]{} node[dot][pos=0.66][label=right:\( a_8 \)]{} node[dot][pos=0.83](afterZtoX){} (z);
\draw (y)--node[dot][pos=0.17](afterYtoZ){} node[dot][pos=0.33]{} node[dot][pos=0.5]{} node[dot][pos=0.66]{} node[dot][pos=0.83](afterZtoY){} (z);
\draw (afterXtoY)--+(150:0.5) node[dot]{}
      (afterYtoX)--+(150:0.5) node[dot]{};
\draw (afterXtoZ)--+(30:0.5) node[dot]{}
      (afterZtoX)--+(30:0.5) node[dot]{};
\draw (afterYtoZ)--+(-90:0.5) node[dot]{}
      (afterZtoY)--+(-90:0.5) node[dot]{};
\end{tikzpicture}
\caption{}
\label{fig:gadget component}
\end{subfigure}%
\begin{subfigure}[b]{0.5\textwidth}
\centering
\begin{tikzpicture}
\node[dot] (w)[label=below:\( w \)][label={[vcolour]above right:0}]{};
\draw (w)--+(90:2.5) node[dot] (x)[label=above:\( x \)][label={[vcolour]right:1}]{}
      (w)--+(-30:2.5) node[dot] (z)[label=right:\( z \)][label={[vcolour]below:1}]{}
      (w)--+(-150:2.5) node[dot] (y)[label=left:\( y \)][label={[vcolour]below:1}]{};
\draw (x)--node[dot][pos=0.17](afterXtoY)[label={[vcolour]right:2}]{} node[dot][pos=0.33][label={[vcolour]right:0}]{} node[dot][pos=0.5][label={[vcolour]right:1}]{} node[dot][pos=0.66][label={[vcolour]right:0}]{} node[dot][pos=0.83][label={[vcolour]right:2}](afterYtoX){} (y);
\draw (x)--node[dot][pos=0.17](afterXtoZ)[label={[vcolour]left:2}]{} node[dot][pos=0.33][label={[vcolour]left:0}]{} node[dot][pos=0.5][label={[vcolour]left:1}]{} node[dot][pos=0.66][label={[vcolour]left:0}]{} node[dot][pos=0.83][label={[vcolour]left:2}](afterZtoX){} (z);
\draw (y)--node[dot][pos=0.17][label={[vcolour]above:2}](afterYtoZ){} node[dot][pos=0.33][label={[vcolour]above:0}]{} node[dot][pos=0.5][label={[vcolour]above:1}]{} node[dot][pos=0.66][label={[vcolour]above:0}]{} node[dot][pos=0.83][label={[vcolour]above:2}](afterZtoY){} (z);
\draw (afterXtoY)--+(150:0.5) node[dot][label={[vcolour]left:0}]{}
      (afterYtoX)--+(150:0.5) node[dot][label={[vcolour]left:0}]{};
\draw (afterXtoZ)--+(30:0.5) node[dot][label={[vcolour]right:0}]{}
      (afterZtoX)--+(30:0.5) node[dot][label={[vcolour]right:0}]{};
\draw (afterYtoZ)--+(-90:0.5) node[dot][label={[vcolour]below:0}]{}
      (afterZtoY)--+(-90:0.5) node[dot][label={[vcolour]below:0}]{};
\end{tikzpicture}
\caption{}
\label{fig:3-star colouring of component}
\end{subfigure}
\caption{(a)~Gadget component, and (b)~\textit{Unique} 3-star colouring of it.}
\end{figure}
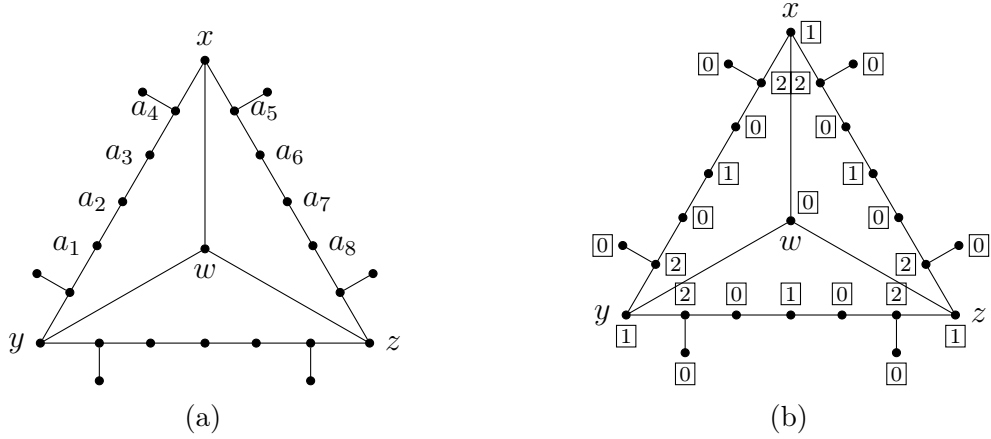

\begin{proof}
Let \( f \) be a 3-star colouring of the gadget component. 
Without loss of generality, assume that \( f(w)=0 \). 
First, we prove that \( f \) must use the same colour on vertices \( x,y \) and \( z \). 
Since only colours 1 and 2 are available for vertices \( x,y \) and \( z \), at least two of them should get the same colour. 
Without loss of generality, assume that \( f(y)=f(z)=1 \). 
Since \( y,w,z \) is a bicoloured \( P_3 \), neighbours of \( y \) on the outer cycle must be coloured 2 due to Observation~\ref{obs:P3 force colours}. 
Repeated application of Observation~\ref{obs:P3 force colours} reveals that colours are forced on vertices of the gadget component as shown in Figure~\ref{fig:colours in component}. 

\begin{figure}[hbt]
\centering
\begin{minipage}[b]{0.48\textwidth}
\centering
\begin{tikzpicture}
\node[dot] (w)[label=below:\( w \)][label={[vcolour]above right:0}]{};
\draw (w)--+(90:2.5) node[dot] (x)[label=above:\( x \)]{}
      (w)--+(-30:2.5) node[dot] (z)[label=right:\( z \)][label={[vcolour]below:1}]{}
      (w)--+(-150:2.5) node[dot] (y)[label=left:\( y \)][label={[vcolour]below:1}]{};
\draw (x)--node[dot][pos=0.17](afterXtoY)[label=left:\( a_4 \)]{} node[dot][pos=0.33][label=left:\( a_3 \)]{} node[dot][pos=0.5][label={[vcolour]right:1}][label=left:\( a_2 \)]{} node[dot][pos=0.66][label={[vcolour]right:0}][label=left:\( a_1 \)]{} node[dot][pos=0.83][label={[vcolour]right:2}](afterYtoX){} (y);
\draw (x)--node[dot][pos=0.17](afterXtoZ)[label=right:\( a_5 \)]{} node[dot][pos=0.33][label=right:\( a_6 \)]{} node[dot][pos=0.5][label={[vcolour]left:1}][label=right:\( a_7 \)]{} node[dot][pos=0.66][label={[vcolour]left:0}][label=right:\( a_8 \)]{} node[dot][pos=0.83][label={[vcolour]left:2}](afterZtoX){} (z);
\draw (y)--node[dot][pos=0.17][label={[vcolour]above:2}](afterYtoZ){} node[dot][pos=0.33][label={[vcolour]above:0}]{} node[dot][pos=0.5][label={[vcolour]above:1}]{} node[dot][pos=0.66][label={[vcolour]above:0}]{} node[dot][pos=0.83][label={[vcolour]above:2}](afterZtoY){} (z);
\draw (afterXtoY)--+(150:0.5) node[dot]{}
      (afterYtoX)--+(150:0.5) node[dot][label={[vcolour]above left:0}]{};
\draw (afterXtoZ)--+(30:0.5) node[dot]{}
      (afterZtoX)--+(30:0.5) node[dot][label={[vcolour]above right:0}]{};
\draw (afterYtoZ)--+(-90:0.5) node[dot][label={[vcolour]below:0}]{}
      (afterZtoY)--+(-90:0.5) node[dot][label={[vcolour]below:0}]{};
\end{tikzpicture}
\captionsetup{width=0.7\textwidth}
\caption{Colours forced\\in the gadget component.}
\label{fig:colours in component}
\end{minipage}%
\hspace{5pt}
\begin{minipage}[b]{0.48\textwidth}
\centering
\begin{tikzpicture}
\node[dot] (w)[label=below:\( w \)][label={[vcolour]above right:0}]{};
\draw (w)--+(90:2.5) node[dot] (x)[label=above:\( x \)][label={[vcolour]right:2}]{}
      (w)--+(-30:2.5) node[dot] (z)[label=right:\( z \)][label={[vcolour]below:1}]{}
      (w)--+(-150:2.5) node[dot] (y)[label=left:\( y \)][label={[vcolour]below:1}]{};
\draw (x)--node[dot][pos=0.17](afterXtoY)[label=left:\( a_4 \)][label={[vcolour]right:0}](a4){} node[dot][pos=0.33][label=left:\( a_3 \)][label={[vcolour]right:2}](a3){} node[dot][pos=0.5][label={[vcolour]right:1}][label=left:\( a_2 \)]{} node[dot][pos=0.66][label={[vcolour]right:0}][label=left:\( a_1 \)]{} node[dot][pos=0.83][label={[vcolour]right:2}](afterYtoX){} (y);
\draw (x)--node[dot][pos=0.17](afterXtoZ)[label=right:\( a_5 \)][label={[vcolour]left:0}](a5){} node[dot][pos=0.33][label=right:\( a_6 \)][label={[vcolour]left:2}](a6){} node[dot][pos=0.5][label={[vcolour]left:1}][label=right:\( a_7 \)]{} node[dot][pos=0.66][label={[vcolour]left:0}][label=right:\( a_8 \)]{} node[dot][pos=0.83][label={[vcolour]left:2}](afterZtoX){} (z);
\draw (y)--node[dot][pos=0.17][label={[vcolour]above:2}](afterYtoZ){} node[dot][pos=0.33][label={[vcolour]above:0}]{} node[dot][pos=0.5][label={[vcolour]above:1}]{} node[dot][pos=0.66][label={[vcolour]above:0}]{} node[dot][pos=0.83][label={[vcolour]above:2}](afterZtoY){} (z);
\draw (afterXtoY)--+(150:0.5) node[dot]{}
      (afterYtoX)--+(150:0.5) node[dot]{};
\draw (afterXtoZ)--+(30:0.5) node[dot]{}
      (afterZtoX)--+(30:0.5) node[dot]{};
\draw (afterYtoZ)--+(-90:0.5) node[dot][label={[vcolour]below:0}]{}
      (afterZtoY)--+(-90:0.5) node[dot][label={[vcolour]below:0}]{};

\draw [ultra thick] (a3)--(a4)--(x)--(a5);
\end{tikzpicture}
\captionsetup{width=0.7\textwidth}
\caption{\( f(x)=2 \) leads to\\a contradiction.}
\label{fig:f(x)=2 gives contradiction}
\end{minipage}
\end{figure}

~\\
\noindent \textbf{Claim:} \( f(x)= 1 \).\\[5pt]
On the contrary, assume that \( f(x)=2 \). 
If \( f(a_3)=0 \), then the bicoloured path \( a_1,a_2,a_3 \) forces colour 2 at \( a_4 \) by Observation~\ref{obs:P3 force colours}; this is a contradiction because \( f(x)=2 \). 
So, \( f(a_3)=2 \). 
Similarly, \( f(a_6)=2 \). 
Now, \( f(a_4)\neq 1 \) (if not, \( a_2,a_3,a_4,x \) is a bicoloured~\( P_4 \)). 
Hence, \( f(a_4)=0 \). 
Similarly, \( f(a_5)=0 \). 
Then, \( a_3,a_4,x,a_5 \) is a bicoloured \( P_4 \); a contradiction (see Figure~\ref{fig:f(x)=2 gives contradiction}). 
This proves the claim.\\

Therefore, \( f(x)=f(y)=f(z) \). 
Thus, by symmetry, the colouring in Figure~\ref{fig:3-star colouring of component} is the \textit{unique} 3-star colouring of the gadget component.
\end{proof}

For every construction in this paper, the output graph is made up of gadgets. 
For every gadget, only some of the vertices in it are allowed to have edges to vertices outside the gadget; we call these vertices as \emph{terminals}. 
In diagrams, we draw a \textit{circle around each terminal}. 

\begin{construct}\label{make:3-star}
\emph{Input:} A graph \( G \) of maximum degree four.\\
\emph{Output:} A bipartite graph \( G' \) of maximum degree three and girth eight.\\
\emph{Guarantee 1:} \( G \) is 3-colourable if and only if \( G' \) is 3-star colourable.\\
\emph{Guarantee 2:} \( G' \) has only \( O(n) \) vertices where \( n=|V(G)| \).\\
\emph{Guarantee 3:} If \( G \) is planar, then \( G' \) is planar (and the construction can be done in polynomial time).\\
\emph{Steps:}\\
First, replace each vertex \( v \) of \( G \) by a vertex gadget as shown in Figure~\ref{fig:vertex replacement}. 
For each vertex \( v \) of \( G \), the vertex gadget for \( v \) has four terminals \( v_1,v_2,v_3,v_4 \) which accommodate the edges incident on \( v \) in \( G \) (each terminal takes at most one edge; order does not matter). 
Replacement of vertices by vertex gadgets converts each edge \( uv \) of \( G \) to an edge \( u_iv_j \) between two terminals (i.e., there exists a unique \( i\in\{1,2,3,4\} \) and a unique \( j\in\{1,2,3,4\} \) such that \( u_iv_j \) is an edge). 
Finally, replace each edge \( u_iv_j \) between two terminals by an edge gadget as shown in Figure~\ref{fig:edge replacement} (the edge gadget is not symmetric; it does not matter which way we connect). 
Figure~\ref{fig:vertex edge replacement} shows an overview. 

The graph \( G\bm{'} \) is bipartite (small dots form one part and big dots form the other part in the bipartition; see Figure~\ref{fig:vertex edge replacement}).
Observe that the vertex gadget and the edge gadget have maximum degree three and girth eight. 
Since a terminal in \( G' \) is shared by at most two gadgets in \( G' \), \( G' \) has maximum degree three. 
Since the distance between two terminals in \( G' \) is at least eight, there is no cycle of length less than eight in \( G' \). 
So, \( G' \) has girth eight. 

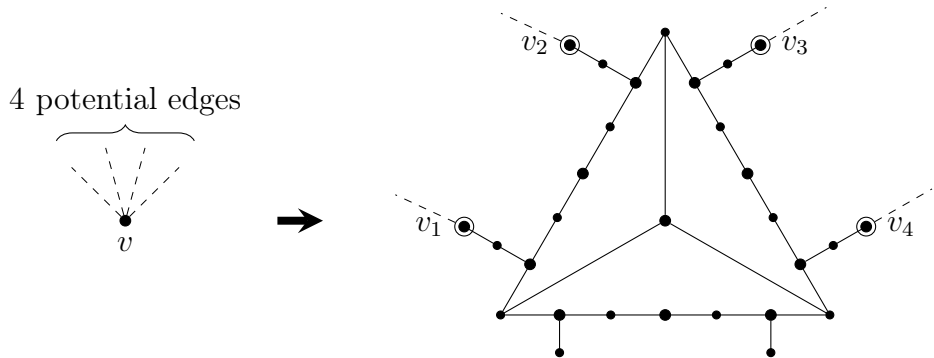
\begin{figure}[hbt]
\centering
\begin{tikzpicture}
\path (0,0)  node[bigdot](v)[label=below:\( v \)]{} --++(2,0) coordinate(from) --++(0.6,0) coordinate(to);
\draw[dashed,very thin] (v)--+(45:1) coordinate(rightNbr)
                        (v)--+(75:1)
                        (v)--+(105:1)
                        (v)--+(135:1)coordinate(leftNbr);
\draw [decorate,decoration={brace,amplitude=6pt,raise=10pt}] (leftNbr)++(-0.2,0)--node[yshift=25pt]{\( 4 \) potential edges} ($(rightNbr)+(0.2,0)$);

\draw [-stealth,draw=black,line width=3pt] (from)--(to);

\path (to) --+(4.5,0) node[bigdot] (w){};
\draw (w)--+(90:2.5) node[dot] (x){}
      (w)--+(-30:2.5) node[dot] (z){}
      (w)--+(-150:2.5) node[dot] (y){};
\draw (x)--node[bigdot][pos=0.17](afterXtoY){} node[dot][pos=0.33]{} node[bigdot][pos=0.5]{} node[dot][pos=0.66]{} node[bigdot][pos=0.83](afterYtoX){} (y);
\draw (x)--node[bigdot][pos=0.17](afterXtoZ){} node[dot][pos=0.33]{} node[bigdot][pos=0.5]{} node[dot][pos=0.66]{} node[bigdot][pos=0.83](afterZtoX){} (z);
\draw (y)--node[bigdot][pos=0.17](afterYtoZ){} node[dot][pos=0.33]{} node[bigdot][pos=0.5]{} node[dot][pos=0.66]{} node[bigdot][pos=0.83](afterZtoY){} (z);
\draw (afterXtoY)--++(150:0.5) node[dot]{}--++(150:0.5) node[bigdot]{} node[terminal](v1)[label=left:\( v_2 \)]{}
      (afterYtoX)--++(150:0.5) node[dot]{}--++(150:0.5) node[bigdot]{} node[terminal](v2)[label=left:\( v_1 \)]{};
\draw (afterXtoZ)--++(30:0.5) node[dot]{}--++(30:0.5) node[bigdot]{} node[terminal](v3)[label=right:\( v_3 \)]{}
      (afterZtoX)--++(30:0.5) node[dot]{}--++(30:0.5) node[bigdot]{} node[terminal](v4)[label=right:\( v_4 \)]{};
\draw (afterYtoZ)--+(-90:0.5) node[dot]{}
      (afterZtoY)--+(-90:0.5) node[dot]{};
\draw[dashed,very thin] (v4)--+(30:1)
                        (v3)--+(30:1)
                        (v1)--+(155:1)
                        (v2)--+(155:1);
\end{tikzpicture}
\caption{Replacement of vertex by vertex gadget.}
\label{fig:vertex replacement}
\end{figure}

\begin{figure}[hbt]
\centering
\begin{tikzpicture}
\draw (0,0) node[bigdot]{} node[terminal](uk)[label=below:\( u_i \)]{} --++(1,0) node[bigdot]{} node[terminal](vl)[label=below:\( v_j \)]{};
\path (vl) --++(1,0) coordinate(from) --++(0.6,0) coordinate(to);
\draw [-stealth,draw=black,line width=3pt] (from)--(to);

\path (to)--+(4.5,0) node[bigdot] (w){};
\draw (w)--+(90:2.5) node[dot] (x){}
      (w)--+(-30:2.5) node[dot] (z){}
      (w)--+(-150:2.5) node[dot] (y){};
\draw (x)--node[bigdot][pos=0.17](afterXtoY){} node[dot][pos=0.33]{} node[bigdot][pos=0.5]{} node[dot][pos=0.66]{} node[bigdot][pos=0.83](afterYtoX){} (y);
\draw (x)--node[bigdot][pos=0.17](afterXtoZ){} node[dot][pos=0.33]{} node[bigdot][pos=0.5]{} node[dot][pos=0.66]{} node[bigdot][pos=0.83](afterZtoX){} (z);
\draw (y)--node[bigdot][pos=0.17](afterYtoZ){} node[dot][pos=0.33]{} node[bigdot][pos=0.5]{} node[dot][pos=0.66]{} node[bigdot][pos=0.83](afterZtoY){} (z);
\draw (afterXtoY)--+(150:0.5) node[dot]{}
      (afterYtoX)--++(150:0.5) node[dot]{}--+(150:0.5) node[bigdot]{} node[terminal][label={left:\( u_i \)}]{};
\draw (afterXtoZ)--+(30:0.5) node[dot]{}
      (afterZtoX)--++(30:0.5) node[dot](3vertex1){} --++(30:0.5) node[bigdot](3vertex2){}--++(30:0.5) node[dot]{}--++(30:0.5) node[bigdot]{} node[terminal][label={right:\( v_j \)}]{};
\draw (3vertex1)--+(120:0.5) node[bigdot]{}
      (3vertex2)--+(120:0.5) node[dot]{};
\draw (afterYtoZ)--+(-90:0.5) node[dot]{}
      (afterZtoY)--+(-90:0.5) node[dot]{};

\end{tikzpicture}
\caption{Replacement of edge between terminals by edge gadget.}
\label{fig:edge replacement}
\end{figure}
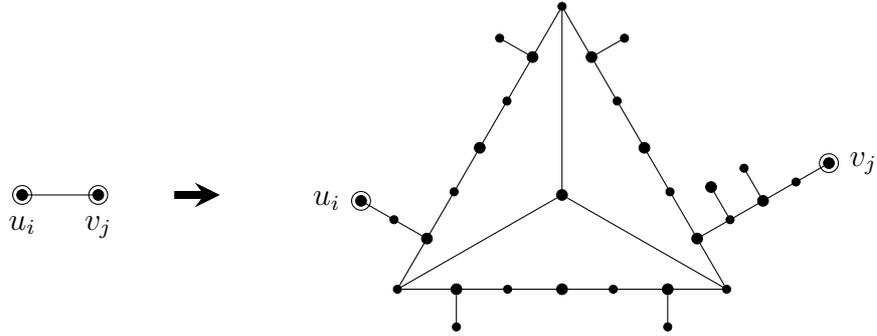

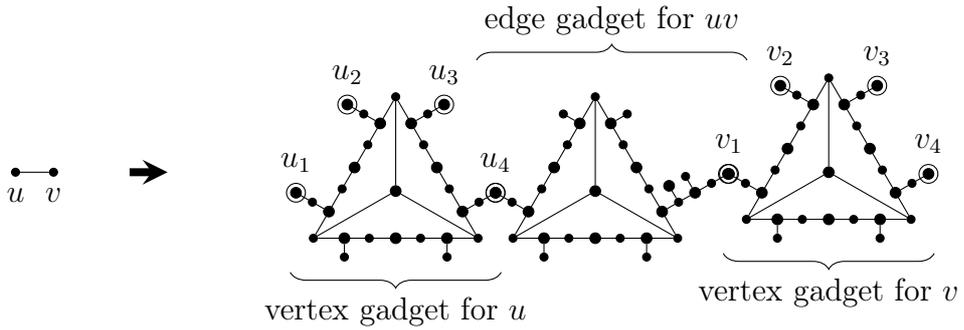
\begin{figure}[hbt]
\centering
\begin{tikzpicture}[scale=0.5]
\path (0,0) node[dot](u)[label=below:\( u \)]{} --++(1,0) node[dot](v)[label=below:\( v \)]{};
\draw (u)-- coordinate(u2v) (v);
\path (v) --++(2,0) coordinate(from) --++(1,0) coordinate(to);
\draw [-stealth,draw=black,line width=3pt] (from)--(to);

\path (to) --+(6,-0.5) node[bigdot] (w){};
\draw (w)--+(90:2.5) node[dot] (x){}
      (w)--+(-30:2.5) node[dot] (z){}
      (w)--+(-150:2.5) node[dot] (y){};
\draw (x)--node[bigdot][pos=0.17](afterXtoY){} node[dot][pos=0.33]{} node[bigdot][pos=0.5]{} node[dot][pos=0.66]{} node[bigdot][pos=0.83](afterYtoX){} (y);
\draw (x)--node[bigdot][pos=0.17](afterXtoZ){} node[dot][pos=0.33]{} node[bigdot][pos=0.5]{} node[dot][pos=0.66]{} node[bigdot][pos=0.83](afterZtoX){} (z);
\draw (y)--node[bigdot][pos=0.17](afterYtoZ){} node[dot][pos=0.33]{} node[bigdot][pos=0.5]{} node[dot][pos=0.66]{} node[bigdot][pos=0.83](afterZtoY){} (z);
\draw (afterXtoY)--++(150:0.5) node[dot]{}--++(150:0.5) node[bigdot]{} node[terminal](u1)[label=above:\( u_2 \)]{}
      (afterYtoX)--++(150:0.5) node[dot]{}--++(150:0.5) node[bigdot]{} node[terminal](u2)[label=above:\( u_1 \)]{};
\draw (afterXtoZ)--++(30:0.5) node[dot]{}--++(30:0.5) node[bigdot]{} node[terminal](u3)[label=above:\( u_3 \)]{}
      (afterZtoX)--++(30:0.5) node[dot]{}--++(30:0.5) node[bigdot]{} node[terminal](u4)[label=above:\( u_4 \)]{};
\draw (afterYtoZ)--+(-90:0.5) node[dot]{}
      (afterZtoY)--+(-90:0.5) node[dot]{};

\path (u2)--+(-0.15,0) coordinate(brace1Start);
\path (u4)--+(0.15,0) coordinate(brace1End);
\draw [decorate,decoration={brace,amplitude=6pt,raise=30pt,mirror}] (brace1Start)--  node[yshift=-45pt]{vertex gadget for \( u \)} (brace1End);

\path (w)--+(5.25,0) node[bigdot] (w){};
\draw (w)--+(90:2.5) node[dot] (x){}
      (w)--+(-30:2.5) node[dot] (z){}
      (w)--+(-150:2.5) node[dot] (y){};
\draw (x)--node[bigdot][pos=0.17](afterXtoY){} node[dot][pos=0.33]{} node[bigdot][pos=0.5]{} node[dot][pos=0.66]{} node[bigdot][pos=0.83](afterYtoX){} (y);
\draw (x)--node[bigdot][pos=0.17](afterXtoZ){} node[dot][pos=0.33]{} node[bigdot][pos=0.5]{} node[dot][pos=0.66]{} node[bigdot][pos=0.83](afterZtoX){} (z);
\draw (y)--node[bigdot][pos=0.17](afterYtoZ){} node[dot][pos=0.33]{} node[bigdot][pos=0.5]{} node[dot][pos=0.66]{} node[bigdot][pos=0.83](afterZtoY){} (z);
\draw (afterXtoY)--+(150:0.5) node[dot]{}
      (afterYtoX)--++(150:0.5) node[dot]{}--+(150:0.5) node[bigdot]{} node[terminal](u){};
\draw (afterXtoZ)--+(30:0.5) node[dot]{}
      (afterZtoX)--++(30:0.5) node[dot](3vertex1){} --++(30:0.5) node[bigdot](3vertex2){}--++(30:0.5) node[dot]{}--++(30:0.5) node[bigdot]{} node[terminal](v){};
\draw (3vertex1)--+(120:0.5) node[bigdot]{}
      (3vertex2)--+(120:0.5) node[dot]{};
\draw (afterYtoZ)--+(-90:0.5) node[dot]{}
      (afterZtoY)--+(-90:0.5) node[dot]{};

\path (u)--+(-0.5,0) coordinate(brace1Start);
\path (v)--+(0.5,-0.5) coordinate(brace1End);
\draw [decorate,decoration={brace,amplitude=6pt,raise=50pt}] (brace1Start)--  node[yshift=65pt]{edge gadget for \( uv \)} (brace1End);


\path (w)--+(6.14,0.5) node[bigdot] (w){};
\draw (w)--+(90:2.5) node[dot] (x){}
      (w)--+(-30:2.5) node[dot] (z){}
      (w)--+(-150:2.5) node[dot] (y){};
\draw (x)--node[bigdot][pos=0.17](afterXtoY){} node[dot][pos=0.33]{} node[bigdot][pos=0.5]{} node[dot][pos=0.66]{} node[bigdot][pos=0.83](afterYtoX){} (y);
\draw (x)--node[bigdot][pos=0.17](afterXtoZ){} node[dot][pos=0.33]{} node[bigdot][pos=0.5]{} node[dot][pos=0.66]{} node[bigdot][pos=0.83](afterZtoX){} (z);
\draw (y)--node[bigdot][pos=0.17](afterYtoZ){} node[dot][pos=0.33]{} node[bigdot][pos=0.5]{} node[dot][pos=0.66]{} node[bigdot][pos=0.83](afterZtoY){} (z);
\draw (afterXtoY)--++(150:0.5) node[dot]{}--++(150:0.5) node[bigdot]{} node[terminal](v1)[label=above:\( v_2 \)]{}
      (afterYtoX)--++(150:0.5) node[dot]{}--++(150:0.5) node[bigdot]{} node[terminal](v2)[label=above:\( v_1 \)]{};
\draw (afterXtoZ)--++(30:0.5) node[dot]{}--++(30:0.5) node[bigdot]{} node[terminal](v3)[label=above:\( v_3 \)]{}
      (afterZtoX)--++(30:0.5) node[dot]{}--++(30:0.5) node[bigdot]{} node[terminal](v4)[label=above:\( v_4 \)]{};
\draw (afterYtoZ)--+(-90:0.5) node[dot]{}
      (afterZtoY)--+(-90:0.5) node[dot]{};

\path (v2)--+(-0.15,0) coordinate(brace1Start);
\path (v4)--+(0.15,0) coordinate(brace1End);
\draw [decorate,decoration={brace,amplitude=6pt,raise=30pt,mirror}] (brace1Start)--  node[yshift=-45pt]{vertex gadget for \( v \)} (brace1End);

\end{tikzpicture}
\caption{Construction of \( G\bm{'} \) from \( G \). Only vertices \( u,v \) and edge \( uv \) in \( G \) and corresponding gadgets in \( G\bm{'} \) are shown.}
\label{fig:vertex edge replacement}
\end{figure}

\end{construct}

\begin{proof}[Proof of Guarantee 1]
The following claim demonstrates how the vertex gadget and the edge gadget serve their respective roles.\\

\noindent \textbf{Claim~1:} The colourings shown in Figure~\ref{fig:3-star colouring vertex and edge gadgets} are the \textit{unique} 3-star colouring of the vertex gadget and the edge gadget. In particular, under every 3-star colouring, all four terminals of a vertex gadget must get the same colour whereas the terminals of an edge gadget must get different colours.\\

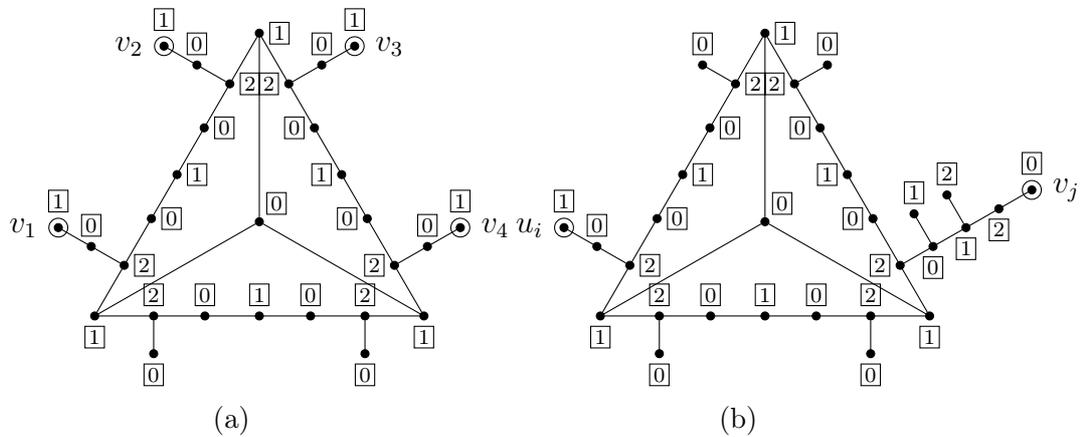
\begin{figure}[hbt]
\centering
\begin{subfigure}[b]{0.4\textwidth}
\begin{tikzpicture}

\node[dot] (w)[label={[vcolour]above right:0}]{};
\draw (w)--+(90:2.5) node[dot] (x)[label={[vcolour]right:1}]{}
      (w)--+(-30:2.5) node[dot] (z)[label={[vcolour]below:1}]{}
      (w)--+(-150:2.5) node[dot] (y)[label={[vcolour]below:1}]{};
\draw (x)--node[dot][pos=0.17](afterXtoY)[label={[vcolour]right:2}]{} node[dot][pos=0.33][label={[vcolour]right:0}]{} node[dot][pos=0.5][label={[vcolour]right:1}]{} node[dot][pos=0.66][label={[vcolour]right:0}]{} node[dot][pos=0.83][label={[vcolour]right:2}](afterYtoX){} (y);
\draw (x)--node[dot][pos=0.17](afterXtoZ)[label={[vcolour]left:2}]{} node[dot][pos=0.33][label={[vcolour]left:0}]{} node[dot][pos=0.5][label={[vcolour]left:1}]{} node[dot][pos=0.66][label={[vcolour]left:0}]{} node[dot][pos=0.83][label={[vcolour]left:2}](afterZtoX){} (z);
\draw (y)--node[dot][pos=0.17][label={[vcolour]above:2}](afterYtoZ){} node[dot][pos=0.33][label={[vcolour]above:0}]{} node[dot][pos=0.5][label={[vcolour]above:1}]{} node[dot][pos=0.66][label={[vcolour]above:0}]{} node[dot][pos=0.83][label={[vcolour]above:2}](afterZtoY){} (z);

\draw (afterYtoZ)--+(-90:0.5) node[dot][label={[vcolour]below:0}]{}
      (afterZtoY)--+(-90:0.5) node[dot][label={[vcolour]below:0}]{};

\draw (afterXtoY)--++(150:0.5) node[dot][label={[vcolour]above:0}]{}--++(150:0.5) node[dot]{} node[terminal](v1)[label=left:\( v_2 \)][label={[vcolour]above:1}]{}
      (afterYtoX)--++(150:0.5) node[dot][label={[vcolour]above:0}]{}--++(150:0.5) node[dot]{} node[terminal](v2)[label=left:\( v_1 \)][label={[vcolour]above:1}]{};
\draw (afterXtoZ)--++(30:0.5) node[dot][label={[vcolour]above:0}]{}--++(30:0.5) node[dot]{} node[terminal](v3)[label=right:\( v_3 \)][label={[vcolour]above:1}]{}
      (afterZtoX)--++(30:0.5) node[dot][label={[vcolour]above:0}]{}--++(30:0.5) node[dot]{} node[terminal](v4)[label=right:\( v_4 \)][label={[vcolour]above:1}]{};

\end{tikzpicture}
\caption{}%
\label{fig:3-star colouring vertex gadget}
\end{subfigure}%
\hspace{0.5cm}%
\begin{subfigure}[b]{0.4\textwidth}
\begin{tikzpicture}
\node[dot] (w)[label={[vcolour]above right:0}]{};
\draw (w)--+(90:2.5) node[dot] (x)[label={[vcolour]right:1}]{}
      (w)--+(-30:2.5) node[dot] (z)[label={[vcolour]below:1}]{}
      (w)--+(-150:2.5) node[dot] (y)[label={[vcolour]below:1}]{};
\draw (x)--node[dot][pos=0.17](afterXtoY)[label={[vcolour]right:2}]{} node[dot][pos=0.33][label={[vcolour]right:0}]{} node[dot][pos=0.5][label={[vcolour]right:1}]{} node[dot][pos=0.66][label={[vcolour]right:0}]{} node[dot][pos=0.83][label={[vcolour]right:2}](afterYtoX){} (y);
\draw (x)--node[dot][pos=0.17](afterXtoZ)[label={[vcolour]left:2}]{} node[dot][pos=0.33][label={[vcolour]left:0}]{} node[dot][pos=0.5][label={[vcolour]left:1}]{} node[dot][pos=0.66][label={[vcolour]left:0}]{} node[dot][pos=0.83][label={[vcolour]left:2}](afterZtoX){} (z);
\draw (y)--node[dot][pos=0.17][label={[vcolour]above:2}](afterYtoZ){} node[dot][pos=0.33][label={[vcolour]above:0}]{} node[dot][pos=0.5][label={[vcolour]above:1}]{} node[dot][pos=0.66][label={[vcolour]above:0}]{} node[dot][pos=0.83][label={[vcolour]above:2}](afterZtoY){} (z);

\draw (afterYtoZ)--+(-90:0.5) node[dot][label={[vcolour]below:0}]{}
      (afterZtoY)--+(-90:0.5) node[dot][label={[vcolour]below:0}]{};

\draw (afterXtoY)--+(150:0.5) node[dot][label={[vcolour]above:0}]{}
      (afterYtoX)--++(150:0.5) node[dot][label={[vcolour]above:0}]{}--+(150:0.5) node[dot]{} node[terminal][label={left:\( u_i \)}][label={[vcolour]above:1}]{};
\draw (afterXtoZ)--+(30:0.5) node[dot][label={[vcolour]above:0}]{}
      (afterZtoX)--++(30:0.5) node[dot](3vertex1)[label={[vcolour]below:0}]{} --++(30:0.5) node[dot](3vertex2)[label={[vcolour]below:1}]{}--++(30:0.5) node[dot][label={[vcolour]below:2}]{}--++(30:0.5) node[dot]{} node[terminal][label={right:\( v_j \)}][label={[vcolour]above:0}]{};
\draw (3vertex1)--+(120:0.5) node[dot][label={[vcolour]above:1}]{}
      (3vertex2)--+(120:0.5) node[dot][label={[vcolour]above:2}]{};
\end{tikzpicture}
\caption{}%
\label{fig:3-star colouring edge gadget}
\end{subfigure}
\caption{The \textit{unique} 3-star colouring of the vertex gadget and the edge gadget.}
\label{fig:3-star colouring vertex and edge gadgets}
\end{figure}

\noindent Recall that Figure~\ref{fig:3-star colouring of component} exhibits the \textit{unique} 3-star colouring of the gadget component by Observation~\ref{lem:gadget component}. 
This fixes colours on the gadget component within the vertex gadget (resp.\ edge gadget). We obtain Claim~1 by repeated application of Observation~\ref{obs:P3 force colours}.\\

The \textit{unique} 3-star colouring of the vertex gadget (resp.\ edge gadget) exhibited in Figure~\ref{fig:3-star colouring vertex and edge gadgets} ensures the following claim.\\[5pt]
\noindent \textbf{Claim~2:} If \( Q^* \) is a 3-vertex path in a vertex gadget (resp.\ edge gadget) and \( Q^* \) contains a terminal of the gadget, then \( Q^* \) is tricoloured.\\ 

Suppose that \( G \) admits a 3-colouring \( f \). 
We produce a 3-colouring \( f\bm{'} \) of \( G\bm{'} \) by assigning \( f\bm{'}(v_i)=f(v) \) for each vertex \( v \) of \( G \) and \( 1\leq i\leq 4 \), and extending it to vertex gadgets and edge gadgets using the schemes in Figure~\ref{fig:3-star colouring vertex and edge gadgets}.\\

\noindent \textbf{Claim~3:} \( f\bm{'} \) is a 3-star colouring of \( G\bm{'} \).\\[5pt]
On the contrary, assume that there is a 4-vertex path \( Q \) in \( G\bm{'} \) bicoloured by \( f\bm{'} \). 
Since star colouring schemes are used on gadgets in \( G\bm{'} \), \( Q \) must contain vertices from two gadgets. 
As a result, the terminal shared by the two gadgets must be in \( Q \) and the segment of \( Q \) in one of the gadgets is a 3-vertex path \( Q^* \). 
Clearly, the shared terminal must be in \( Q^* \). 
Since \( Q^* \) is a 3-vertex path in a gadget and \( Q^* \) contains a terminal of the gadget, \( Q^* \) is tricoloured by Claim~2. 
This is a contradiction to the assumption that \( Q \) is bicoloured. 
This proves Claim~3. 
Therefore, if \( G \) is 3-colourable, then \( G\bm{'} \) is 3-star colourable.\\


Conversely, suppose that \( G\bm{'} \) admits a 3-star colouring \( f\bm{'} \). By Claim~1, \( f\bm{'} \) must use (i)~the same colour on terminals of each vertex gadget, and (ii)~different colours on terminals of each edge gadget. Hence, the function \( f \) defined as \( f(v)=f\bm{'}(v_1) \) for all \( v\in V(G) \) is a 3-colouring of \( G \). 
\end{proof}
\begin{proof}[Proof of Guarantee 2]
Let us count the number of vertices and edges in \( G\bm{'} \). 
The vertex gadget has 29 vertices and 31 edges. 
The edge gadget has 29 vertices excluding the terminals, and 33 edges (let us count terminals as part of vertex gadgets). 
So, \( G\bm{'} \) has \( 29n+29m \) vertices and \( 31n+33m \) edges where \( n=|V(G)| \) and \( m=|E(G)| \). 
As \( \Delta(G)=4 \), we have \( m\leq n\Delta(G)/2=2n=O(n) \). 
Therefore, \( G' \) has only \( O(n) \) vertices and \( O(n) \) edges.
\end{proof}
\begin{proof}[Proof of Guarantee 3]
Suppose that \( G \) is planar. Fix a plane drawing of \( G \). For each vertex \( v \) of \( G \), the cyclic order of edges around \( v \) in \( G \) (usually called the rotation system at \( v \)) can be computed in time polynomial in the size of the input \( G \) \cite{nishizeki_rahman}. Hence, it is possible to construct \( G\bm{'} \) in such a way that \( G\bm{'} \) is planar, and the construction still requires only time polynomial in the size of \( G \). 
\end{proof}
\noindent Remark: If \( N=\Theta(n) \), \( g(n)=h(N) \) and \( h(N)=2^{o(N)} \) (resp.\ \( 2^{o(\sqrt{N})} \)), then \( g(n)=2^{o(n)} \) (resp.\ \( 2^{o(\sqrt{n})} \)). 
\begin{theorem}\label{thm:3-star planar bip girth 8}
\textsc{3-Star Colourability}\( ( \)planar, bipartite, \( \Delta=3 \), girth \( =8 \)\( ) \) is NP-complete, and the problem does not admit a \( 2^{o(\sqrt{n})} \)-time algorithm unless ETH fails. Moreover, the problem \textsc{3-Star Colourability}\( ( \)bipartite, \( \Delta=3 \), girth \( =8 \)\( ) \) does not admit a \( 2^{o(n)} \)-time algorithm unless ETH fails. 
\end{theorem}
\begin{proof}
We employ Construction~\ref{make:3-star} to establish a reduction from \textsc{3-Colourability}(planar, \( \Delta=4 \)) to \textsc{3-Star Colourability}\( ( \)planar, bipartite, \( \Delta=3 \), girth \( =8 \)\( ) \).
%
Let \( G \) be an instance of \textsc{3-Colourability}(planar, \( \Delta=4 \)). 
From \( G \), construct an instance \( G' \) of \textsc{3-Star Colourability}(planar, bipartite, \( \Delta=3 \), girth \( =8 \)) by Construction~\ref{make:3-star}. 
By Guarantee~1 of Construction~\ref{make:3-star}, \( G \) is 3-colourable if and only if \( G' \) is 3-star colourable. 
Since \( G \) is planar, the graph \( G' \) is planar by Guarantee~3 of Construction~\ref{make:3-star}. 
By Guarantee~2 of Construction~\ref{make:3-star}, the number of vertices in \( G' \) is \( N=O(n) \) where \( n=|V(G)| \).
Hence, \( N=\Theta(n) \). 

Since \textsc{3-Colourability}(planar, \( \Delta=4 \)) is NP-complete \cite{garey_johnson_stockmeyer1976} and the problem does not admit a \( 2^{o(\sqrt{n})} \)-time algorithm unless ETH fails \cite{biro}, \textsc{3-Star Colourability}\( ( \)planar, bipartite, \( \Delta=3 \), girth \( =8 \)\( ) \) is NP-complete, and the problem does not admit a \( 2^{o(\sqrt{n})} \)-time algorithm unless ETH fails.

Similarly, Construction~\ref{make:3-star} establishes a reduction from \textsc{3-Colourability}(\( \Delta=4 \)) to \textsc{3-Star Colourability}\( ( \)bipartite, \( \Delta=3 \), girth \( =8 \)\( ) \). Since \textsc{3-Colourability}(\( \Delta=4 \)) does not admit a \( 2^{o(n)} \)-time algorithm \cite[Lemma~2.1]{cygan2016} and \( |V(G')|=\Theta(|V(G)|) \) in Construction~\ref{make:3-star} (see Guarantee~2), \textsc{3-Star Colourability}\( ( \)bipartite, \( \Delta=3 \),\allowbreak \( \text{girth}=8 \)\( ) \) does not admit a \( 2^{o(n)} \)-time algorithm.
\end{proof}

Note that in Construction~\ref{make:3-star}, a 3-colouring \( f \) of \( G \) can be extended into a 3-star colouring \( f' \) of \( G' \) in \( 2^n \) ways where \( n=|V(G)| \) (since the colours on the terminals are fixed, each edge gadget has exactly one 3-star colouring whereas two choices are available for each vertex gadget; e.g., swapping colour~0 with colour~2 is possible in Figure~\ref{fig:3-star colouring vertex gadget}).
Therefore, the reduction in Theorem~\ref{thm:3-star planar bip girth 8} is weakly parsimonious. 
Thus, we have the following theorem since it is \#P-complete to count the number of 3-colourings of a graph of maximum degree four \cite{bubley}.
\begin{theorem}
It is \#P-complete to count the number of \( 3 \)-star colourings of a bipartite graph of maximum degree three and girth eight.\qed
\end{theorem}


The output graph \( G\bm{'} \) in Construction~\ref{make:3-star} has girth eight. 
We can modify Construction~\ref{make:3-star} to give \( G\bm{'} \) arbitrarily large girth. 
The modification required is to replace the gadget component by the new one displayed in Figure~\ref{fig:new component} below. 
For \( s=1 \), the new gadget component is as shown in Figure~\ref{fig:new component for s=1}. 

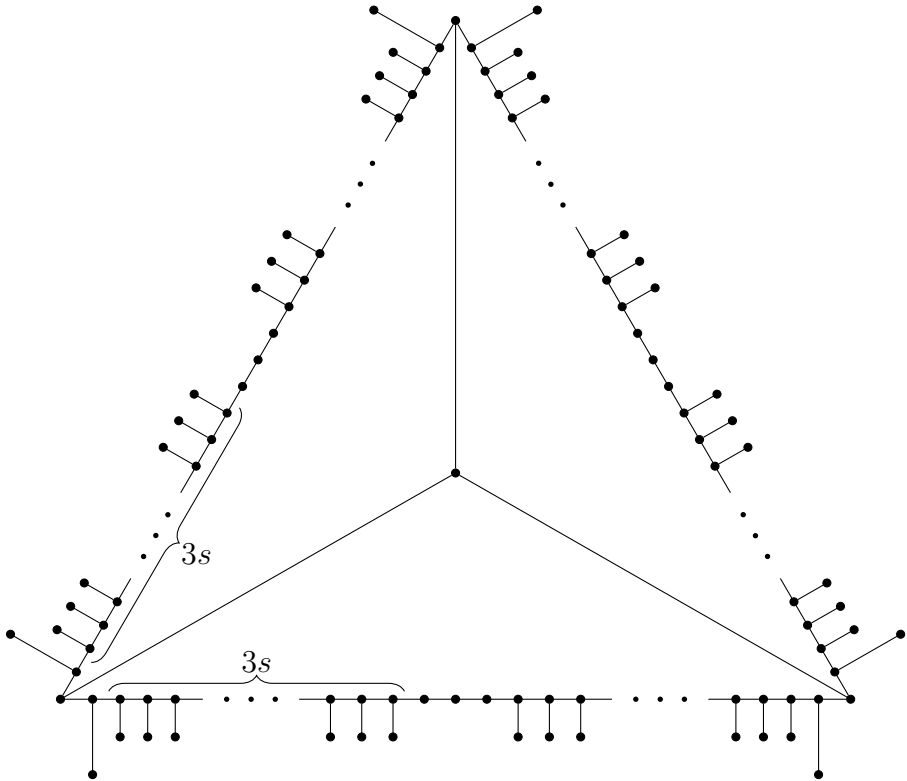
\begin{figure}[hbt]
\centering
\begin{tikzpicture}
\node[dot] (w){};
\draw (w)--+(90:6) node[dot](x){}
      (w)--+(-30:6) node[dot](z){}
      (w)--+(-150:6) node[dot](y){};

\path (x)-- node[pos=6/21](x2y-1of3){} node[pos=15/21](x2y-2of3){}  (y);

\draw (x2y-1of3)-- node[pos=1/10](x2y-mid1)[dot]{} node[pos=2/10](x2y-mid2)[dot]{} node[pos=3/10](x2y-mid3)[dot]{} node[pos=4/10](x2y-mid4)[dot]{} node[pos=5/10](x2y-mid5)[dot]{} node[pos=6/10](x2y-mid6)[dot]{} node[pos=7/10](x2y-mid7)[dot]{} node[pos=8/10](x2y-mid8)[dot]{} node[pos=9/10](x2y-mid9)[dot]{}  (x2y-2of3);

\path (x)-- node[pos=4/21](left){} node[pos=17/21](right){}  (y);

\draw (x)-- node[dot][pos=1/5](x2y-left1){} node[dot][pos=2/5](x2y-left2){} node[dot][pos=3/5](x2y-left3){} node[dot][pos=4/5](x2y-left4){} (left);

\draw (y)-- node[dot][pos=1/5](x2y-right1){} node[dot][pos=2/5](x2y-right2){} node[dot][pos=3/5](x2y-right3){} node[dot][pos=4/5](x2y-right4){} (right);

\path (left)-- node[sloped,font=\LARGE]{\( \dots \)} (x2y-1of3);
\path (right)-- node[sloped,font=\LARGE]{\( \dots \)} (x2y-2of3);

\draw (x2y-left1)--+(150:1) node[dot]{};
\draw (x2y-left2)--+(150:0.5) node[dot]{}
      (x2y-left3)--+(150:0.5) node[dot]{}
      (x2y-left4)--+(150:0.5) node[dot]{};

\draw (x2y-right1)--+(150:1) node[dot]{};
\draw (x2y-right2)--+(150:0.5) node[dot]{}
      (x2y-right3)--+(150:0.5) node[dot]{}
      (x2y-right4)--+(150:0.5) node[dot]{};

\draw (x2y-mid1)--+(150:0.5) node[dot]{}
      (x2y-mid2)--+(150:0.5) node[dot]{}
      (x2y-mid3)--+(150:0.5) node[dot]{}
      (x2y-mid7)--+(150:0.5) node[dot]{}
      (x2y-mid8)--+(150:0.5) node[dot]{}
      (x2y-mid9)--+(150:0.5) node[dot]{};

\path (x)-- node[pos=6/21](x2z-1of3){} node[pos=15/21](x2z-2of3){}  (z);

\draw (x2z-1of3)-- node[pos=1/10](x2z-mid1)[dot]{} node[pos=2/10](x2z-mid2)[dot]{} node[pos=3/10](x2z-mid3)[dot]{} node[pos=4/10](x2z-mid4)[dot]{} node[pos=5/10](x2z-mid5)[dot]{} node[pos=6/10](x2z-mid6)[dot]{} node[pos=7/10](x2z-mid7)[dot]{} node[pos=8/10](x2z-mid8)[dot]{} node[pos=9/10](x2z-mid9)[dot]{}  (x2z-2of3);

\path (x)-- node[pos=4/21](left){} node[pos=17/21](right){}  (z);

\draw (x)-- node[dot][pos=1/5](x2z-left1){} node[dot][pos=2/5](x2z-left2){} node[dot][pos=3/5](x2z-left3){} node[dot][pos=4/5](x2z-left4){} (left);

\draw (z)-- node[dot][pos=1/5](x2z-right1){} node[dot][pos=2/5](x2z-right2){} node[dot][pos=3/5](x2z-right3){} node[dot][pos=4/5](x2z-right4){} (right);

\path (left)-- node[sloped,font=\LARGE]{\( \dots \)} (x2z-1of3);
\path (right)-- node[sloped,font=\LARGE]{\( \dots \)} (x2z-2of3);

\draw (x2z-left1)--+(30:1) node[dot]{};
\draw (x2z-left2)--+(30:0.5) node[dot]{}
      (x2z-left3)--+(30:0.5) node[dot]{}
      (x2z-left4)--+(30:0.5) node[dot]{};

\draw (x2z-right1)--+(30:1) node[dot]{};
\draw (x2z-right2)--+(30:0.5) node[dot]{}
      (x2z-right3)--+(30:0.5) node[dot]{}
      (x2z-right4)--+(30:0.5) node[dot]{};

\draw (x2z-mid1)--+(30:0.5) node[dot]{}
      (x2z-mid2)--+(30:0.5) node[dot]{}
      (x2z-mid3)--+(30:0.5) node[dot]{}
      (x2z-mid7)--+(30:0.5) node[dot]{}
      (x2z-mid8)--+(30:0.5) node[dot]{}
      (x2z-mid9)--+(30:0.5) node[dot]{};

\path (y)-- node[pos=6/21](y2z-1of3){} node[pos=15/21](y2z-2of3){}  (z);

\draw (y2z-1of3)-- node[pos=1/10](y2z-mid1)[dot]{} node[pos=2/10](y2z-mid2)[dot]{} node[pos=3/10](y2z-mid3)[dot]{} node[pos=4/10](y2z-mid4)[dot]{} node[pos=5/10](y2z-mid5)[dot]{} node[pos=6/10](y2z-mid6)[dot]{} node[pos=7/10](y2z-mid7)[dot]{} node[pos=8/10](y2z-mid8)[dot]{} node[pos=9/10](y2z-mid9)[dot]{}  (y2z-2of3);

\path (y)-- node[pos=4/21](left){} node[pos=17/21](right){}  (z);

\draw (y)-- node[dot][pos=1/5](y2z-left1){} node[dot][pos=2/5](y2z-left2){} node[dot][pos=3/5](y2z-left3){} node[dot][pos=4/5](y2z-left4){} (left);

\draw (z)-- node[dot][pos=1/5](y2z-right1){} node[dot][pos=2/5](y2z-right2){} node[dot][pos=3/5](y2z-right3){} node[dot][pos=4/5](y2z-right4){} (right);

\path (left)-- node[font=\LARGE]{\( \dots \)} (y2z-1of3);
\path (right)-- node[font=\LARGE]{\( \dots \)} (y2z-2of3);

\draw (y2z-left1)--+(-90:1) node[dot]{};
\draw (y2z-left2)--+(-90:0.5) node[dot]{}
      (y2z-left3)--+(-90:0.5) node[dot]{}
      (y2z-left4)--+(-90:0.5) node[dot]{};

\draw (y2z-right1)--+(-90:1) node[dot]{};
\draw (y2z-right2)--+(-90:0.5) node[dot]{}
      (y2z-right3)--+(-90:0.5) node[dot]{}
      (y2z-right4)--+(-90:0.5) node[dot]{};

\draw (y2z-mid1)--+(-90:0.5) node[dot]{}
      (y2z-mid2)--+(-90:0.5) node[dot]{}
      (y2z-mid3)--+(-90:0.5) node[dot]{}
      (y2z-mid7)--+(-90:0.5) node[dot]{}
      (y2z-mid8)--+(-90:0.5) node[dot]{}
      (y2z-mid9)--+(-90:0.5) node[dot]{};

\path (y2z-left2)--+(-0.15,0) coordinate(brace1Start);
\path (y2z-mid3)--+(0.15,0) coordinate(brace1End);
\draw [decorate,decoration={brace,amplitude=6pt,raise=3pt}] (brace1Start)--  node[yshift=5mm]{\( 3s \)} (brace1End);

\path (x2y-right2)--+(-120:0.15) coordinate(brace2Start);
\path (x2y-mid7)--+(60:0.15) coordinate(brace2End);
\draw [decorate,decoration={brace,amplitude=6pt,raise=3pt,mirror}] (brace2Start)-- node[xshift=5mm,yshift=-3mm]{\( 3s \)}  (brace2End);
\end{tikzpicture}
\caption{New gadget component.}
\label{fig:new component}
\end{figure}

The graph produced by this modification has girth \( 6s+8 \). 
To prove that the new construction preserves the guarantees of Construction~\ref{make:3-star}, it suffices to show that the new gadget component admits a \textit{unique} 3-star colouring similar to the \textit{unique} 3-star colouring of the old gadget component. 
Lemma~\ref{lem:3-star large girth unique colouring} below proves this for \( s=1 \); the proof is similar for higher values of \( s \). 

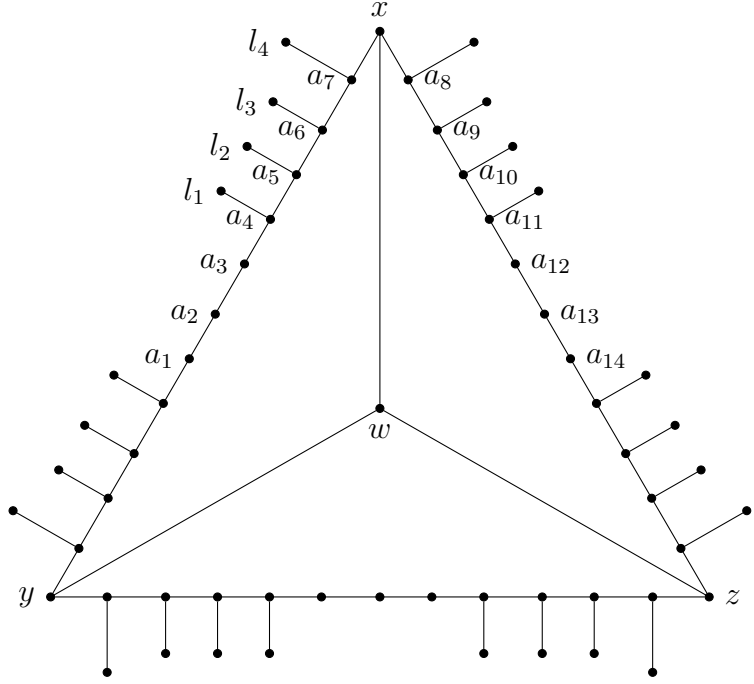
\begin{figure}[hbt]
\centering
\begin{tikzpicture}
\node[dot] (w)[label=below:\( w \)]{};
\draw (w)--+(90:5) node[dot](x)[label=above:\( x \)]{}
      (w)--+(-30:5) node[dot](z)[label=right:\( z \)]{}
      (w)--+(-150:5) node[dot](y)[label=left:\( y \)]{};

\draw (x)-- node[dot][pos=0.08](x2y-1)[label=left:\( a_7 \)]{} node[dot][pos=0.17](x2y-2)[label=left:\( a_6 \)]{} node[dot][pos=0.25](x2y-3)[label=left:\( a_5 \)]{} node[dot][pos=0.33](x2y-4)[label=left:\( a_4 \)]{} node[dot][pos=0.41](x2y-5)[label=left:\( a_3 \)]{} node[dot][pos=0.5](x2y-6)[label=left:\( a_2 \)]{} node[dot][pos=0.58](x2y-7)[label=left:\( a_1 \)]{} node[dot][pos=0.66](x2y-8){} node[dot][pos=0.75](x2y-9){} node[dot][pos=0.83](x2y-10){} node[dot][pos=0.92](x2y-11){}  (y);
\draw (x2y-1)--+(150:1) node[dot][label=left:\( l_4 \)]{}
      (x2y-11)--+(150:1) node[dot]{};
\draw (x2y-2)--+(150:0.75) node[dot][label=left:\( l_3 \)]{}
      (x2y-3)--+(150:0.75) node[dot][label=left:\( l_2 \)]{}
      (x2y-4)--+(150:0.75) node[dot][label=left:\( l_1 \)]{};
\draw (x2y-8)--+(150:0.75) node[dot]{}
      (x2y-9)--+(150:0.75) node[dot]{}
      (x2y-10)--+(150:0.75) node[dot]{};

\draw (x)-- node[dot][pos=0.08](x2z-1)[label=right:\( a_8 \)]{} node[dot][pos=0.17](x2z-2)[label=right:\( a_9 \)]{} node[dot][pos=0.25](x2z-3)[label=right:\( a_{10} \)]{} node[dot][pos=0.33](x2z-4)[label=right:\( a_{11} \)]{} node[dot][pos=0.41](x2z-5)[label=right:\( a_{12} \)]{} node[dot][pos=0.5](x2z-6)[label=right:\( a_{13} \)]{} node[dot][pos=0.58](x2z-7)[label=right:\( a_{14} \)]{} node[dot][pos=0.66](x2z-8){} node[dot][pos=0.75](x2z-9){} node[dot][pos=0.83](x2z-10){} node[dot][pos=0.92](x2z-11){}  (z);
\draw (x2z-1)--+(30:1) node[dot]{}
      (x2z-11)--+(30:1) node[dot]{};
\draw (x2z-2)--+(30:0.75) node[dot]{}
      (x2z-3)--+(30:0.75) node[dot]{}
      (x2z-4)--+(30:0.75) node[dot]{};
\draw (x2z-8)--+(30:0.75) node[dot]{}
      (x2z-9)--+(30:0.75) node[dot]{}
      (x2z-10)--+(30:0.75) node[dot]{};

\draw (y)-- node[dot][pos=0.08](y2z-1){} node[dot][pos=0.17](y2z-2){} node[dot][pos=0.25](y2z-3){} node[dot][pos=0.33](y2z-4){} node[dot][pos=0.41](y2z-5){} node[dot][pos=0.5](y2z-6){} node[dot][pos=0.58](y2z-7){} node[dot][pos=0.66](y2z-8){} node[dot][pos=0.75](y2z-9){} node[dot][pos=0.83](y2z-10){} node[dot][pos=0.92](y2z-11){}  (z);
\draw (y2z-1)--+(-90:1) node[dot]{}
      (y2z-11)--+(-90:1) node[dot]{};
\draw (y2z-2)--+(-90:0.75) node[dot]{}
      (y2z-3)--+(-90:0.75) node[dot]{}
      (y2z-4)--+(-90:0.75) node[dot]{};
\draw (y2z-8)--+(-90:0.75) node[dot]{}
      (y2z-9)--+(-90:0.75) node[dot]{}
      (y2z-10)--+(-90:0.75) node[dot]{};

\end{tikzpicture}
\caption{New gadget component when \( s=1 \).}
\label{fig:new component for s=1}
\end{figure}


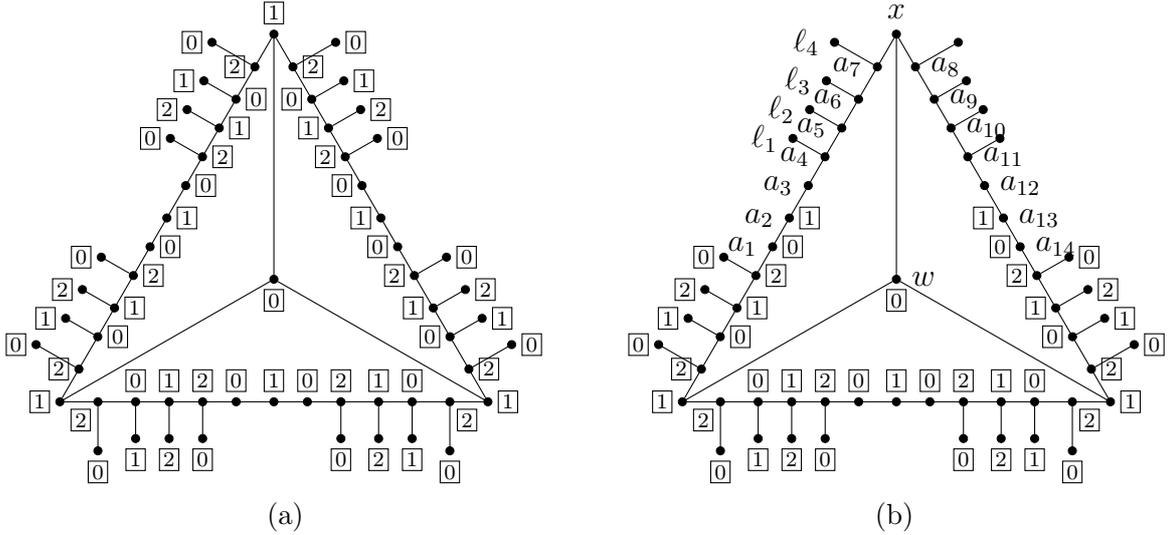
\begin{figure}[hbt]
\centering
\begin{subfigure}[b]{0.48\textwidth}
\begin{tikzpicture}[scale=0.65]
\node[dot] (w)[label={[vcolour]below:0}]{};
\draw (w)--+(90:5) node[dot](x)[label={[vcolour]above:1}]{}
      (w)--+(-30:5) node[dot](z)[label={[vcolour]right:1}]{}
      (w)--+(-150:5) node[dot](y)[label={[vcolour]left:1}]{};

\draw (x)-- node[dot][pos=0.08](x2y-1)[label={[vcolour]left:2}]{} node[dot][pos=0.17](x2y-2)[label={[vcolour]right:0}]{} node[dot][pos=0.25](x2y-3)[label={[vcolour]right:1}]{} node[dot][pos=0.33](x2y-4)[label={[vcolour]right:2}]{} node[dot][pos=0.41](x2y-5)[label={[vcolour]right:0}]{} node[dot][pos=0.5](x2y-6)[label={[vcolour]right:1}]{} node[dot][pos=0.58](x2y-7)[label={[vcolour]right:0}]{} node[dot][pos=0.66](x2y-8)[label={[vcolour]right:2}]{} node[dot][pos=0.75](x2y-9)[label={[vcolour]right:1}]{} node[dot][pos=0.83](x2y-10)[label={[vcolour]right:0}]{} node[dot][pos=0.92](x2y-11)[label={[vcolour]left:2}]{}  (y);
\draw (x2y-1)--+(150:1) node[dot][label={[vcolour]left:0}]{}
      (x2y-11)--+(150:1) node[dot][label={[vcolour]left:0}]{};
\draw (x2y-2)--+(150:0.75) node[dot][label={[vcolour]left:1}]{}
      (x2y-3)--+(150:0.75) node[dot][label={[vcolour]left:2}]{}
      (x2y-4)--+(150:0.75) node[dot][label={[vcolour]left:0}]{};
\draw (x2y-8)--+(150:0.75) node[dot][label={[vcolour]left:0}]{}
      (x2y-9)--+(150:0.75) node[dot][label={[vcolour]left:2}]{}
      (x2y-10)--+(150:0.75) node[dot][label={[vcolour]left:1}]{};

\draw (x)-- node[dot][pos=0.08](x2z-1)[label={[vcolour]right:2}]{} node[dot][pos=0.17](x2z-2)[label={[vcolour]left:0}]{} node[dot][pos=0.25](x2z-3)[label={[vcolour]left:1}]{} node[dot][pos=0.33](x2z-4)[label={[vcolour]left:2}]{} node[dot][pos=0.41](x2z-5)[label={[vcolour]left:0}]{} node[dot][pos=0.5](x2z-6)[label={[vcolour]left:1}]{} node[dot][pos=0.58](x2z-7)[label={[vcolour]left:0}]{} node[dot][pos=0.66](x2z-8)[label={[vcolour]left:2}]{} node[dot][pos=0.75](x2z-9)[label={[vcolour]left:1}]{} node[dot][pos=0.83](x2z-10)[label={[vcolour]left:0}]{} node[dot][pos=0.92](x2z-11)[label={[vcolour]right:2}]{}  (z);
\draw (x2z-1)--+(30:1) node[dot][label={[vcolour]right:0}]{}
      (x2z-11)--+(30:1) node[dot][label={[vcolour]right:0}]{};
\draw (x2z-2)--+(30:0.75) node[dot][label={[vcolour]right:1}]{}
      (x2z-3)--+(30:0.75) node[dot][label={[vcolour]right:2}]{}
      (x2z-4)--+(30:0.75) node[dot][label={[vcolour]right:0}]{};
\draw (x2z-8)--+(30:0.75) node[dot][label={[vcolour]right:0}]{}
      (x2z-9)--+(30:0.75) node[dot][label={[vcolour]right:2}]{}
      (x2z-10)--+(30:0.75) node[dot][label={[vcolour]right:1}]{};

\draw (y)-- node[dot][pos=0.08](y2z-1)[label={[vcolour]below left:2}]{} node[dot][pos=0.17](y2z-2)[label={[vcolour]above:0}]{} node[dot][pos=0.25](y2z-3)[label={[vcolour]above:1}]{} node[dot][pos=0.33](y2z-4)[label={[vcolour]above:2}]{} node[dot][pos=0.41](y2z-5)[label={[vcolour]above:0}]{} node[dot][pos=0.5](y2z-6)[label={[vcolour]above:1}]{} node[dot][pos=0.58](y2z-7)[label={[vcolour]above:0}]{} node[dot][pos=0.66](y2z-8)[label={[vcolour]above:2}]{} node[dot][pos=0.75](y2z-9)[label={[vcolour]above:1}]{} node[dot][pos=0.83](y2z-10)[label={[vcolour]above:0}]{} node[dot][pos=0.92](y2z-11)[label={[vcolour]below right:2}]{}  (z);
\draw (y2z-1)--+(-90:1) node[dot][label={[vcolour]below:0}]{}
      (y2z-11)--+(-90:1) node[dot][label={[vcolour]below:0}]{};
\draw (y2z-2)--+(-90:0.75) node[dot][label={[vcolour]below:1}]{}
      (y2z-3)--+(-90:0.75) node[dot][label={[vcolour]below:2}]{}
      (y2z-4)--+(-90:0.75) node[dot][label={[vcolour]below:0}]{};
\draw (y2z-8)--+(-90:0.75) node[dot][label={[vcolour]below:0}]{}
      (y2z-9)--+(-90:0.75) node[dot][label={[vcolour]below:2}]{}
      (y2z-10)--+(-90:0.75) node[dot][label={[vcolour]below:1}]{};

\end{tikzpicture}
\caption{}
\label{fig:3-star colouring when s=1}
\end{subfigure}%
\hfill
\begin{subfigure}[b]{0.48\textwidth}
\centering
\begin{tikzpicture}[scale=0.65]
\node[dot] (w)[label=right:\( w \)][label={[vcolour]below:0}]{};
\draw (w)--+(90:5) node[dot](x)[label=above:\( x \)]{}
      (w)--+(-30:5) node[dot](z)[label={[vcolour]right:1}]{}
      (w)--+(-150:5) node[dot](y)[label={[vcolour]left:1}]{};

\draw (x)-- node[dot][pos=0.08](x2y-1)[label=left:\( a_7 \)]{} node[dot][pos=0.17](x2y-2)[label=left:\( a_6 \)]{} node[dot][pos=0.25](x2y-3)[label=left:\( a_5 \)]{} node[dot][pos=0.33](x2y-4)[label=left:\( a_4 \)]{} node[dot][pos=0.41](x2y-5)[label=left:\( a_3 \)]{} node[dot][pos=0.5](x2y-6)[label=left:\( a_2 \)][label={[vcolour]right:1}]{} node[dot][pos=0.58](x2y-7)[label=left:\( a_1 \)][label={[vcolour]right:0}]{} node[dot][pos=0.66](x2y-8)[label={[vcolour]right:2}]{} node[dot][pos=0.75](x2y-9)[label={[vcolour]right:1}]{} node[dot][pos=0.83](x2y-10)[label={[vcolour]right:0}]{} node[dot][pos=0.92](x2y-11)[label={[vcolour]left:2}]{}  (y);
\draw (x2y-1)--+(150:1) node[dot][label=left:\( \ell_4 \)]{}
      (x2y-11)--+(150:1) node[dot][label={[vcolour]left:0}]{};
\draw (x2y-2)--+(150:0.75) node[dot][label=left:\( \ell_3 \)]{}
      (x2y-3)--+(150:0.75) node[dot][label=left:\( \ell_2 \)]{}
      (x2y-4)--+(150:0.75) node[dot][label=left:\( \ell_1 \)]{};
\draw (x2y-8)--+(150:0.75) node[dot][label={[vcolour]left:0}]{}
      (x2y-9)--+(150:0.75) node[dot][label={[vcolour]left:2}]{}
      (x2y-10)--+(150:0.75) node[dot][label={[vcolour]left:1}]{};

\draw (x)-- node[dot][pos=0.08](x2z-1)[label=right:\( a_8 \)]{} node[dot][pos=0.17](x2z-2)[label=right:\( a_9 \)]{} node[dot][pos=0.25](x2z-3)[label=right:\( a_{10} \)]{} node[dot][pos=0.33](x2z-4)[label=right:\( a_{11} \)]{} node[dot][pos=0.41](x2z-5)[label=right:\( a_{12} \)]{} node[dot][pos=0.5](x2z-6)[label={[vcolour]left:1}][label=right:\( a_{13} \)]{} node[dot][pos=0.58](x2z-7)[label={[vcolour]left:0}][label=right:\( a_{14} \)]{} node[dot][pos=0.66](x2z-8)[label={[vcolour]left:2}]{} node[dot][pos=0.75](x2z-9)[label={[vcolour]left:1}]{} node[dot][pos=0.83](x2z-10)[label={[vcolour]left:0}]{} node[dot][pos=0.92](x2z-11)[label={[vcolour]right:2}]{}  (z);
\draw (x2z-1)--+(30:1) node[dot]{}
      (x2z-11)--+(30:1) node[dot][label={[vcolour]right:0}]{};
\draw (x2z-2)--+(30:0.75) node[dot]{}
      (x2z-3)--+(30:0.75) node[dot]{}
      (x2z-4)--+(30:0.75) node[dot]{};
\draw (x2z-8)--+(30:0.75) node[dot][label={[vcolour]right:0}]{}
      (x2z-9)--+(30:0.75) node[dot][label={[vcolour]right:2}]{}
      (x2z-10)--+(30:0.75) node[dot][label={[vcolour]right:1}]{};

\draw (y)-- node[dot][pos=0.08](y2z-1)[label={[vcolour]below left:2}]{} node[dot][pos=0.17](y2z-2)[label={[vcolour]above:0}]{} node[dot][pos=0.25](y2z-3)[label={[vcolour]above:1}]{} node[dot][pos=0.33](y2z-4)[label={[vcolour]above:2}]{} node[dot][pos=0.41](y2z-5)[label={[vcolour]above:0}]{} node[dot][pos=0.5](y2z-6)[label={[vcolour]above:1}]{} node[dot][pos=0.58](y2z-7)[label={[vcolour]above:0}]{} node[dot][pos=0.66](y2z-8)[label={[vcolour]above:2}]{} node[dot][pos=0.75](y2z-9)[label={[vcolour]above:1}]{} node[dot][pos=0.83](y2z-10)[label={[vcolour]above:0}]{} node[dot][pos=0.92](y2z-11)[label={[vcolour]below right:2}]{}  (z);
\draw (y2z-1)--+(-90:1) node[dot][label={[vcolour]below:0}]{}
      (y2z-11)--+(-90:1) node[dot][label={[vcolour]below:0}]{};
\draw (y2z-2)--+(-90:0.75) node[dot][label={[vcolour]below:1}]{}
      (y2z-3)--+(-90:0.75) node[dot][label={[vcolour]below:2}]{}
      (y2z-4)--+(-90:0.75) node[dot][label={[vcolour]below:0}]{};
\draw (y2z-8)--+(-90:0.75) node[dot][label={[vcolour]below:0}]{}
      (y2z-9)--+(-90:0.75) node[dot][label={[vcolour]below:2}]{}
      (y2z-10)--+(-90:0.75) node[dot][label={[vcolour]below:1}]{};

\end{tikzpicture}
\caption{}
\label{fig:colours forced on new component}
\end{subfigure}
\captionsetup{width=0.85\textwidth}
\caption{(a)~A 3-star colouring of new gadget component (\( s=1 \)),\\
and (b)~Colours forced in the new gadget component (\( s=1 \)).}
\end{figure}

\begin{lemma}\label{lem:3-star large girth unique colouring}
For \( s=1 \), the colouring shown in Figure~\ref{fig:3-star colouring when s=1} is the \textit{unique} 3-star colouring of the new gadget component (see Figure~\ref{fig:new component for s=1} for the new gadget component). 
\end{lemma}
\begin{proof}
Suppose \( f \) is a 3-star colouring of the new gadget component. 
%
Without loss of generality, assume that \( f(w)=0 \). 
First, we prove that \( f \) must use the same colour on vertices \( x,y \) and \( z \). 
Since only colours 1 and 2 are available for vertices \( x,y \) and \( z \), at least two of them should get the same colour. 
Without loss of generality, assume that \( f(y)=f(z)=1 \). 
Since \( y,w,z \) is a bicoloured \( P_3 \), neighbours of \( y \) on the outer cycle must be coloured 2 due to Observation~\ref{obs:P3 force colours}. 
Repeating application of Observation~\ref{obs:P3 force colours} reveals that colours are forced on the new gadget component as shown in Figure~\ref{fig:colours forced on new component}.

We claim that \( f(x)=1 \). To produce a contradiction, assume that \( f(x)=2 \). Clearly, \( f(a_7)=0 \) or 1. First, we show that \( f(a_7)=0 \) leads to a contradiction. Assume that \( f(a_7)=0 \). Since \( w,x,a_7 \) is a bicoloured \( P_3 \), \( f(a_6)=1 \) and \( f(l_4)=1 \) (by Observation~\ref{obs:P3 force colours}). Since \( l_4,a_7,a_6 \) is a bicoloured \( P_3 \), \( f(a_5)=2 \) and \( f(l_3)=2 \). Since \( l_3,a_6,a_5 \) is a bicoloured \( P_3 \), \( f(a_4)=0 \) and \( f(l_2)=0 \). Vertex \( a_3 \) can get only colour 2. So, \( a_3,a_4,a_5,l_2 \) is a bicoloured \( P_4 \). This contradiction proves that \( f(a_7)\neq 0 \).

Therefore, \( f(a_7)=1 \). By symmetry, \( f(a_8)=1 \) as well. Since \( a_8,x,a_7 \) is a bicoloured \( P_3 \), \( f(a_6)=0 \) and \( f(l_4)=0 \). Since \( l_4,a_7,a_6 \) is a bicoloured \( P_3 \), \( f(a_5)=2 \) and \( f(l_3)=2 \). Since \( l_3,a_6,a_5 \) is a bicoloured \( P_3 \), \( f(a_4)=1 \) and \( f(l_2)=1 \). Since \( l_2,a_5,a_4 \) is a bicoloured \( P_3 \), \( f(a_3)=0 \). But, then \( a_4,a_3,a_2,a_1 \) is a bicoloured \( P_4 \). This contradiction proves that \( f(x)\neq 2 \). Hence, \( f(x)=1 \). 
Therefore, by symmetry, the colouring shown in Figure~\ref{fig:3-star colouring when s=1} is the \textit{unique} 3-star colouring of the new gadget component. 
\end{proof}

The new vertex gadget (resp. edge gadget) is constructed from its old counterpart by replacing the old gadget component by the new gadget component (new gadgets are displayed as Figure~3 and Figure~4 in the supplementary material). 
It is easy to see that the new vertex gadget and the new edge gadget preserve the following properties of their old counterparts: (i)~for every 3-star colouring of the vertex gadget, its terminals should get the same colour, (ii)~for every 3-star colouring of the edge gadget, its terminals should get different colours, (iii)~there exist a 3-star colouring of the vertex gadget (resp. edge gadget) such that each \( P_3 \) in it containing a terminal is tricoloured. 

Thus, we have the following theorems.


\begin{theorem}\label{thm:3-star planar bip girth g}
Let \( g\geq 8 \). The problem \textsc{3-Star Colourability}\( ( \)planar, bipartite, \( \Delta=3 \), girth~\( \geq g \)\( ) \) is NP-complete, and it does not admit a \( 2^{o(\sqrt{n})} \)-time algorithm unless ETH fails. Moreover, the problem \textsc{3-Star Colourability}\( ( \)bipartite, \( \Delta=3 \), girth~\( \geq g \)\( ) \) does not admit a \( 2^{o(n)} \)-time algorithm unless ETH fails. 
\qed
\end{theorem}
\begin{theorem}
For all \( g\geq 8 \), it is \#P-complete to count the number of \( 3 \)-star colourings of a bipartite graph of maximum degree three and girth at least \( g \).
\qed
\end{theorem}




It is known that for all \( k\geq 3 \), testing whether a graph has a \textit{unique} \( k \)-colouring is coNP-hard. 
As far as we know, there is no hardness result on \textit{unique} star colouring. 
We show that it is coNP-hard to check whether a graph has a \textit{unique} 3-star colouring. 
The decision problem \textsc{Unique \( k \)-Colouring} takes a graph \( G \) as input and asks whether \( G \) has a \textit{unique} \( k \)-colouring. 
The problem \textsc{Unique \( k \)-Star Colouring} is defined likewise. 
The problem \textsc{Another \( k \)-Colouring} is closely related to the problem \textsc{Unique \( k \)-Colouring}. 
The problem \textsc{Another \( k \)-Colouring} takes a graph \( G \) and a \( k \)-colouring \( f_1 \) of \( G \) as input and asks whether \( G \) admits a \( k \)-colouring \( f_2 \) of \( G \) which cannot be obtained from \( f_1 \) by merely swapping colours. 
The problem \textsc{Another \( k \)-Star Colouring} is defined likewise. 
Let \( k\geq 3 \). 
Dailey~\cite{dailey} established a reduction from \textsc{\( k \)-Colourability} to \textsc{Another \( k \)-Colouring}. 
So, \textsc{Another \( k \)-Colouring} is NP-hard. 
That is, given a graph \( G \) and a \( k \)-colouring of \( G \), it is NP-hard to test whether \( G \) has \textit{another} \( k \)-colouring (i.e., \( G \) is a no instance of \textsc{Unique \( k \)-Colouring}). 
Hence, given a (\( k \)-colourable) graph \( G \), it is coNP-hard to check whether \( G \) has a \textit{unique} \( k \)-colouring. 
Therefore, \textsc{Unique \( k \)-Colouring} is coNP-hard even when restricted to the class of \( k \)-colourable graphs.

It is easy to observe that Dailey’s construction provides a reduction from the problem \textsc{3-Colourability}(\( \Delta=4 \)) to \textsc{Another 3-Colouring}(\( \Delta=8 \)). 
So, \textsc{Another 3-Colouring}(\( \Delta=8 \)) is NP-complete. 
We prove that \textsc{Another 3-Star Colouring} is NP-complete for 2-degenerate bipartite graphs of maximum degree eight and arbitrarily large girth. 
As a result, \textsc{Unique 3-Star Colouring} is coNP-hard for the same class.

Using the construction of Coleman and Mor\'{e}~\cite{coleman_more}, one can show that \textsc{Another 3-Star Colouring} is NP-complete for 2-degenerate bipartite graphs of maximum degree twenty-four. 
We utilize ideas from Construction~\ref{make:3-star} to reduce the degree bound to eight. 

\begin{construct}\label{make:3-star R_swap max degree 8}
\emph{Input:} A graph \( G \) of maximum degree eight.\\
\emph{Output:} A 2-degenerate bipartite graph \( G' \) of maximum degree eight and girth eight.\\
\emph{Guarantee:} The number of \( 3 \)-colourings of \( G \) up to colour swaps equals the number of \( 3 \)-star colourings of \( G' \) up to colour swaps.\\
\emph{Steps:}\\
Replace each edge \( e=uv \) of \( G \) by an edge gadget as shown in Figure~\ref{fig:3-star R_swap max degree 8}.

Clearly, \( G' \) is bipartite, \( \Delta(G')=8 \) and \( \text{girth}(G')=8 \). 
Also, \( G' \) is 2-degenerate because we can remove all vertices from \( G' \) by repeatedly removing vertices of degree one or two. 
\begin{figure}[hbt]
\centering
\begin{tikzpicture}
\draw (0,0) node[bigdot](uk)[label=below:\( u \)]{} --++(1,0) node[bigdot](vl)[label=below:\( v \)]{};
\path (vl) --++(1,0) coordinate(from) --++(0.6,0) coordinate(to);
\draw [-stealth,draw=black,line width=3pt] (from)--(to);

\path (to)--+(4.5,0) node[bigdot] (w){};
\draw (w)--+(90:2.5) node[dot] (x){}
      (w)--+(-30:2.5) node[dot] (z){}
      (w)--+(-150:2.5) node[dot] (y){};
\draw (x)--node[bigdot][pos=0.17](afterXtoY){} node[dot][pos=0.33]{} node[bigdot][pos=0.5]{} node[dot][pos=0.66]{} node[bigdot][pos=0.83](afterYtoX){} (y);
\draw (x)--node[bigdot][pos=0.17](afterXtoZ){} node[dot][pos=0.33]{} node[bigdot][pos=0.5]{} node[dot][pos=0.66]{} node[bigdot][pos=0.83](afterZtoX){} (z);
\draw (y)--node[bigdot][pos=0.17](afterYtoZ){} node[dot][pos=0.33]{} node[bigdot][pos=0.5]{} node[dot][pos=0.66]{} node[bigdot][pos=0.83](afterZtoY){} (z);
\draw (afterXtoY)--+(150:0.5) node[dot]{}
      (afterYtoX)--++(150:0.5) node[dot]{}--+(150:0.5) node[bigdot]{} node[terminal][label={left:\( u \)}]{};
\draw (afterXtoZ)--+(30:0.5) node[dot]{}
      (afterZtoX)--++(30:0.5) node[dot](3vertex1){} --++(30:0.5) node[bigdot](3vertex2){}--++(30:0.5) node[dot]{}--++(30:0.5) node[bigdot]{} node[terminal][label={right:\( v \)}]{};
\draw (3vertex1)--+(120:0.5) node[bigdot]{}
      (3vertex2)--+(120:0.5) node[dot]{};
\draw (afterYtoZ)--+(-90:0.5) node[dot]{}
      (afterZtoY)--+(-90:0.5) node[dot]{};

\end{tikzpicture}
\caption{Replacement of edge by edge gadget in Construction~\ref{make:3-star R_swap max degree 8}.}
\label{fig:3-star R_swap max degree 8}
\end{figure}
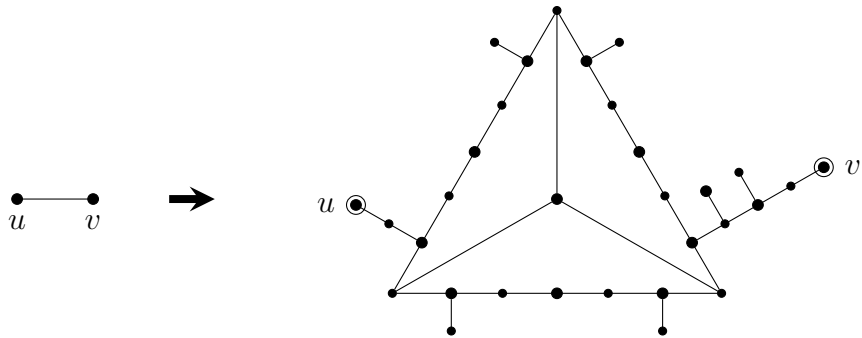
\end{construct}
\begin{proof}[Proof of Guarantee]
Note that the edge gadget is identical to the edge gadget in Construction~\ref{make:3-star}. 
%
%
In Construction~\ref{make:3-star}, the following properties of the edge gadget are proved: (i)~for every 3-star colouring of the gadget, the terminals get different colours, and every 3-vertex path containing a terminal is tricoloured, and (ii)~the edge gadget has a \textit{unique} 3-star colouring, namely the scheme in Figure~\ref{fig:3-star colouring edge gadget} (the scheme is repeated here as Figure~\ref{fig:re-3-star colouring edge gadget} for convenience).\\
\begin{figure}[hbt]
\centering
\begin{tikzpicture}
\node[dot] (w)[label={[vcolour]above right:0}]{};
\draw (w)--+(90:2.5) node[dot] (x)[label={[vcolour]right:1}]{}
      (w)--+(-30:2.5) node[dot] (z)[label={[vcolour]below:1}]{}
      (w)--+(-150:2.5) node[dot] (y)[label={[vcolour]below:1}]{};
\draw (x)--node[dot][pos=0.17](afterXtoY)[label={[vcolour]right:2}]{} node[dot][pos=0.33][label={[vcolour]right:0}]{} node[dot][pos=0.5][label={[vcolour]right:1}]{} node[dot][pos=0.66][label={[vcolour]right:0}]{} node[dot][pos=0.83][label={[vcolour]right:2}](afterYtoX){} (y);
\draw (x)--node[dot][pos=0.17](afterXtoZ)[label={[vcolour]left:2}]{} node[dot][pos=0.33][label={[vcolour]left:0}]{} node[dot][pos=0.5][label={[vcolour]left:1}]{} node[dot][pos=0.66][label={[vcolour]left:0}]{} node[dot][pos=0.83][label={[vcolour]left:2}](afterZtoX){} (z);
\draw (y)--node[dot][pos=0.17][label={[vcolour]above:2}](afterYtoZ){} node[dot][pos=0.33][label={[vcolour]above:0}]{} node[dot][pos=0.5][label={[vcolour]above:1}]{} node[dot][pos=0.66][label={[vcolour]above:0}]{} node[dot][pos=0.83][label={[vcolour]above:2}](afterZtoY){} (z);

\draw (afterYtoZ)--+(-90:0.5) node[dot][label={[vcolour]below:0}]{}
      (afterZtoY)--+(-90:0.5) node[dot][label={[vcolour]below:0}]{};

\draw (afterXtoY)--+(150:0.5) node[dot][label={[vcolour]above:0}]{}
      (afterYtoX)--++(150:0.5) node[dot][label={[vcolour]above:0}]{}--+(150:0.5) node[dot]{} node[terminal][label={left:\( u \)}][label={[vcolour]above:1}]{};
\draw (afterXtoZ)--+(30:0.5) node[dot][label={[vcolour]above:0}]{}
      (afterZtoX)--++(30:0.5) node[dot](3vertex1)[label={[vcolour]below:0}]{} --++(30:0.5) node[dot](3vertex2)[label={[vcolour]below:1}]{}--++(30:0.5) node[dot][label={[vcolour]below:2}]{}--++(30:0.5) node[dot]{} node[terminal][label={right:\( v \)}][label={[vcolour]above:0}]{};
\draw (3vertex1)--+(120:0.5) node[dot][label={[vcolour]above:1}]{}
      (3vertex2)--+(120:0.5) node[dot][label={[vcolour]above:2}]{};
\end{tikzpicture}
\caption{The \textit{unique} 3-star colouring of the edge gadget.}
\label{fig:re-3-star colouring edge gadget}
\end{figure}

\noindent \textbf{Claim~1:} For distinct colours \( i,j\in \{0,1,2\} \), an edge gadget with terminals \( u \) and \( v \) has exactly one 3-star colouring \( f \) such that \( f(u)=i \) and \( f(v)=j \).\\[5pt] 
Let \( i,j\in\{0,1,2\} \) be distinct colours, and let \( c \) be the third colour (i.e., \( \{i,j,c\}=\{0,1,2\} \)). 
Using the scheme in Figure~\ref{fig:re-3-star colouring edge gadget} and swapping colour~1 with colour~\( i \) and colour~0 with colour~\( j \) gives a 3-star colouring \( f \) of the edge gadget with \( f(u)=i \) and \( f(v)=j \). 
We need to prove that \( f \) is the unique 3-star colouring of the edge gadget with this property. 
Suppose that \( f^* \) is a 3-star colouring of the edge gadget such that \( f^*(u)=i \) and \( f^*(v)=j \). 
We know that the edge gadget has a \textit{unique} 3-star colouring (i.e., unique up to colour swaps; see Claim~1 in proof of Guarantee~1 in Construction~\ref{make:3-star}). 
Hence, we can obtain \( f^* \) from \( f \) by merely swapping colours. 
Since \( f^*(u)=f(u)=i \), colour~\( i \) is not swapped when we go from \( f \) to \( f^* \); that is, every vertex of the gadget with colour~\( i \) under \( f \) gets colour \( i \) under \( f^* \) as well. 
Similarly, since \( f^*(v)=f(v)=j \), every vertex of the gadget with colour~\( j \) under \( f \) gets colour \( j \) under \( f^* \) as well. 
Since \( c \) is the only remaining colour, every vertex of the gadget with colour~\( c \) under \( f \) gets colour \( c \) under \( f^* \).  
Therefore, \( f^*=f \). 
This proves the uniqueness of \( f \) and thus Claim~1.\\ 

We define a function \( \phi \) from the set of 3-colourings of \( G \) to the set of 3-star colourings of \( G' \). 
Each 3-colouring \( f \) of \( G \), \( f:V(G)\to\{0,1,2\} \), can be extended into a 3-star colouring \( f' \) of \( G' \) by applying the 3-star colouring scheme shown in Figure~\ref{fig:re-3-star colouring edge gadget} on each edge gadget (colour swaps may be needed). 
Note that \( f' \) will be a 3-star colouring of \( G' \) because (i)~3-star colouring schemes are used on edge gadgets, and (ii)~every 3-vertex path containing a terminal is tricoloured by \( f' \). 
The extension of \( f \) into a 3-star colouring of \( G' \) is unique because the edge gadget has exactly one 3-star colouring once the colours on the terminals are fixed
 (see Claim~1). 
We define \( \phi \) as the function that maps each 3-colouring \( f \) of \( G \) to the unique 3-star colouring extension \( f' \) of \( f \) into \( V(G') \). 

Note that for every 3-colouring \( f\) of \( G \), the restriction of \( \phi(f) \) into \( V(G) \) is precisely \( f \) (i.e., \( \phi(f)_{\restriction V(G)}=f \)). 
The function \( \phi \) is one-one because \( \phi(f_1)=\phi(f_2) \) implies that \( f_1=\phi(f_1)_{\restriction V(G)}=\phi(f_2)_{\restriction V(G)}=f_2 \). 
We claim that \( \phi \) is onto. 
Let \( f' \) be a 3-star colouring of \( G' \).
Then, \( f'(u)\neq f'(v) \) whenever there is an edge gadget between terminals \( u \) and \( v \) in \( G' \) because \( f' \) assigns different colours to terminals of each edge gadget. 
For every edge \( uv \) of \( G \), there is an edge gadget between terminals \( u \) and \( v \) in \( G' \). 
Hence, \( f'(u)\neq f'(v) \) for every edge \( uv \) of \( G \). 
So, the restriction of \( f' \) into \( V(G) \) is a 3-colouring of \( G \). 
This proves that \( \phi \) is onto. 
So, \( \phi \) is a bijection from the set of 3-colourings of \( G \) to the set of 3-star colourings of \( G' \). 
If \( f_1 \) and \( f_2 \) are two colourings of the same graph and \( f_2 \) can be obtained from \( f_1 \) by merely swapping colours, then we say that \( f_1 \) and \( f_2 \) are \emph{equivalent under colour swaps}. 
If two 3-colourings \( f_1 \) and \( f_2 \) of \( G \) are equivalent under colour swaps, then their images \( \phi(f_1) \) and \( \phi(f_2) \) are equivalent under colour swaps because they are the unique extensions of \( f_1 \) and \( f_2 \) respectively as 3-star colourings of \( G' \). 
Also, if two 3-star colourings \( \phi(f_1) \) and \( \phi(f_2) \) are equivalent under colour swaps, then their pre-images \( f_1=\phi(f_1)_{\restriction V(G)} \) and \( f_2=\phi(f_2)_{\restriction V(G)} \) are equivalent under colour swaps. 
So, two 3-colourings \( f_1 \) and \( f_2 \) of \( G \) are non-equivalent under colour swaps if and only if \( \phi(f_1) \) and \( \phi(f_2) \) are non-equivalent under colour swaps. 
Therefore, the number of \( 3 \)-colourings of \( G \) up to colour swaps is equal to the number of \( 3 \)-star colourings of \( G' \) up to colour swaps. 
\end{proof}

Thanks to Construction~\ref{make:3-star R_swap max degree 8}, we have the following theorem.
\begin{theorem}\label{thm:3-star R_swap max degree 8}
For 2-degenerate bipartite graphs of maximum degree eight and girth eight, \textsc{Another 3-Star Colouring} is NP-complete and \textsc{Unique 3-Star Colouring} is coNP-hard.
\end{theorem}
\begin{proof}
%
The reduction is from \textsc{Another 3-Colouring}(\( \Delta=8 \)). 
Let \( (G,f) \) be an instance of the source problem. 
From \( G \), produce a graph \( G' \) by Construction~\ref{make:3-star R_swap max degree 8}. 
In Construction~\ref{make:3-star R_swap max degree 8}, it is established that there is a bijection \( \phi \) from the set of 3-colourings of \( G \) to the set of 3-star colourings of \( G' \). 
In particular, \( f'=\phi(f) \) is a 3-star colouring of \( G' \). 
Recall that if \( f_1 \) and \( f_2 \) are two colourings of the same graph and \( f_2 \) can be obtained from \( f_1 \) by merely swapping colours, then we say that \( f_1 \) and \( f_2 \) are \emph{equivalent under colour swaps}.\\

By guarantee in Construction~\ref{make:3-star R_swap max degree 8}, the number of 3-colourings of \( G \) up to colour swaps is equal to the number of 3-star colourings of \( G' \) up to colour swaps. 
In particular, we have the following.\\[5pt]
\noindent \textbf{Claim~1:} \( G \) has at least two 3-colourings non-equivalent under colour swaps if and only if \( G' \) has at least two 3-star colourings non-equivalent under colour swaps.\\

\noindent \textbf{Claim~2:} \( (G,f) \) is a yes instance of \textsc{Another 3-Colouring} if and only if \( (G',f') \) is a yes instance of \textsc{Another 3-Star Colouring}.\\[5pt] 
Suppose that \( (G,f) \) is a yes instance of \textsc{Another 3-Colouring}. 
Hence, \( G \) admits a 3-colouring not equivalent to \( f \). 
So, \( G \) has at least two 3-colourings non-equivalent under colour swaps. 
By Claim~1, this implies that \( G' \) has at least two 3-star colourings non-equivalent under colour swaps. 
Therefore, \( G' \) has a 3-star colouring not equivalent to \( f' \). 
That is, \( (G',f') \) is a yes instance of \textsc{Another 3-Star Colouring}.

Conversely, suppose that \( (G',f') \) is a yes instance of \textsc{Another 3-Star Colouring}. 
Hence, \( G' \) admits a 3-star colouring not equivalent to \( f' \). 
So, \( G' \) has at least two 3-star colourings non-equivalent under colour swaps. 
By Claim~1, this implies that \( G \) has at least two 3-colourings non-equivalent under colour swaps. 
Therefore, \( G \) has a 3-colouring not equivalent to \( f \). 
That is, \( (G,f) \) is a yes instance of \textsc{Another 3-Colouring}. 
This proves the converse part and thus Claim~2.\\

Thanks to Claim~2, we have a reduction from \textsc{Another 3-Colouring}(\( \Delta=8 \)) to \textsc{Another 3-Star Colouring}(2-degenerate, bipartite, \( \Delta(G)=8 \), \( \text{girth}(G)=8 \)). 
Therefore, \textsc{Another 3-Star Colouring} is NP-complete for 2-degenerate bipartite  graphs of maximum degree eight and girth eight, and thus \textsc{Unique 3-Star Colouring} is coNP-hard for the same class.
\end{proof}

Recall that Theorem~\ref{thm:3-star planar bip girth g} improved on the girth of \( G' \) in Construction~\ref{make:3-star} by replacing the gadget component by a new gadget component. Applying the same idea to Construction~\ref{make:3-star R_swap max degree 8} gives the following result. 
\begin{theorem}\label{thm:3-star R_swap max degree 8 large girth}
Let \( g\geq 8 \) be a fixed integer. For 2-degenerate bipartite graphs of maximum degree eight and girth at least \( g \), \textsc{Another 3-Star Colouring} is NP-complete and \textsc{Unique 3-Star Colouring} is coNP-hard.
\qed
\end{theorem}

\subsection{\( k \)-Star Colouring with \( k\geq 4 \)}\label{sec:hardness k-star k>=4}
For \( k\geq 4 \), it is known that \textsc{\( k \)-Star Colourability} is NP-complete for bipartite graphs \cite{coleman_more}. 
We prove that the problem remains NP-complete when further restricted to graphs of maximum degree \( k \). 
We employ Construction~\ref{make:k-star} to this end. 
The gadget component used in the construction is displayed in Figure~\ref{fig:k-star colouring gadget component}. 
The set \( W_1 \) is an independent set of cardinality \( k-2 \). 
Similarly,  \( W_2 \), \( W_3 \), \( U_1 \), \( U_2 \) and \( U_3 \) are independent sets of cardinality \( k-3 \). 
Also, for each \( j\in\{1,2,3\} \), every vertex in \( W_j \) is adjacent to vertices \( x_j,y_j,z_j \) and members of \( U_j \). 
In upcoming diagrams, the gadget component is drawn by the symbol in Figure~\ref{fig:gadget component symbol}. 
When \( k=5 \), the gadget component is as in Figure~\ref{fig:5-star colouring gadget component}. 
Note that the gadget component has maximum degree \( k \), and it is bipartite (small dots form one part and big dots form the other part; see Figure~\ref{fig:5-star colouring gadget component}).

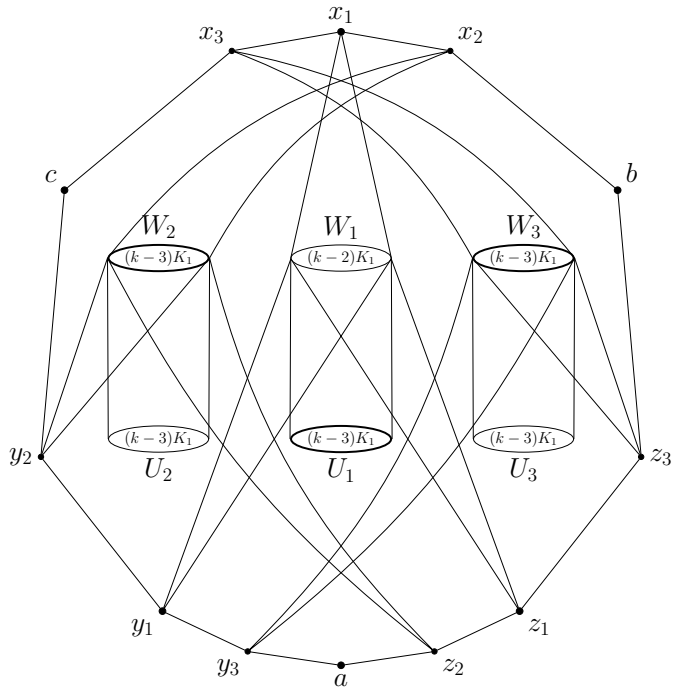
\begin{figure}[hbt]
\centering
\scalebox{0.5}{%
\tikzset{ 
dotplus/.style={bigdot,minimum size=4pt},
bigdotplus/.style={bigdot,minimum size=5pt}
}
\begin{tikzpicture}[scale=1.2]
\tikzstyle every label=[font=\LARGE]
\path (-90:7) node(a)[bigdotplus][label=below:\( a \)]{}
      ( 30:7) node(b)[bigdotplus][label=above right:\( b \)]{}
      (150:7) node(c)[bigdotplus][label=above left:\( c \)]{};
\path (90:2) node(W1)[ellipse,inner sep=1pt,draw][label=\( W_1 \)]{\( (k-2)K_1 \)}
      (-90:2) node(U1)[ellipse,inner sep=1pt,draw,line width=1.5pt][label=below:\( U_1 \)]{\( (k-3)K_1 \)};
\path (W1)+(-4,0) node(W2)[ellipse,inner sep=1pt,draw,line width=1.5pt][label=\( W_2 \)]{\( (k-3)K_1 \)}
      (U1)+(-4,0) node(U2)[ellipse,inner sep=1pt,draw][label=below:\( U_2 \)]{\( (k-3)K_1 \)};
\path (W1)+(4,0) node(W3)[ellipse,inner sep=1pt,draw,line width=1.5pt][label=\( W_3 \)]{\( (k-3)K_1 \)}
      (U1)+(4,0) node(U3)[ellipse,inner sep=1pt,draw][label=below:\( U_3 \)]{\( (k-3)K_1 \)};
\path ( 90:7) node(x)[bigdotplus][label=\( x_1 \)]{}
      (-124:7) node(y)[bigdotplus][label=below left:\( y_1 \)]{}
      (-56:7) node(z)[bigdotplus][label=below right:\( z_1 \)]{};
\draw (x)--(W1.180)  (x)--(W1.0)
      (y)--(W1.181)  (y)--(W1.-1)
      (z)--(W1.181)  (z)--(W1.-1);
\path (-107:7) node(y2)[dotplus][label=below left:\( y_3 \)]{}
      (- 73:7) node(z1)[dotplus][label=below right:\( z_2 \)]{};
\draw (y2) to[bend right=12] (W3.0)
      (y2) to[bend right=15] (W3.180);
\draw (z1) to[bend left=12] (W2.180)
      (z1) to[bend left=15] (W2.0);
\path (-160:7) node(y1)[dotplus][label=left:\( y_2 \)]{}
      ( -20:7) node(z2)[dotplus][label=right:\( z_3 \)]{};
\draw (y1)--(W2.180)  (y1)--(W2.0);
\draw (z2)--(W3.180)  (z2)--(W3.0);
\path (110:7) node(x2)[dotplus][label=above left:\( x_3 \)]{}
      ( 70:7) node(x1)[dotplus][label=above right:\( x_2 \)]{};
\draw (x2) to[bend left=20] (W3.180)
      (x2) to[bend left=20] (W3.0);
\draw (x1) to[bend right=20] (W2.180)
      (x1) to[bend right=20] (W2.0);
\draw (W1.0)--(U1.0)  (W1.180)--(U1.180)
      (W2.0)--(U2.0)  (W2.180)--(U2.180)
      (W3.0)--(U3.0)  (W3.180)--(U3.180);
\draw (x)--(x1)--(b)--(z2)--(z)--(z1)--(a)--(y2)--(y)--(y1)--(c)--(x2)--(x);
\end{tikzpicture}
}
\caption{The gadget component in Construction~\ref{make:k-star}. 
For each \( j\in \{1,2,3\} \), vertices in \( W_j \) are adjacent to \( x_j,y_j,z_j \) and members of \( U_j \).}
\label{fig:k-star colouring gadget component}
\end{figure}

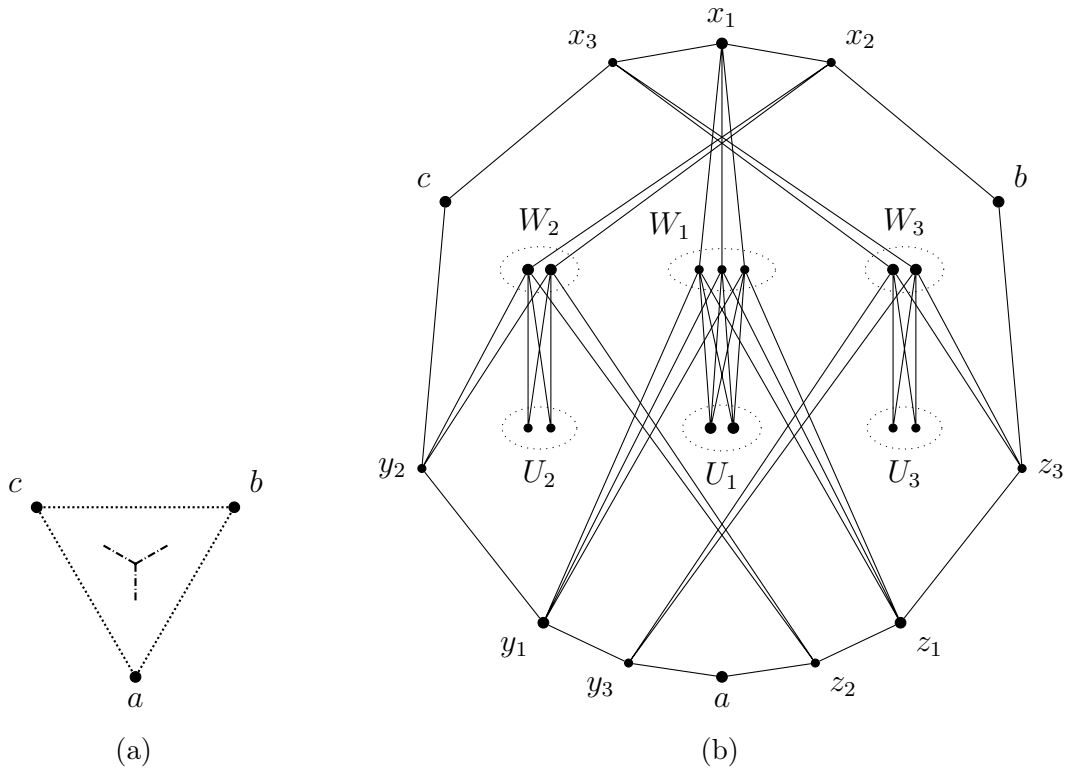
\begin{figure}[hbt]
\centering
\begin{subfigure}[b]{0.3\textwidth}
\centering
\begin{tikzpicture}
\path (-90:1.5) node(a)[bigdot][label=below:\( a \)]{}
      ( 30:1.5) node(b)[bigdot][label=above right:\( b \)]{}
      (150:1.5) node(c)[bigdot][label=above left:\( c \)]{};
\draw[densely dotted,thick] (a)--(b)--(c)--(a);
\draw[densely dash dot,thick]
(0,0)--(-90:0.5)
(0,0)--( 30:0.5)
(0,0)--(150:0.5);
\end{tikzpicture}
\caption{}
\label{fig:gadget component symbol}
\end{subfigure}%
\hfill
\begin{subfigure}[b]{0.7\textwidth}
\centering
\begin{tikzpicture}[scale=0.6]
\path (-90:7) node(a)[bigdot][label=below:\( a \)]{}
      ( 30:7) node(b)[bigdot][label=above right:\( b \)]{}
      (150:7) node(c)[bigdot][label=above left:\( c \)]{};
\path (90:2) node(w12)[dot]{}+(0.5,0) node(w13)[dot]{}
                             +(-0.5,0) node(w11)[dot]{}
      (-90:1.5) coordinate(u12)+(0.25,0) node(u13)[bigdot]{}
                             +(-0.25,0) node(u11)[bigdot]{};
\path (w12)++(-4,0) coordinate(w22)+(0.25,0) node(w23)[bigdot]{}
                                   +(-0.25,0) node(w21)[bigdot]{}
      (u12)++(-4,0) coordinate(u22)+(0.25,0) node(u23)[dot]{}
                                   +(-0.25,0) node(u21)[dot]{};
\path (w12)++(4,0) coordinate(w32)+(0.25,0) node(w33)[bigdot]{}
                                  +(-0.25,0) node(w31)[bigdot]{}
      (u12)++(4,0) coordinate(u32)+(0.25,0) node(u33)[dot]{}
                                  +(-0.25,0) node(u31)[dot]{};
\draw (u11)--(w11)   (u11)--(w12)   (u11)--(w13)
      (u13)--(w11)   (u13)--(w12)   (u13)--(w13);
\draw (u21)--(w21)    (u21)--(w23)
      (u23)--(w21)    (u23)--(w23);
\draw (u31)--(w31)    (u31)--(w33)
      (u33)--(w31)    (u33)--(w33);

\path ( 90:7) node(x)[bigdot][label=\( x_1 \)]{}
      (-124:7) node(y)[bigdot][label=below left:\( y_1 \)]{}
      (-56:7) node(z)[bigdot][label=below right:\( z_1 \)]{};
\draw (x)--(w11)   (x)--(w12)   (x)--(w13)  
      (y)--(w11)   (y)--(w12)   (y)--(w13)  
      (z)--(w11)   (z)--(w12)   (z)--(w13);
\path (-107:7) node(y2)[dot][label=below left:\( y_3 \)]{}
      (- 73:7) node(z1)[dot][label=below right:\( z_2 \)]{};
\draw (y2)--(w31)    (y2)--(w33);
\draw (z1)--(w21)    (z1)--(w23);
\path (-160:7) node(y1)[dot][label=left:\( y_2 \)]{}
      ( -20:7) node(z2)[dot][label=right:\( z_3 \)]{};
\draw (y1)--(w21)    (y1)--(w23);
\draw (z2)--(w31)    (z2)--(w33);
\path (110:7) node(x2)[dot][label=above left:\( x_3 \)]{}
      ( 70:7) node(x1)[dot][label=above right:\( x_2 \)]{};
\draw (x1)--(w21)    (x1)--(w23);
\draw (x2)--(w31)    (x2)--(w33);
\draw (x)--(x1)--(b)--(z2)--(z)--(z1)--(a)--(y2)--(y)--(y1)--(c)--(x2)--(x);

\node(W1) [fit={(w11) (w12) (w13)},draw,ellipse,dotted][label=above left:\( W_1 \)]{};
\node(W2) [fit={(w21) (w22) (w23)},draw,ellipse,dotted][label=\( W_2 \)]{};
\node(W3) [fit={(w31) (w32) (w33)},draw,ellipse,dotted][label=\( W_3 \)]{};
\node(U1) [fit={(u11) (u12) (u13)},draw,ellipse,dotted][label=below:\( U_1 \)]{};
\node(U2) [fit={(u21) (u22) (u23)},draw,ellipse,dotted][label=below:\( U_2 \)]{};
\node(U3) [fit={(u31) (u32) (u33)},draw,ellipse,dotted][label=below:\( U_3 \)]{};
\end{tikzpicture}
\caption{}
\label{fig:5-star colouring gadget component}
\end{subfigure}
\captionsetup{width=0.65\textwidth}
\caption{(a) The symbol for gadget component, and (b) the gadget component when \( k=5 \).}
\end{figure}

The following lemma shows the usefulness of the gadget component.
\begin{lemma}
Every \( k \)-star colouring of the gadget component must use the same colour on vertices \( a, b \) and \( c \).
\label{lem: k-star colouring gadget component}
\end{lemma}
\begin{proof}
We start with a simple claim.\\[5pt]
\textbf{Claim~1:} Let \( f \) is a \( k \)-star colouring of a graph \( G \), and let \( w' \) and \( w'' \) be two vertices in \( G \) with \( k \) common neighbours. Then, \( f(w')\neq f(w'') \).\\[5pt]
Suppose that \( v_1,v_2,\dots,v_k \) are \( k \) common neighbours of \( w' \) and \( w'' \) in \( G \). 
Clearly, the colour \( f(w') \) is unavailable for vertices \( v_1,v_2,\dots,v_k \). 
Since only \( k-1 \) colours are available for vertices \( v_1,v_2,\dots,v_k \), at least two of these vertices have the same colour, say, \( f(v_1)=f(v_2) \).  
Hence, we have \( f(w')\neq f(w'') \) (otherwise, \( w',v_1,w'',v_2 \) will be a bicoloured \( P_4 \); a contradiction). 
This proves Claim~1.\\ 

Let \( f \) be a \( k \)-star colouring of the gadget component. 
The following claim deals with \( W_1 \), \( W_2 \), and \( W_3 \).\\[5pt]
\noindent \textbf{Claim~2:} For each \( j\in\{1,2,3\} \), vertices in \( W_j \) get pair-wise distinct colours under \( f \).\\[5pt]
\noindent Since every pair of vertices in \( W_1 \) have \( k \) common neighbours (namely \( x_1,y_1,z_1 \) and members of \( U_1 \)), vertices in \( W_1 \) get pair-wise distinct colours under \( f \) by Claim~1. 
A~similar argument works for \( W_2 \) and \( W_3 \). 
This proves Claim~2.\\

Thanks to Claim~2, we may assume without loss of generality that vertices in \( W_1 \) are assigned a permutation of colours \( 0,1,\dots,k-3 \) (by \( f \)). 
Therefore, only colours \( k-2 \) and \( k-1 \) are available for vertices \( x_1,y_1 \) and \( z_1 \). 
Hence, at least two of vertices \( x_1,y_1,z_1 \) get the same colour. 
By symmetry, we assume without loss of generality  that \( f(y_1)=f(z_1)=k-2 \). 
For each colour \( i\in\{0,1,\dots,k-3\} \), there is a vertex \( w\in W_1 \) coloured~\( i \) so that \( y_1,w,z_1 \) is a 3-vertex path coloured \( k-2,i,k-2 \), and thus \( f(y_2)\neq i \) (if not, path \( y_2,y_1,w,z_1 \) is a bicoloured \( P_4 \)). 
Thus, \( f(y_2)\neq i \) for each \( i\in\{0,1,\dots,k-3\} \). 
Also, \( f(y_2)\neq k-2=f(y_1) \) since \( y_2 \) is adjacent to \( y_1 \). 
So, \( f(y_2)=k-1 \). 
Similarly, \( f(y_3)=k-1 \) and \( f(z_2)=f(z_3)=k-1 \). 
Since \( y_2,y_1,y_3 \) is a 3-vertex path coloured \( k-1,k-2,k-1 \) and \( y_2 \) is adjacent to every member of \( W_2 \), no member of \( W_2 \) is coloured \( k-2 \). 
Also, no member \( w \) of \( W_2 \) is coloured \( k-1 \) because \( w \) is adjacent to \( x_2 \) and \( f(x_2)=k-1 \). 
So, only colours \( 0,1,\dots,k-3 \) are available for vertices in \( W_2 \). 
Thanks to Claim~2 and the fact \( |W_2|=k-3 \), exactly one colour from \( \{0,1,\dots,k-3\} \) is missing in \( W_2 \). 
Without loss of generality, we may assume that the missing colour is \( 0 \); that is, vertices in \( W_2 \) are assigned a permutation of colours \( 1,2,\dots,k-3 \). 
Now, for each colour \( i\in\{1,2,\dots,k-3\} \), there is a vertex \( w\in W_2 \) coloured \( i \) so that \( y_2,w,z_2 \) is a 3-vertex path coloured \( k-1,i,k-1 \); as a result, \( c \) cannot be coloured \( i \) (otherwise, path \( c,y_2,w,z_2 \) is a bicoloured \( P_4 \)). 
Thus, \( f(c)\neq i \) for each \( i\in \{1,2,\dots,k-3\} \). 
Besides, \( f(c)\neq k-2 \) because \( y_2,y_1,y_3 \) is a 3-vertex path coloured \( k-1,k-2,k-1 \) and \( c \) is a neighbour of \( y_2 \) (if \( f(c)=k-2 \), then path \( c,y_2,y_1,y_3 \) is a bicoloured \( P_4 \)). 
Therefore, \( f(c)=0 \). 
By similar arguments, \( f(a)=0 \) as well (\( \because \) \( a \) is a neighbour of \( y_3 \)).

To complete the proof of the lemma, it suffices to show that \( f(b)=0 \). 
Consider the set \( W_3 \) and an arbitrary vertex \( w\in W_3 \). 
Vertex \( w \) is not coloured \( 0 \) since otherwise \( a,y_3,w,z_3 \) is a bicoloured \( P_4 \). 
Also, \( w \) is not coloured \( k-2 \) since \( z_2,z_1,z_3 \) is 3-vertex path coloured \( k-1,k-2,k-1 \) and \( w \) is adjacent to \( z_3 \). 
Thus, \( f(w)\neq 0 \) and \( f(w)\neq k-2 \). 
Also, \( f(w)\neq k-1=f(y_3) \) since \( y_3 \) is a neighbour of \( w \). 
Therefore, only colours \( 1,2,\dots,k-3 \) are available for vertices in \( W_3 \). 
Thanks to Claim~2 and the fact \( |W_3|=k-3 \), vertices in \( W_3 \) are assigned a permutation of colours \( 1,2,\dots,k-3 \). 
Hence, for each colour \( i\in\{1,2,\dots,k-3\} \), there is a 3-vertex path from \( y_3 \) to \( z_3 \) coloured \( k-1,i,k-1 \) so that vertex \( b \) cannot be coloured~\( i \) (if not, there is a 4-vertex path from \( y_3 \) to \( b \) coloured \( k-1,i,k-1,i \)). 
Thus, \( f(b)\neq i \) for each \( i\in\{1,2,\dots,k-3\} \). 
Moreover, \( b \) cannot be coloured~\( k-2 \) since \( z_2,z_1,z_3 \) is a 3-vertex path coloured \( k-1,k-2,k-1 \) and \( z_3 \) is adjacent to \( b \). 
Therefore, the only colour available at \( b \) is \( 0 \). 
This proves that \( f(a)=f(b)=f(c) \).
\end{proof}

\begin{construct}\label{make:k-star}
\emph{Parameter:} An integer \( k\geq 4 \).\\
\emph{Input:} A graph \( G \) of maximum degree \( 2(k-1) \).\\
\emph{Output:} A bipartite graph \( G' \) of maximum degree \( k \).\\
\emph{Guarantee 1:} \( G \) is \( k \)-colourable if and only if \( G' \) is \( k \)-star colourable.\\
\emph{Guarantee 2:} \( G' \) has only \( O(n) \) vertices where \( n=|V(G)| \).\\
\emph{Steps:}\\
Replace each vertex \( v \) of \( G \) by a vertex gadget as shown in Figure~\ref{fig:k-star colouring vertex replacement}. 
For each vertex \( v \) of \( G \), the vertex gadget for \( v \) has \( 2(k-1) \) terminals \( v_1,v_2,\dots,v_{2(k-1)} \) which accommodate the edges incident on \( v \) in \( G \) (each terminal takes at most one edge; order does not matter). 
Replacement of vertices by vertex gadgets converts each edge \( uv \) of \( G \) into an edge \( u_iv_j \) between terminals \( u_i \) and \( v_j \) of corresponding vertex gadgets (where \( i,j\in\{1,2,\dots,2(k-1)\} \)). 
\begin{figure}[hbt]
\centering
\begin{tikzpicture}
\path (0,0)  node[bigdot](v)[label=\( v \)]{} --++(2,-0.5) coordinate(from) --++(0.6,0) coordinate(to);
\draw[dashed,very thin] (v)--+(-0.75,-1) coordinate(w1)
                        (v)--+(-0.5,-1) coordinate(w2)
                        (v)+(0.20,-0.75) node{\dots}
                        (v)--+(1,-1) coordinate(w2k-2);
\draw [decorate,decoration={brace,amplitude=6pt,raise=5pt,mirror}] (w1)+(-0.1,0)--node[below=9pt]{\( 2(k-1) \) potential edges} ($(w2k-2)+(0.1,0)$);

\draw [-stealth,draw=black,line width=3pt] (from)--(to);

\path (to) --+(2.5,0.5) coordinate(centre1);
\path (centre1)+(-90:1.5) node(a)[bigdot]{} node(v1)[terminal][label=right:\( v_1 \)]{}
      (centre1)+( 30:1.5) node(b)[bigdot]{}
      (centre1)+(150:1.5) node(c)[bigdot]{};
\draw[densely dotted,thick] (a)--(b)--(c)--(a);
\draw[densely dash dot,thick]
(centre1)--+(-90:0.5)
(centre1)--+( 30:0.5)
(centre1)--+(150:0.5);

\path (centre1) --+(2.60,0) coordinate(centre2);
\path (centre2)+(-90:1.5) node(a)[bigdot]{} node(v2)[terminal][label=right:\( v_2 \)]{}
      (centre2)+( 30:1.5) node(b)[bigdot]{}
      (centre2)+(150:1.5) node(c)[bigdot]{};
\draw[densely dotted,thick] (a)--(b)--(c)--(a);
\draw[densely dash dot,thick]
(centre2)--+(-90:0.5)
(centre2)--+( 30:0.5)
(centre2)--+(150:0.5);

\path (centre2) --+(4.30,0) coordinate(centre3);
\path (centre3)+(-90:1.5) node(a)[bigdot]{} node(v2k-2)[terminal][label=right:\( v_{2(k-1)} \)]{}
      (centre3)+( 30:1.5) node(b)[bigdot]{}
      (centre3)+(150:1.5) node(c)[bigdot]{};
\draw[densely dotted,thick] (a)--(b)--(c)--(a);
\draw[densely dash dot,thick]
(centre3)--+(-90:0.5)
(centre3)--+( 30:0.5)
(centre3)--+(150:0.5);

\path (centre2)-- node[font=\huge,yshift=-5mm]{\dots} (centre3);

\draw[dashed,very thin] (v1)--+(-0.75,-1)
                        (v2)--+(-0.5,-1)
                        (v2k-2)--+(1,-1);

\end{tikzpicture}
\caption{Replacement of vertex by vertex gadget.}
\label{fig:k-star colouring vertex replacement}
\end{figure}
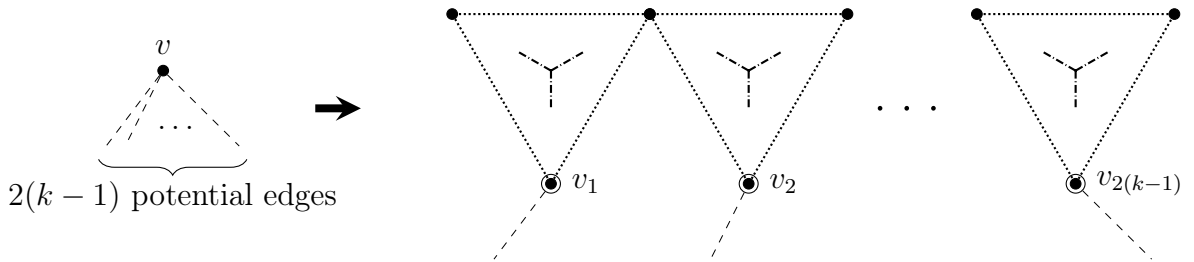
\begin{figure}[hbt]
\centering
\begin{tikzpicture}
\draw (0,0) node[bigdot]{} node[terminal](ui)[label=below:\( u_i \)]{} --++(1,0) node[bigdot]{} node[terminal](vj)[label=below:\( v_j \)]{};
\path (vj) --++(1,0) coordinate(from) --++(0.6,0) coordinate(to);
\draw [-stealth,draw=black,line width=3pt] (from)--(to);

\path (to) --++(1.5,0) node(ui2)[bigdot]{} node[terminal][label={left:\( u_i \)}]{} --++(4.25,0) coordinate(centre1)--++(3.5,0) node(vj2)[bigdot]{} node[terminal][label={right:\( v_j \)}]{};

\path (centre1)+(0:1.5) node(a)[dot]{} node(v1)[label=right:\( a \)]{}
      (centre1)+(120:1.5) node(b)[dot][label=above:\( b \)]{}
      (centre1)+(240:1.5) node(c)[dot][label=below:\( c \)]{};
\draw[densely dotted,thick] (a)--(b)--(c)--(a);
\draw[densely dash dot,thick]
(centre1)--+(0:0.5)
(centre1)--+(120:0.5)
(centre1)--+(240:0.5);

\draw (ui2)--(b)--(vj2)
      (ui2)--(c)--(vj2);
\end{tikzpicture}
\caption{Replacement of edge between terminals by edge gadget.}
\label{fig:k-star colouring edge replacement}
\end{figure}
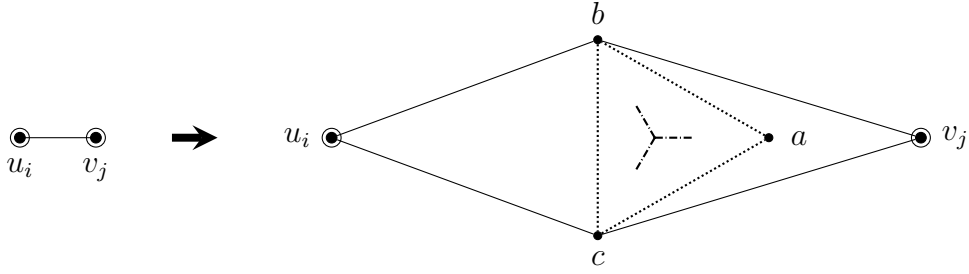

Finally, replace each edge \( u_iv_j \) between terminals by an edge gadget as shown in Figure~\ref{fig:k-star colouring edge replacement}. 
An example is available in the supplementary material.

Observe that every vertex of \( G' \) is within a gadget component. 
We know that the gadget component has maximum degree \( k \) and the `corners' \( a,b,c \) of the gadget component has only two neighbours within the gadget component. 
Since the vertex gadget is a `chain' of gadget components and \( k\geq 4 \), the vertex gadget has maximum degree \( k \) and each terminal of the gadget has only two neighbours within the gadget. 
Also, the edge gadget has maximum degree \( k \) and each terminal of the gadget has only two neighbours within the gadget. 
So, each terminal in \( G' \) has degree two or four. 
The graph \( G' \) has maximum degree \( k \) because gadgets in \( G' \) have maximum degree \( k \), terminals in \( G' \) have degree at most four, and \( k\geq 4 \). 
\end{construct}


\begin{proof}[Proof of Guarantee 1]
Suppose that \( G \) admits a \( k \)-colouring \( f\colon V(G)\to\{0,1,\dots,k-1\} \). We construct a \( k \)-colouring \( f' \) of \( G' \) as follows. For each gadget component within the vertex gadget for \( v\in V(G) \), use the scheme in Figure~\ref{fig:k-star colouring of gadget component}, but swap colours~0 and \( f(v) \). For each gadget component between terminals \( u_i \) and \( v_j \), choose two distinct colours \( p,q\in\{0,1,\dots,k-1\}\setminus\{f(u),f(v)\} \) and apply the scheme in Figure~\ref{fig:k-star colouring of gadget component}, but swap colour~0 with colour~\( p \) and swap colour~\( k-1 \) with colour~\( q \). Clearly, \( f' \) is a \( k \)-colouring. 

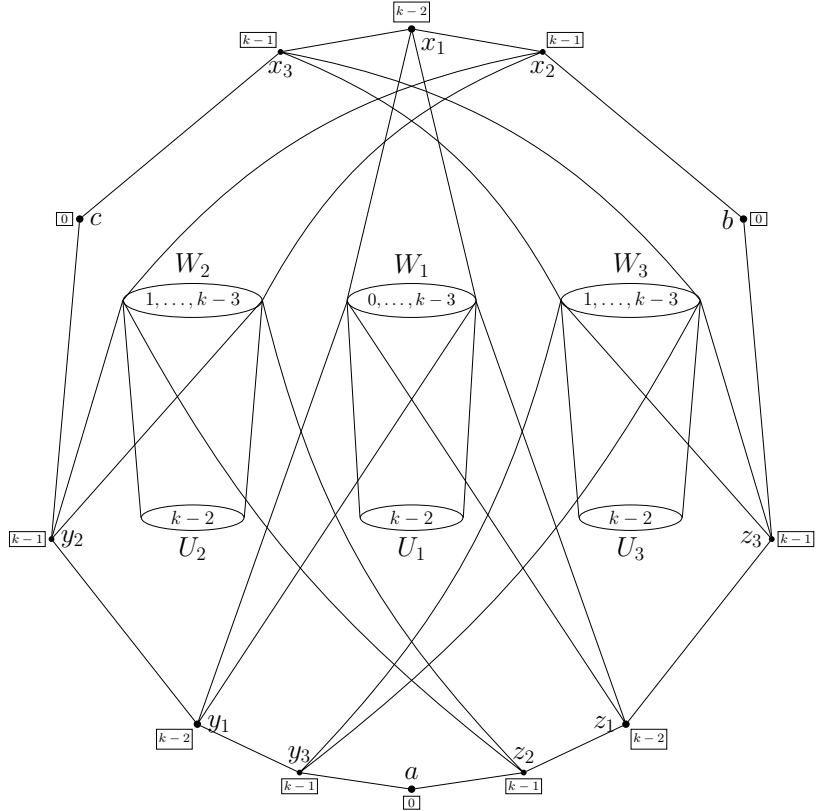
\begin{figure}[hbt]
\centering
\scalebox{0.60}{%
\tikzset{ 
dotplus/.style={bigdot,minimum size=4pt},
bigdotplus/.style={bigdot,minimum size=5pt}
}
\begin{tikzpicture}[scale=1.2]
\tikzstyle every label=[font=\Large]
\path (-90:7) node(a)[bigdot][label=above:\( a \)][label={[vcolour]below:\,\( 0 \)\,}]{}
      ( 30:7) node(b)[bigdot][label=left:\( b \)][label={[vcolour]right:\,\( 0 \)\,}]{}
      (150:7) node(c)[bigdot][label=right:\( c \)][label={[vcolour]left:\,\( 0 \)\,}]{};
\path (90:2) node(W1)[ellipse,inner sep=1pt,draw,font=\small,minimum height=0.75cm][label=\( W_1 \)]{\( 0,\dots,k-3 \)}
      (-90:2) node(U1)[ellipse,inner sep=1pt,draw,minimum width=2.25cm][label=below:\( U_1 \)]{\( k-2 \)};
\path (W1)+(-4,0) node(W2)[ellipse,inner sep=1pt,draw,minimum height=0.75cm][label=\( W_2 \)]{\( 1,\dots,k-3 \)}
      (U1)+(-4,0) node(U2)[ellipse,inner sep=1pt,draw,minimum width=2.25cm][label=below:\( U_2 \)]{\( k-2 \)};
\path (W1)+(4,0) node(W3)[ellipse,inner sep=1pt,draw,minimum height=0.75cm][label=\( W_3 \)]{\( 1,\dots,k-3 \)}
      (U1)+(4,0) node(U3)[ellipse,inner sep=1pt,draw,minimum width=2.25cm][label=below:\( U_3 \)]{\( k-2 \)};
\path ( 90:7) node(x)[bigdot][label=below right:\( x_1 \)][label={[vcolour]above:\( k-2\strut \)}]{}
      (-124:7) node(y)[bigdot][label=right:\( y_1 \)][label={[vcolour]below left:\( k-2\strut \)}]{}
      (-56:7) node(z)[bigdot][label=left:\( z_1 \)][label={[vcolour]below right:\( k-2\strut \)}]{};
\draw (x)--(W1.180)  (x)--(W1.0)
      (y)--(W1.181)  (y)--(W1.-1)
      (z)--(W1.181)  (z)--(W1.-1);
\path (-107:7) node(y2)[dot][label=above:\( y_3 \)][label={[vcolour]below:\( k-1 \)}]{}
      (- 73:7) node(z1)[dot][label=above:\( z_2 \)][label={[vcolour]below:\( k-1 \)}]{};
\draw (y2) to[bend right=12] (W3.0)
      (y2) to[bend right=15] (W3.180);
\draw (z1) to[bend left=12] (W2.180)
      (z1) to[bend left=15] (W2.0);
\path (-160:7) node(y1)[dot][label=right:\( y_2 \)][label={[vcolour]left:\( k-1 \)}]{}
      ( -20:7) node(z2)[dot][label=left:\( z_3 \)][label={[vcolour]right:\( k-1 \)}]{};
\draw (y1)--(W2.180)  (y1)--(W2.0);
\draw (z2)--(W3.180)  (z2)--(W3.0);
\path (110:7) node(x2)[dot][label=below:\( x_3 \)][label={[vcolour]above left:\( k-1 \)}]{}
      ( 70:7) node(x1)[dot][label=below:\( x_2 \)][label={[vcolour]above right:\( k-1 \)}]{};
\draw (x2) to[bend left=20] (W3.180)
      (x2) to[bend left=20] (W3.0);
\draw (x1) to[bend right=20] (W2.180)
      (x1) to[bend right=20] (W2.0);
\draw (W1.0)--(U1.0)  (W1.180)--(U1.180)
      (W2.0)--(U2.0)  (W2.180)--(U2.180)
      (W3.0)--(U3.0)  (W3.180)--(U3.180);
\draw (x)--(x1)--(b)--(z2)--(z)--(z1)--(a)--(y2)--(y)--(y1)--(c)--(x2)--(x);
\end{tikzpicture}
}
\caption{A \( k \)-star colouring scheme for the gadget component (for every vertex in \( U_1\cup U_2\cup U_3 \), use colour \( k-2 \)).}
\label{fig:k-star colouring of gadget component}
\end{figure}

~\\
\noindent \textbf{Claim:} \( f' \) is a \( k \)-star colouring.\\

\noindent Contrary to the claim, assume that there is a 4-vertex path \( Q \) in \( G' \) bicoloured by \( f' \). By construction, the restriction of \( f' \) to each gadget component is a \( k \)-star colouring. It is easy to verify that \( f' \) restricted to each edge gadget is a \( k \)-star colouring. Therefore, either \( Q \) is within a vertex gadget, or \( Q \) contains vertices from a vertex gadget and an edge gadget. We show that the latter leads to a contradiction (the former can be ruled out by similar arguments). Suppose \( Q \) contains vertices from the vertex gadget for vertex \( v\in V(G) \) and from an edge gadget between two terminals \( u_i \) and \( v_j \) of \( G' \). Clearly, \( Q \) must contain a bicoloured 3-vertex path \( Q^* \) such that (i)~\( Q^* \) is entirely within a gadget, and (ii)~an endpoint of \( Q^* \) is a terminal of the gadget. Observe that \( Q^* \) cannot be in the vertex gadget because there is no bicoloured \( P_3 \) in Figure~\ref{fig:k-star colouring of gadget component} having a terminal as an endpoint. Similarly, the choice of colour \( q \) ensures that \( q\notin\{f(u),f(v)\} \), and thus \( Q^* \) cannot be in the edge gadget either. This contradiction proves the claim. Hence, \( G' \) is \( k \)-star colourable.

Conversely, suppose that \( G' \) admits a \( k \)-star colouring \( f':V(G')\to\{0,1,\dots,k-1\} \). 
Thanks to Lemma~\ref{lem: k-star colouring gadget component}, \( f' \) must colour all terminals of a vertex gadget by the same colour (because the vertex gadget is a `chain' of gadget components). 
Again by Lemma~\ref{lem: k-star colouring gadget component}, for each edge gadget between terminals \( u_i \) and \( v_j \) in \( G' \), \( f'(b)=f'(c) \) and hence \( f'(u_i)\neq f'(v_j) \) (if not, \( b,u_i,c,v_j \) is a bicoloured \( P_4 \)). 
So, all terminals of a vertex gadget have the same colour, and terminals of each edge gadget have different colours. 
Therefore, the function \( f\colon V(G)\to\{0,1,\dots,k-1\} \) defined as \( f(v)=f'(v_1) \) for all \( v\in V(G) \), is a \( k \)-colouring of \( G \). 
Hence, \( G \) is \( k \)-colourable whenever \( G' \) is \( k \)-star colourable.
\end{proof}
\begin{proof}[Proof of Guarantee 2]
%
Let us count the number of vertices in \( G' \). 
Each gadget component has \( (k-2)+5(k-3)+12=6k-5 \) vertices. 
Each vertex gadget has \( (2k-2)(6k-5)-(2k-1)=12k^2-24k+11 \) vertices. 
Each edge gadget has \( 6k-5 \) vertices excluding the terminals (let us count terminals as part of vertex gadgets). 
So, \( G' \) has \( n(12k^2-24k+11)+m(6k-5)=O(m+n) \) vertices where \( m=|E(G)| \) and \( n=|V(G)| \). 
As \( \Delta(G)=2(k-1) \), we have \( m\leq n\Delta(G)/2=n(k-1)=O(n) \). 
Therefore, \( G' \) has only \( O(m+n)=O(n) \) vertices.
\end{proof}

Thanks to Construction~\ref{make:k-star}, we have the following theorem since \( k \)-Colourability is NP-complete for graphs of maximum degree \( 2(k-1) \) (in fact, NP-complete for line graphs of \( k \)-regular graphs \cite{leven_galil}) and the problem does not admit a \( 2^{o(n)} \)-time algorithm unless ETH fails (the latter can be observed from a reduction of Emden-Weinert et al.~\cite{emden-weinert}; Theorem~3 in supplementary material presents a shorter alternate proof with the help of an alternate reduction). 
\begin{theorem}\label{thm:k-star max degree k}
For all \( k\geq 4 \), \textsc{\( k \)-Star Colourability}(bipartite, \( \Delta=k \)) is NP-complete, and the problem does not admit a \( 2^{o(n)} \)-time algorithm unless ETH fails.
\qed
\end{theorem}
\section{Open Problems and Related Work}\label{sec:open}

Many problems related to star colouring are open even for the class of cubic graphs. 
Chen et al.~\cite{chen2013} proved that cubic graphs are 6-star colourable. 
On the other hand,  Xie et al.~\cite{xie} proved that cubic graphs are not 3-star colourable. 
So, \( 4\leq \chi_s(G)\leq 6 \) for every cubic graph \( G \). 
Both bounds are tight because \( \chi_s(K_4)=4 \) and \( \chi_s(M_8)=6 \) \cite{fertin2004} where \( M_8 \) is the Wagner's graph (also called the 8-vertex M\"{o}bius ladder graph). 
Conjecture~12 of Almeter et al.~\cite{almeter} implies that \( M_8 \) is the only cubic graph with star chromatic number six. 
\begin{conjecture}[\cite{almeter}]\label{conj:cubi neq M8 5-star colourable}
Every cubic graph except \( M_8 \) is 5-star colourable.
\end{conjecture}
\noindent Conjecture~\ref{conj:cubi neq M8 5-star colourable} implies that \textsc{5-Star Colourability} is in P for cubic graphs.
\begin{problem}
What is the complexity of \textsc{4-Star Colourability} in cubic graphs?
\end{problem}

\section*{Statements and Declarations}

\subsection*{Funding}
The first author is supported by SERB(DST), MATRICS scheme MTR/2018/000086. 

\subsection*{Conflict of interest}
The authors declare that they have no conflict of interest.

\subsection*{Acknowledgement}
We thank  Emil Je\v{r}\'{a}bek and an anonymous referee for their valuable comments. 


\begin{thebibliography}{99}
\providecommand{\url}[1]{\texttt{#1}}
\providecommand{\urlprefix}{URL }

\bibitem{albertson}
{Albertson}, M.O., {Chappell}, G.G., {Kierstead}, H.A., {K\"undgen}, A.,
  {Ramamurthi}, R.: {Coloring with no 2-colored \(P_4\)'s}. {The Electronic
  Journal of Combinatorics}  \textbf{11}(1), ~26 (2004). \doi{10.37236/1779}

\bibitem{almeter}
Almeter, J., Demircan, S., Kallmeyer, A., Milans, K.G., Winslow, R.: Graph
  2-rankings. Graphs and Combinatorics  \textbf{35}(1),  91--102 (2019).
  \doi{10.1007/s00373-018-1979-4}

\bibitem{belmonte}
{Belmonte}, R., {Vatshelle}, M.: {Graph classes with structured neighborhoods
  and algorithmic applications}. {Theoretical Computer Science}  \textbf{511},
  54--65 (2013). \doi{10.1016/j.tcs.2013.01.011}

\bibitem{biro}
{Bir\'o}, C., {Bonnet}, E., {Marx}, D., {Miltzow}, T., {Rz\k{a}\.zewski}, P.:
  Fine-grained complexity of coloring unit disks and balls. {Journal of
  Computational Geometry}  \textbf{9}(2),  47--80 (2018).
  \doi{10.20382/jocg.v9i2a4}

\bibitem{bok}
Bok, J., Jedlickov{\'a}, N., Martin, B., Paulusma, D., Smith, S.: Acyclic, star
  and injective colouring: A complexity picture for {H}-free graphs (2020),
  arXiv preprint arXiv:2008.09415

\bibitem{borie}
Borie, R.B., Parker, R.G., Tovey, C.A.: Automatic generation of linear-time
  algorithms from predicate calculus descriptions of problems on recursively
  constructed graph families. Algorithmica  \textbf{7}(1-6),  555--581 (1992).
  \doi{10.1007/BF01758777}

\bibitem{bubley}
{Bubley}, R., {Dyer}, M., {Greenhill}, C., {Jerrum}, M.: On approximately
  counting colorings of small degree graphs. {SIAM Journal on Computing}
  \textbf{29}(2),  387--400 (1998). \doi{10.1137/S0097539798338175}

\bibitem{bui-xuan2010}
{Bui-Xuan}, B.M., {Telle}, J.A., {Vatshelle}, M.: {\(H\)-join decomposable
  graphs and algorithms with runtime single exponential in rankwidth}.
  {Discrete Applied Mathematics}  \textbf{158}(7),  809--819 (2010).
  \doi{10.1016/j.dam.2009.09.009}

\bibitem{bui-xuan2013}
{Bui-Xuan}, B.M., {Telle}, J.A., {Vatshelle}, M.: {Fast dynamic programming for
  locally checkable vertex subset and vertex partitioning problems}.
  {Theoretical Computer Science}  \textbf{511},  66--76 (2013).
  \doi{10.1016/j.tcs.2013.01.009}

\bibitem{chen2013}
{Chen}, M., {Raspaud}, A., {Wang}, W.: {6-star-coloring of subcubic graphs}.
  {Journal of Graph Theory}  \textbf{72}(1-2),  128--145 (2013).
  \doi{https://doi.org/10.1002/jgt.21636}

\bibitem{coleman_more}
Coleman, T.F., Mor{\'e}, J.J.: Estimation of sparse {Jacobian} matrices and
  graph coloring problems. SIAM Journal on Numerical Analysis  \textbf{20}(1),
  187--209 (1983)

\bibitem{courcelle}
Courcelle, B.: The monadic second-order logic of graphs. {I}. recognizable sets
  of finite graphs. Information and Computation  \textbf{85}(1),  12--75
  (1990). \doi{10.1016/0890-5401(90)90043-H}

\bibitem{cygan2016}
{Cygan}, M., {Fomin}, F.V., {Golovnev}, A., {Kulikov}, A.S., {Mihajlin}, I.,
  {Pachocki}, J., {Soca{\l}a}, A.: Tight bounds for graph homomorphism and
  subgraph isomorphism. In: {Proceedings of the 27th annual ACM-SIAM symposium
  on discrete algorithms, SODA 2016}, pp. 1643--1649. Philadelphia, PA: SIAM;
  New York, NY: ACM (2016). \doi{10.1137/1.9781611974331.ch112}

\bibitem{dailey}
{Dailey}, D.P.: {Uniqueness of colorability and colorability of planar
  4-regular graphs are NP-complete}. {Discrete Mathematics}  \textbf{30},
  289--293 (1980). \doi{10.1016/0012-365X(80)90236-8}

\bibitem{dunbar}
{Dunbar}, J.E., {Hedetniemi}, S.M., {Hedetniemi}, S.T., {Jacobs}, D.P.,
  {Knisely}, J., {Laskar}, R.C., {Rall}, D.F.: {Fall colorings of graphs}.
  {JCMCC. The Journal of Combinatorial Mathematics and Combinatorial Computing}
   \textbf{33},  257--273 (2000)

\bibitem{dvorak_esperet}
{Dvo\v{r}\'ak}, Z., {Esperet}, L.: Distance-two coloring of sparse graphs.
  {European Journal of Combinatorics}  \textbf{36},  406--415 (2014).
  \doi{10.1016/j.ejc.2013.09.002}

\bibitem{emden-weinert}
{Emden-Weinert}, T., {Hougardy}, S., {Kreuter}, B.: {Uniquely colourable graphs
  and hardness of colouring graphs of large girth}. {Combinatorics, Probability
  and Computing}  \textbf{7}(4),  375--386 (1998).
  \doi{10.1017/S0963548398003678}

\bibitem{fertin2003}
{Fertin}, G., {Godard}, E., {Raspaud}, A.: {Acyclic and \(k\)-distance coloring
  of the grid}. {Information Processing Letters}  \textbf{87}(1),  51--58
  (2003). \doi{10.1016/S0020-0190(03)00232-1}

\bibitem{fertin2004}
{Fertin}, G., {Raspaud}, A., {Reed}, B.: Star coloring of graphs. {Journal of
  Graph Theory}  \textbf{47}(3),  163--182 (2004). \doi{10.1002/jgt.20029}

\bibitem{fiala_kratochvil}
{Fiala}, J., {Kratochv\'{\i}l}, J.: {Locally constrained graph homomorphisms --
  structure, complexity, and applications}. {Computer Science Review}
  \textbf{2}(2),  97--111 (2008). \doi{10.1016/j.cosrev.2008.06.001}

\bibitem{fomin2019}
{Fomin}, F.V., {Lokshtanov}, D., {Saurabh}, S., {Zehavi}, M.: Kernelization.
  {Theory} of parameterized preprocessing. Cambridge: Cambridge University
  Press (2019). \doi{10.1017/9781107415157}

\bibitem{garey_johnson_stockmeyer1976}
{Garey}, M.R., {Johnson}, D.S., {Stockmeyer}, L.: {Some simplified NP-complete
  graph problems}. {Theoretical Computer Science}  \textbf{1},  237--267
  (1976). \doi{10.1016/0304-3975(76)90059-1}

\bibitem{gebremedhin2005}
Gebremedhin, A.H., Manne, F., Pothen, A.: What color is your {Jacobian}?
  {Graph} coloring for computing derivatives. SIAM Review  \textbf{47}(4),
  629--705 (2005). \doi{10.1137/S0036144504444711}

\bibitem{gerber_kobler}
{Gerber}, M.U., {Kobler}, D.: {Algorithms for vertex-partitioning problems on
  graphs with fixed clique-width}. {Theoretical Computer Science}
  \textbf{299}(1-3),  719--734 (2003). \doi{10.1016/S0304-3975(02)00725-9}

\bibitem{grunbaum}
{Gr\"unbaum}, B.: {Acyclic colorings of planar graphs}. {Israel Journal of
  Mathematics}  \textbf{14},  390--408 (1973).
  \doi{https://doi.org/10.1007/BF02764716}

\bibitem{harshita}
{Harshita}, K., {Mishra}, S., {Sadagopan}, N., {Renjith}, P.: {FO} and {MSO}
  approach to some graph problems: approximation and poly time results (2017),
  arXiv preprint arXiv:1711.02889

\bibitem{lei}
{Lei}, H., {Shi}, Y., {Song}, Z.X.: Star chromatic index of subcubic
  multigraphs. {Journal of Graph Theory}  \textbf{88}(4),  566--576 (2018).
  \doi{10.1002/jgt.22230}

\bibitem{leven_galil}
{Leven}, D., {Galil}, Z.: {NP-completeness of finding the chromatic index of
  regular graphs}. {J. Algorithms}  \textbf{4},  35--44 (1983).
  \doi{10.1016/0196-6774(83)90032-9}

\bibitem{levit}
{Levit}, M., {Chandran}, L.S., {Cheriyan}, J.: {On Eulerian orientations of
  even-degree hypercubes}. {Oper. Res. Lett.}  \textbf{46}(5),  553--556
  (2018). \doi{10.1016/j.orl.2018.09.002}

\bibitem{meyer}
{Meyer}, W.: {Equitable coloring}. {Am. Math. Mon.}  \textbf{80},  920--922
  (1973). \doi{10.2307/2319405}

\bibitem{nesetril_mendez2003}
Ne{\v{s}}et{\v{r}}il, J., de~Mendez, P.O.: Colorings and homomorphisms of minor
  closed classes. In: Discrete and Computational Geometry, pp. 651--664.
  Springer (2003)

\bibitem{nishizeki_rahman}
{Nishizeki}, T., {Rahman}, M.S.: {Planar graph drawing}, vol.~12. River Edge,
  NJ: World Scientific (2004)

\bibitem{oksimets}
{Oksimets}, N.: {A characterization of Eulerian graphs with triangle free Euler
  tours}. Tech. rep., Department of Mathematics, Ume{\aa} University (2003)

\bibitem{oum}
{Oum, S.-i.}, {S{\ae}ther}, S.H., {Vatshelle}, M.: {Faster algorithms for
  vertex partitioning problems parameterized by clique-width}. {Theoretical
  Computer Science}  \textbf{535},  16--24 (2014).
  \doi{10.1016/j.tcs.2014.03.024}

\bibitem{stacho}
{Stacho}, J.: 3-colouring {AT-free} graphs in polynomial time. {Algorithmica}
  \textbf{64}(3),  384--399 (2012). \doi{10.1007/s00453-012-9624-8}

\bibitem{telle}
Telle, J.A.: {Vertex partitioning problems: characterization, complexity and
  algorithms on partial \(k\)-trees}. Ph.D. thesis, University of Oregon (1994)

\bibitem{telle_proskurowski}
{Telle}, J.A., {Proskurowski}, A.: {Algorithms for vertex partitioning problems
  on partial \(k\)-trees}. {SIAM Journal on Discrete Mathematics}
  \textbf{10}(4),  529--550 (1997). \doi{10.1137/S0895480194275825}

\bibitem{west}
West, D.B.: Introduction to graph theory. Prentice Hall, Upper Saddle River, 2
  edn. (2001)

\bibitem{xie}
{Xie}, D., {Xiao}, H., {Zhao}, Z.: {Star coloring of cubic graphs}.
  {Information Processing Letters}  \textbf{114}(12),  689--691 (2014).
  \doi{10.1016/j.ipl.2014.05.013}

\end{thebibliography}
\end{document}